\documentclass[a4paper,10pt,leqno]{amsart}
\usepackage[margin=3cm]{geometry}
\usepackage{tikz-cd}
\usepackage{amssymb}
\usepackage{amsxtra}
\usepackage{amsthm}
\usepackage[T1]{fontenc}
\usepackage{lmodern}
\usepackage[pagebackref,breaklinks]{hyperref}
\usepackage[abbrev,alphabetic]{amsrefs}
\usepackage[shortlabels]{enumitem}
\usepackage{colonequals}
\usepackage{mathtools}
\theoremstyle{definition}
\newtheorem{defn}{Definition}[section]
\newtheorem{example}[defn]{Example}
\newtheorem{remark}[defn]{Remark}
\newtheorem{notation}[defn]{Notation}
\theoremstyle{plain}
\newtheorem{lemma}[defn]{Lemma}
\newtheorem{prop}[defn]{Proposition}
\newtheorem{cor}[defn]{Corollary}
\newtheorem{theorem}[defn]{Theorem}
\newtheorem{conj}[defn]{Conjecture}
\numberwithin{equation}{section}
\setlist[enumerate,1]{(a)}

\newcommand{\N}{\mathbb{N}}
\newcommand{\Z}{\mathbb{Z}}
\newcommand{\Q}{\mathbb{Q}}
\newcommand{\R}{\mathbb{R}}

\newcommand{\C}{\mathbb{C}}

\newcommand{\PP}{\mathbb{P}}

\newcommand{\be}{\mathbf{e}}
\newcommand{\bv}{\mathbf{v}}
\newcommand{\bw}{\mathbf{w}}
\newcommand{\bx}{\mathbf{x}}
\newcommand{\by}{\mathbf{y}}
\newcommand{\ux}{\underline{x}}
\newcommand{\uy}{\underline{y}}
\newcommand{\uz}{\underline{z}}

\newcommand{\cD}{\mathcal{D}}

\newcommand{\cI}{\mathcal{I}}

\newcommand{\cL}{\mathcal{L}}

\newcommand{\cO}{\mathcal{O}}
\newcommand{\cP}{\mathcal{P}}

\newcommand{\cS}{\mathcal{S}}

\newcommand{\rV}{\mathrm{V}}
\newcommand{\rL}{\mathrm{L}}
\newcommand{\fS}{\mathfrak{S}}

\newcommand{\diff}{\mathrm{diff}}

\newcommand{\Herm}{\mathrm{Herm}}

\DeclareMathOperator{\re}{Re}

\newcommand{\rk}{\mathrm{rk}}

\newcommand{\Amp}{\mathrm{Amp}}
\newcommand{\Map}{\mathrm{Map}}
\newcommand{\Nef}{\mathrm{Nef}}

\newcommand{\Gr}{\mathrm{Gr}}
\newcommand{\spa}{\mathrm{span}}
\newcommand{\Hom}{\mathbf{Hom}}
\newcommand{\Comp}{\mathrm{Comp}}

\newcommand{\id}{\mathrm{id}}
\newcommand{\vol}{\mathrm{vol}}

\newcommand{\MD}{\mathrm{MD}}

\newcommand{\HR}{\mathrm{pHR}}
\newcommand{\HRw}{\mathrm{pHR}^{\mathrm{w}}}
\newcommand{\fHR}{\mathrm{HR}}
\newcommand{\fHRw}{\mathrm{HR}^{\mathrm{w}}}
\newcommand{\fHRvw}{\mathrm{HR}^{\mathrm{vw}}}

\newcommand{\sL}{\mathrm{sL}}

\newcommand{\Fl}{\mathrm{Fl}}

\newcommand{\NS}{N^1}

\newcommand{\Sesq}{\mathrm{Sesq}}

\newcommand{\prim}{\mathrm{prim}}
\newcommand{\hprim}{\text{-}\prim}

\newcommand{\bone}{\mathbf{1}}
\newcommand{\simto}{\xrightarrow{\sim}}

\newcommand{\cHR}[1]{\prescript{}{#1}{\mathcal{HR}}}
\newcommand{\cHRp}[1]{\prescript{\equiv}{#1}{\mathcal{HR}}}
\newcommand{\cHRpp}[1]{\prescript{(\equiv)}{#1}{\mathcal{HR}}}
\newcommand{\cHRw}[2]{\prescript{}{#1}{\mathcal{HR}}^{\mathrm{w},#2}}
\newcommand{\cHRww}[2]{\prescript{}{#1}{\mathcal{HR}}^{(\mathrm{w}),#2}}
\newcommand{\cHRpw}[2]{\prescript{\equiv}{#1}{\mathcal{HR}}^{\mathrm{w},#2}}
\newcommand{\cHRpww}[2]{\prescript{\equiv}{#1}{\mathcal{HR}}^{(\mathrm{w}),#2}}
\newcommand{\cHRppw}[2]{\prescript{(\equiv)}{#1}{\mathcal{HR}}^{\mathrm{w},#2}}
\newcommand{\cHRZ}[1]{\prescript{\Z}{#1}{\mathcal{HR}}}

\DeclareMathOperator{\Rea}{\mathrm{Re}}

\begin{document}
\title{Hodge--Riemann polynomials}

\author{Qing Lu}
\address{School of Mathematical Sciences, Beijing Normal University, Beijing
100875, China} \email{qlu@bnu.edu.cn} 

\author{Weizhe Zheng}
\address{Morningside Center of Mathematics, Academy of Mathematics and Systems Science, Chinese Academy of Sciences, Beijing 100190, 
China; University of Chinese Academy of Sciences, Beijing 100049, China}
\email{wzheng@math.ac.cn}

\date{}

\begin{abstract}
We show that Schur classes of ample vector bundles on smooth projective 
varieties satisfy Hodge--Riemann relations on $H^{p,q}$ under the 
assumption that $H^{p-2,q-2}$ vanishes. More generally, we study 
Hodge--Riemann polynomials, which are partially symmetric polynomials that 
produce cohomology classes satisfying the Hodge--Riemann property when 
evaluated at Chern roots of ample vector bundles. In the case of line 
bundles and in bidegree $(1,1)$, these are precisely the nonzero dually 
Lorentzian polynomials. We prove various properties of Hodge--Riemann 
polynomials, confirming predictions and answering questions of Ross and 
Toma. As an application, we deduce Schur log-concavity of derivative 
sequences of products of Schur polynomials, confirming conjectures of Ross 
and Wu. We also prove monomial log-concavity of derivative sequences of 
dually Lorentzian polynomials. In an appendix, Shouda Wang uses this result 
to give a new proof of the Dowling--Zhao conjecture. 
\end{abstract}

\maketitle 

\section{Introduction}

In Hodge theory, the hard Lefschetz theorem asserts that for any ample line 
bundle $L$ on a smooth projective variety $X$ of dimension $d$ over $\C$ and 
$p+q+n=d$, the map 
\[-\wedge c_1(L)^n\colon H^{p,q}(X)\to H^{d-q,d-p}(X)\]
is a bijection and the Hodge--Riemann bilinear relations assert that the
Hermitian pairing
\[(\alpha,\beta)\mapsto i^{q-p}(-1)^{\frac{(p+q)(p+q+1)}{2}}\int_X \alpha\wedge \bar\beta\wedge c_1(L)^n\]
is positive definite on the kernel of
\[-\wedge c_1(L)^{n+1}\colon
H^{p,q}(X)\to H^{d-q+1,d-p+1}(X).
\]
These theorems have been generalized by many authors.

The goal of this paper is to further extend these theorems to certain 
characteristic classes of ample vector bundles $E$ and combinations of such 
classes, under the assumption $H^{p-2,q-2}(X)=0$. In particular, we prove the 
following theorem for Schur classes $s_\lambda(E)$. 

\begin{theorem}\label{t.main}
Let $E$ be an ample vector bundle of rank $e$ on a smooth projective variety 
$X$ of dimension $d$ over $\C$. Let $p$ and $q$ be integers such that $p+q\le 
d$ and $H^{p-2,q-2}(X)=0$. Let $\lambda$ be a partition of $d-p-q$ satisfying 
$\lambda_1\le e$. Let $h=c_1(L)$, where $L$ is an ample line bundle on $X$. 
\begin{enumerate}
\item (Hard Lefschetz) The map $-\wedge s_\lambda(E)\colon H^{p,q}(X)\to 
    H^{d-q,d-p}(X)$ is a bijection. 
\item (Lefschetz decomposition) $H^{p,q}(X)= H^{p-1,q-1}(X)\wedge h \oplus
    H^{p,q}(X)_{s_\lambda(E) h\hprim}$, where $H^{p,q}(X)_{s_\lambda(E)
    h\hprim}$ denotes the kernel of
    \[-\wedge s_\lambda(E) h\colon
    H^{p,q}(X)\to H^{d-q+1,d-p+1}(X).\]
\item (Hodge--Riemann relations) The Hermitian form 
\[(\alpha,\beta)\mapsto \langle \alpha,\beta\rangle_{s_\lambda(E)} \colonequals i^{q-p}(-1)^{\frac{(p+q)(p+q+1)}{2}}\int_X \alpha\wedge \bar\beta\wedge s_\lambda(E)\]
is positive definite on $H^{p,q}(X)_{s_\lambda(E) h\hprim}$ and negative
definite on $H^{p-1,q-1}(X)\wedge h$.
\end{enumerate}
\end{theorem}

The assumption $\lambda_1\le e$ is necessary since otherwise we would have 
$s_\lambda(E)=0$ by our convention \eqref{eq.Schurconv} on Schur classes. The 
assumption $H^{p-2,q-2}(X)=0$ is notably satisfied if $\min(p,q)\le 1$ but, 
as pointed out by the referee, never satisfied if $p=q\ge 2$. 

\begin{remark}\label{r.intro}\leavevmode
\begin{enumerate}
\item The case $(p,q)=(0,0)$ is a theorem of Fulton and Lazarsfeld
    \cite[Theorem~I]{FL}, which extends \cite[Corollary 3.7]{UT}. The study
    of numerical positivity of polynomials of Chern classes was initiated
    by Griffiths \cite{Griffiths}.
\item The case $(p,q)=(1,1)$ is a theorem of Ross and Toma \cite[Theorem 
    1.1]{RT}. 
\item In the case of the top Chern class $c_e(E)$, the assumption 
    $H^{p-2,q-2}(X)=0$ can be dropped. See Theorem \ref{t.top}. The Hard 
    Lefschetz theorem in this case is due to Bloch and Gieseker 
    \cite[Proposition 1.3]{BG}. 
\item The assumption $H^{p-2,q-2}(X)=0$ is crucial for general Schur 
    classes. See \cite[Example 9.2]{RT} for a counterexample to the 
    Hodge--Riemann relations (and the hard Lefschetz property if 
    $\R$-twists are allowed) for non-top Chern classes in the case 
    $(p,q)=(2,2)$. 
\end{enumerate}
\end{remark}

The Hodge--Riemann property on the primitive part can be expressed as a
generalized Alexandrov--Fenchel/Khovanskii-Teissier \cites{Khovanskii,
Teissier} type inequality.

\begin{cor}\label{c.main}
Notation and assumptions as in Theorem \ref{t.main}. Let $\beta_1,\dots,
\beta_n$ be a basis of $H^{p-1,q-1}(X)$ and let $\alpha_i=\beta_i\wedge h$
for $1\le i\le n$.  For all $\alpha_0\in H^{p,q}(X)$, we have
\[(-1)^{n}\det(\langle \alpha_i,\alpha_j\rangle_{s_\lambda(E)})_{0\le i,j\le n}\ge
    0\]
and equality holds if and only if $\alpha_0\in H^{p-1,q-1}(X)\wedge h$. In
particular, if $H^{p-1,q-1}(X)=\C \beta$ is one-dimensional, then for all
$\alpha\in H^{p,q}(X)$, we have
\[\lvert\langle \alpha,\beta\wedge h\rangle_{s_\lambda(E)}\rvert^2\ge \langle \alpha,\alpha\rangle_{s_\lambda(E)}\langle \beta\wedge h,\beta\wedge h\rangle_{s_\lambda(E)}\]
and equality holds if and only if $\alpha\in H^{p-1,q-1}(X)\wedge h$.
\end{cor}

We prove several generalizations of Theorem \ref{t.main}, such as for 
Schubert classes of filtered vector bundles and for products of derivatives 
of Schur and Schubert classes (Corollaries \ref{c.final} and 
\ref{c.Schubertprod}). While the general results require the assumption 
$H^{p-2,q-2}(X)=0$, our result for products of top Chern classes (Theorem 
\ref{t.top}) does not make this assumption. Although the Hodge--Riemann 
property is not stable under positive linear combinations \cite[Section 
9.2]{RT}, we can prove the Hodge--Riemann property in the following special 
case. 

\begin{defn}
A sequence of nonnegative real numbers $a_0,\dots,a_n$ such that all minors 
of the matrix $(a_{j-i})_{0\le i,j< \infty}$ are nonnegative is called a 
\emph{P\'olya frequency sequence}. Here by convention $a_k=0$ for $k\notin 
[0,n]$. 
\end{defn}

Given a polynomial $s(x_1,\dots,x_e)$, we define $s^{[i]}(x_1,\dots,x_e)$ by
the Taylor expansion
\[s(x_1+t,\dots,x_e+t)=\sum_{i=0}^{\deg(s)} s^{[i]}(x_1,\dots,x_e)t^i.\]

\begin{theorem}\label{t.Polya}
Notation and assumptions as in Theorem \ref{t.main}. (In particular, we 
assume $H^{p-2,q-2}(X)=0$.) Let $a_0,\dots,a_{n}$ be a P\'olya frequency 
sequence such that $a_j>0$ for some $j\le \lvert \lambda\rvert$. Then 
\begin{equation}\label{e.tPolya}
\sum_{i=0}^{n} a_i h^i s_{\lambda}^{[i]}(E)
\end{equation}
satisfies the Hodge--Riemann relations on $H^{p,q}(X)$.
\end{theorem}

The case $(p,q)=(1,1)$ answers a question of Ross and Toma \cite[Question 
9.6]{RT2}, who proved a weaker Hodge--Riemann property for \eqref{e.tPolya} 
\cite[Theorem 9.3]{RT2}. 

More generally, we study partially symmetric polynomials which produce 
cohomology classes satisfying the Hodge--Riemann property when evaluated at 
Chern roots of ample vector bundles. We let $\cS_{e_1,\dots,e_r}^k \subseteq 
\R[x_{1,1},\dots,x_{1,e_1};\dots;x_{r,1},\dots,x_{r,e_r}]=\R[\underline{x}]$ 
denote the space of homogeneous polynomials of degree $k$ that are invariant 
under the action of $\Sigma_{e_1}\times \dots\times \Sigma_{e_r}$, where 
$\Sigma_{e_i}$ permutes $x_{i,1},\dots,x_{i,e_i}$. An $\R$-linear basis is 
given by $s_{\lambda^1}(x_{1,1},\dots, x_{1,e_1})\dotsm 
s_{\lambda^r}(x_{r,1},\dots, x_{r,e_r})$, where $\lambda^{1},\dots,\lambda^r$ 
are partitions satisfying $(\lambda^i)_1\le e_i$ for all $1\le i\le r$ and 
$\lvert \lambda^1\rvert+\dots+\lvert \lambda^r\rvert=k$. For simplicity, we 
will restrict to the case $\min(p,q)= 1$ in the rest of the introduction. 

\begin{defn}\label{d.intro}
Assume $\min(p,q)= 1$. A polynomial $g\in \cS_{e_1,\dots,e_r}^k$ is called a 
\emph{Hodge--Riemann polynomial} in bidegree $(p,q)$ if $g(E_1,\dots,E_r)$ 
satisfies Hodge--Riemann relations on $H^{p,q}(X)$ for every smooth 
projective variety $X$ of dimension $p+q+k$ and all ample $\R$-twisted vector 
bundles $E_1,\dots,E_r$ on $X$ of ranks $e_1,\dots,e_r$, respectively. Here 
$g(E_1,\dots,E_r)\in H^{k,k}(X,\R)$ denotes the value of $g$ at the Chern 
roots of $E_1,\dots,E_r$. We let $\cHR{p,q}^k_{e_1,\dots,e_r}\subseteq 
\cS_{e_1,\dots,e_r}^k$ denote the set of Hodge--Riemann polynomials in 
bidegree $(p,q)$. 
\end{defn}

Polynomials related to the Hodge--Riemann property have been intensively 
studied in recent years. Br\"and\'en and Huh \cite{BH} introduced Lorentzian 
polynomials as a generalization of volume polynomials of nef divisors. An 
equivalent notion was introduced in \cite{AGV}. Ross, S\"u\ss, and Wannerer 
\cite{RSW} introduced dually Lorentzian polynomials and showed that nonzero 
dually Lorentzian polynomials of degree $k$ are Hodge--Riemann in bidegree 
$(1,1)$ for $(e_1,\dots,e_r)=(1,\dots,1)$ \cite[Theorem 1.4]{RSW}. We show 
that the converse also holds, giving geometric characterizations of 
Lorentzian polynomials and dually Lorentzian polynomials. We write $\bone^r$ 
for $(1,\dots,1)$, where $1$ is repeated $r$ times. 

\begin{theorem}\label{t.iLor}
Let $f\in \R[y_1,\dots,y_r]$ be a homogeneous polynomial of degree $k$.
\begin{enumerate}
\item $f$ is dually Lorentzian if and only if $f\in
    \cHR{1,1}^k_{\bone^r}\cup \{0\}$.
\item $f$ is Lorentzian if and only if there exist $l\ge 0$, $g\in 
    \cHR{1,1}^l_{\bone^r}$, $X$ a product of projective spaces satisfying 
    $\dim(X)=k+l$, and nef classes $\xi_1,\dots,\xi_r\in \NS(X)$ such that 
    \begin{equation}\label{eq.iLor}
    f(y_1,\dots,y_r)=\frac{1}{k!}\int_X g(\xi_1,\dots,\xi_r)(y_1\xi_1+\dots+y_r\xi_r)^k.
    \end{equation}
    Moreover, if $g\in \cHR{1,1}^l_{e_1,\dots,e_r}$, $X$ is a smooth 
    projective variety of dimension $k+l$, $E_1,\dots,E_r$ are nef 
    $\R$-twisted vector bundles of ranks $e_1,\dots,e_r$, and 
    $\xi_1,\dots,\xi_s\in\NS(X)_\R$ are nef classes, then 
\[    
p(y_1,\dots,y_s)=\frac{1}{k!}\int_X 
    g(E_1,\dots,E_r)(y_1\xi_1+\dots+y_s\xi_s)^k 
\]
is Lorentzian.
\end{enumerate}
\end{theorem}

Given a smooth projective variety $X$ of dimension $d$ and nef classes 
$\xi_1,\dots,\xi_r\in \NS(X)$, the volume polynomial is defined to be 
\[v(y_1,\dots,y_r)=\frac{1}{d!}\int_X (y_1\xi_1+\dots+y_r\xi_r)^d.\] 
Huh gave an example of a Lorentzian polynomial that is not a volume 
polynomial of nef divisors \cite[Example~14]{Huh}, as a consequence of the 
reverse Khovanskii--Teissier inequality \cite[Theorem 5.7]{LX}. In 
particular, $g$ cannot be dropped in \eqref{eq.iLor}. 

As pointed out by Yiran Lin, Theorem \ref{t.iLor} implies yet another 
characterization of dually Lorentzian polynomials, in terms of generalized 
mixed discriminants of Hermitian matrices (Corollary \ref{c.mdisc}). 

In the definition of $\cHR{1,1}^k_{\bone^r}$, one can in fact restrict to
line bundles without $\R$-twists. It follows that all Hodge--Riemann
polynomials in $\cHR{p,q}^k_{e_1,\dots,e_r}$ are dually Lorentzian.

We prove several preservation properties of the collection of Hodge--Riemann
polynomials. In particular, Theorem \ref{t.Polya}  is deduced from the
preservation of Hodge--Riemann polynomials under certain differential
operators.

\begin{theorem}\label{t.intro8}
Let $v$ be a volume polynomial of $r$ ample $\R$-divisors on a smooth 
projective variety of dimension $n$. Let 
$\partial_v=v(\partial_1,\dots,\partial_r)$, where 
$\partial_i=\frac{\partial}{\partial x_{i,1}}+\dots+\frac{\partial}{\partial 
x_{i,e_i}}$. For $k\ge n$ and $\min(p,q)=1$, we have 
\[\partial_v(\cHR{p,q}^k_{e_1,\dots,e_r})\subseteq \cHR{p,q}^{k-n}_{e_1,\dots,e_r}.\]
\end{theorem}

We show that the characteristic numbers of derivatives of Hodge--Riemann
polynomials can be organized into Lorentzian polynomials. Here is the case of
Schur polynomials.

\begin{theorem}\label{t.intro9}
Let $\lambda^1,\dots,\lambda^r$ be partitions, $m,n_1,\dots,n_r\ge 0$. Let 
$X$ be a smooth projective variety of dimension $m+\sum_{i=1}^r (\lvert 
\lambda^i\rvert -n_i)$ and let $E_1,\dots, E_r$ be nef $\R$-twisted vector 
bundles on $X$ of ranks $e_1,\dots,e_r$, respectively. The polynomial 
\[f(x_1,\dots,x_r)=\sum_{\substack{m_1+\dots+m_r=m\\0\le m_i\le n_i}}\frac{x_1^{m_1}\dotsm x_r^{m_r}}{m_1!\dotsm m_r!}\int_X s_{\lambda^1}^{[n_1-m_1]}(E_1)\dotsm s_{\lambda^r}^{[n_r-m_r]}(E_r)\]
is Lorentzian. If, moreover, $E_1,\dots,E_r$ are ample and, for each $1\le 
i\le r$, we have $m\le n_i\le \lvert \lambda^i\rvert$ and $e_i\ge 
(\lambda^i)_1$, then $f$ is strictly Lorentzian. 
\end{theorem}

This confirms a prediction of Ross and Toma \cite[Question 10.9]{RT2}. 

Recall that a sequence $b_0,\dots,b_n$ of nonnegative real numbers is called 
log-concave if $b_i^2\ge b_{i-1}b_{i+1}$ for all $0<i<n$. It is called 
\emph{strictly log-concave} if $b_i^2>b_{i-1}b_{i+1}$ for all $0<i<n$. 

\begin{cor}\label{c.intro10}
Let $\lambda$ and $\mu$ be partitions. Let $X$ be a smooth projective variety 
of dimension $d$ and let $E$ and $F$ be ample $\R$-twisted vector bundles on 
$X$ of ranks $e\ge \lambda_1$ and $f\ge \mu_1$, respectively. Then 
\[\int_X s^{[i]}_\lambda(E)s^{[\lvert \lambda \rvert+\lvert \mu \rvert -d-i]}_\mu(F),\quad \max(0,\lvert \lambda\rvert-d)\le i\le \min(\lvert \lambda\rvert,\lvert \lambda \rvert+\lvert \mu \rvert-d) \] 
is a strictly log-concave sequence of positive numbers. 
\end{cor}

This confirms a prediction of Ross and Toma \cite[Remark 10.8]{RT2}, who 
proved the non-strict log-concavity \cite[Theorem 10.5]{RT2}. 

Our results have purely combinatorial consequences for the log-concavity of 
derivative sequences. We let $\cP^{k}_{e_1,\dots,e_r}$ denote the set of 
nonnegative linear combinations of 
$s_{\lambda^1}(x_{1,1},\dots,x_{1,e_1})\dotsm 
s_{\lambda^r}(x_{r,1},\dots,x_{r,e_r})$, where $\lambda^1,\dots,\lambda^r$ 
are partitions satisfying $\lvert \lambda^1\rvert+\dots+\lvert 
\lambda^r\rvert=k$. This is a generalization of Schur positivity to several 
groups of variables. 

\begin{theorem}\label{t.ifinal}
Let $f\in\cS^k_{e_1,\dots,e_r}$ such that 
$f(\underline{x})f(\underline{y})\in 
\cHR{1,1}^{2k}_{e_1,\dots,e_r,e_1,\dots,e_r}$. Let $D=\sum_{i=1}^r a_i 
\partial_i$, where $a_i\ge 0$ and $\partial_i=\sum_{j=1}^{e_i}\frac{\partial}{\partial x_{i,j}}$ 
for all $1\le i\le r$. Then, for $1\le m\le n$, we have 
\[\frac{D^mf}{m!} \frac{D^nf}{n!}-\frac{D^{m-1}f}{(m-1)!}\frac{D^{n+1}f}{(n+1)!}\in \cP^{2k-m-n}_{e_1,\dots,e_r}.\]
\end{theorem}

Note that even in the case $r=1$, our assumption is the Hodge--Riemann 
property for a partially symmetric polynomial. 

\begin{cor}\label{c.ifinal}
Let $\lambda^1,\dots,\lambda^r$ be partitions and let $f=s_{\lambda^1}\dotsm
    s_{\lambda^r}\in \Z[x_1,\dots, x_e]$.
Then, for $1\le m\le n$,
\[f^{[m]}f^{[n]}-f^{[m-1]}f^{[n+1]}
\]
is Schur positive.
\end{cor}

In particular, taking $m=n$, this proves conjectures of Ross and Wu 
\cite[Conjectures 1.1, 1.4]{RW}. They proved some special cases of the case 
$r=m=n=1$ by combinatorial methods \cite[Theorem 1.2]{RW}. Previously Ross 
and Toma proved that in the case $r=1$, the sequence of numbers 
$f^{[i]}(x_1,\dots,x_e)$, $i\ge 0$ is log-concave for every 
$(x_1,\dots,x_e)\in \R_{\ge 0}^e$ \cite[Corollary 10.12]{RT2}. 

\begin{cor}\label{c.idLor}
Let $f\in \R[x_1,\dots,x_r]$ be a dually Lorentzian polynomial and let 
$D=\sum_{i=1}^r a_i \frac{\partial}{\partial {x_i}}$, where $a_i\ge 0$ for 
all $1\le i\le r$. Let $1\le m\le n$. Then
\[\frac{D^mf}{m!} \frac{D^nf}{n!}-\frac{D^{m-1}f}{(m-1)!}\frac{D^{n+1}f}{(n+1)!}\]
is monomial-positive. In particular, $f^{[m]}f^{[n]}-f^{[m-1]}f^{[n+1]}$ is 
monomial-positive. 
\end{cor}

This strengthens a result of Ross, S\"u\ss, and Wannerer \cite[Corollary 
8.14]{RSW}, which says that the sequence of numbers $f^{[i]}(y_1,\dots,y_r)$, 
$i\ge 0$ is log-concave for every $(y_1,\dots,y_r)\in \R_{\ge 0}^r$.

Using Corollary \ref{c.idLor}, Shouda Wang gave a new proof of the 
Dowling--Zhao conjecture on generating polynomials of matroids \cite{Zhao}. 
The proof is included as an appendix to this article. This conjecture was 
recently proved in \cite{AMCDEHW} and \cite{CCLW}. 

Our proof of Theorem \ref{t.main} builds on the strategy of Fulton and 
Lazarsfeld \cite{FL} for $H^{0,0}$ and that of Ross and Toma \cite{RT} for 
$H^{1,1}$. While the Hodge--Riemann property implies that the pairing on 
$H^{1,1}$ is Lorentzian, the pairing on $H^{p,q}$ has signature 
$(h^{p,q}-h^{p-1,q-1},h^{p-1,q-1})$. In particular, we need to extend the 
linear algebra machine of Ross and Toma \cite{RT3} for Lorentzian forms to 
Hermitian forms in general. This also puts restrictions on the geometric 
constructions we can use in the proof. At many places, we need to ensure that 
the relevant morphisms of varieties induce isomorphisms on $H^{p-1,q-1}$. 

One crucial point in the proofs of some of our theorems is the preservation 
of forms satisfying the Hodge--Riemann property on $X$ under multiplication 
by ample classes in $\NS(X)_\R$. In particular, we obtain extensions of the 
algebraic case of the mixed Hodge--Riemann relations of Gromov \cite[Theorem 
2.4.B]{Gromov} and Dinh--Nguy\^en \cite[Theorem~A]{DN}. The proof of the 
strictly Lorentzian case of Theorem \ref{t.intro9} relies on a generalized 
mixed Hodge--Riemann property of Hu and Xiao \cite[Corollary~A]{HX1}.

For the convenience of the reader, we summarize the assumptions we make on 
the vanishing of Hodge numbers. For most of our statements, we assume 
$H^{p-2,q-2}(X)=0$ (which, as we have seen in Remark \ref{r.intro}(d), cannot 
be removed). Before proving our general results, we discuss the special case 
$H^{p-1,q-1}(X)=0$ (Proposition \ref{p.def} and Section \ref{s.3.2}), where 
the Hodge--Riemann relations assert that the pairing on $H^{p,q}(X)$ is 
positive definite. As exceptions, our results on products of top Chern 
classes (Corollary \ref{c.multiBG} and Section \ref{s.5.1}) are deduced from 
the Bloch--Gieseker theorem (Proposition \ref{t.BG}) and do not require any 
assumption on the vanishing of Hodge numbers. 

This paper is organized as follows. In Section \ref{s.2}, we review the 
Kempf--Laksov and Fulton formulas for Schur and Schubert classes, and show 
that the degeneracy loci admit resolutions satisfying the invariance of 
$H^{p-1,q-1}$. In Section \ref{s.3}, we show that Schur classes of ample 
vector bundles induce positive definite Hermitian forms under the assumption 
$H^{p-1,q-1}(X)=0$. In Section \ref{s.4}, we develop the linear algebra 
machine that will be used in the proof of the Hodge--Riemann relations. In 
Section \ref{s.5}, we prove a theorem for cone classes, which implies the 
Hodge--Riemann relations for Schur classes (Theorem \ref{t.main}), Schubert 
classes, and products of derivatives of these classes. In Section~\ref{s.6}, 
we study Hodge--Riemann polynomials and prove Theorems \ref{t.Polya} through 
\ref{t.ifinal}. In Appendix~\ref{s.app}, by Shouda Wang, the Dowling--Zhao 
conjecture is proved using Corollary \ref{c.idLor}. 

\subsection*{Convention}
All varieties in this paper are algebraic varieties over the field of complex 
numbers~$\C$. For a smooth projective variety, we let $\NS(X)$ denote the 
group of numerical equivalence classes of divisors on~$X$. We let 
$\Amp(X)\subseteq \NS(X)_\R$ denote the ample cone. We let $\N$ denote the 
set of nonnegative integers. For $\alpha=(\alpha_1,\dots,\alpha_r)\in\N^r$, 
we put $\lvert \alpha \rvert=\alpha_1+\dots+\alpha_r$ and 
$\alpha!=\alpha_1!\dotsm \alpha_r!$. A cone in a real vector space is a 
subset stable under multiplication by positive scalars. 

\subsection*{Acknowledgments}
We learned much about Hodge--Riemann relations from the work of Ross and Toma 
\cites{RT,RT2,RT3}. Part of our work is based on the theory of Lorentzian 
polynomials and dually Lorentzian polynomials \cites{BH,BL,RSW}. We became 
interested in Hodge--Riemann relations through questions of Zhangchi Chen. We 
would like to thank him and Ping Li, Jie Liu, Wenhao Ou, Yichao Tian, and 
Qizheng Yin for useful discussions. We thank Julius Ross and Jian Xiao for 
helpful discussions and comments on drafts of this paper. We are grateful to 
Shizhang Li for comments and the suggestion to consider hypersurfaces, which 
are used in the proof of Proposition \ref{p.converse}. We are indebted to 
Enhan Li, Yiran Lin, and Haofeng Zhang for many discussions during a summer 
school in 2025, which led to Corollary \ref{c.multiBG} and several other 
improvements. We thank Shouda Wang for allowing us to include his proof of 
the Dowling--Zhao conjecture as an appendix. We thank the referee for a 
careful reading of the paper and many helpful comments. 

This work was partially supported by the National Key Research and 
Development Program of China (grant number 2020YFA0712600), the National 
Natural Science Foundation of China (grant numbers 12125107, 12271037, 
12288201), and the Chinese Academy of Sciences Project for Young Scientists 
in Basic Research (grant number YSBR-033).

\section{Degeneracy loci}\label{s.2}

In this section, we review the degeneracy loci formulas for Schur and 
Schubert classes and show that the degeneracy loci admit resolutions 
satisfying the invariance of $H^{p-1,q-1}$. In Section \ref{s.Schur}, we 
review the Kempf--Laksov formula for Schur classes. In Section 
\ref{s.Schubert}, we review the Fulton formula for Schubert classes. In 
Section \ref{s.2.3}, we review the extension of the formulas to $\R$-twisted 
vector bundles. Even though Schur classes are special cases of Schubert 
classes, we have chosen to present Schur classes first, because they are used 
more frequently in this paper. 

We fix some notation. Let $F$ be a vector bundle on a variety $X$. We write
$P=\PP_\bullet(F)$ for the projective bundle of lines in $F$. Let $\pi\colon
\PP_\bullet(F)\to X$ be the projection. We have a canonical short exact
sequence
\[0\to \cO_{\PP_\bullet(F)}(-1)\to \pi^* F\to
Q\to 0.
\]
We call $Q$ the \emph{universal quotient bundle} on $\PP_\bullet(F)$.

Let $X$ be a smooth variety and let $p\colon C\to X$ be a proper morphism of 
algebraic varieties. Let $\alpha\in H^k(C,\R)$. We define $p_*\alpha\in 
H^*(X,\R)$ by the formula $p_*\alpha\cap [X]=p_*(\alpha \cap [C])$. 

\subsection{Schur classes}\label{s.Schur}

Let $\lambda=(\lambda_1,\dots,\lambda_n)$ be a partition, namely a decreasing 
sequence of integers $\lambda_1\ge \dots \ge \lambda_n\ge 0$. We define the 
Schur polynomial $s_\lambda(x_1,\dots,x_e)$ by the formula 
\begin{equation}\label{eq.Schurconv}
s_\lambda =\det (c_{\lambda_i+j-i})_{1\le i,j\le n},
\end{equation}
where $c_i$ denotes the $i$-th elementary symmetric polynomial in 
$x_1,\dots,x_e$. By convention, $c_i=0$ for $i<0$ or $i>e$. The degree of 
$s_\lambda$ is $\lvert \lambda\rvert$, where $\lvert 
\lambda\rvert=\lambda_1+\dots+\lambda_n$. We have 
$s_\lambda(x_1,\dots,x_e)=0$ if $\lambda_1>e$. Let $E$ be a vector bundle of 
rank $e$ on a smooth projective variety $X$. The Schur class $s_\lambda(E)\in 
H^{\lvert \lambda\rvert,\lvert \lambda\rvert}(X,\R)$ is the value of 
$s_\lambda$ at the Chern roots of $E$. 

Assume $e>0$ and $\lambda_1\le e$. We fix a vector space $V$ equipped with a
partial flag
\[A_1 \subseteq A_2\subseteq \dots \subseteq A_n\subseteq V, \]
where $\dim A_i=e-\lambda_i+i$. We assume moreover that $\dim V>n$ (to avoid 
emptiness of the projective determinantal locus). Let $F=\Hom(V_X,E)$. We let 
$\hat C\subseteq F$ denote the affine determinantal locus, whose fiber over 
$x\in X$ consists of all $\sigma \colon V\to E_x$ such that 
$\dim(\ker(\sigma)\cap A_i)\ge i$. Let $C=\PP_\bullet(\hat C)\subseteq 
\PP_\bullet(F)$ be the corresponding projective determinantal locus, which 
has codimension $\lvert \lambda\rvert$ in $\PP_\bullet(F)$. Locally $C$ is a 
product over $X$. 

Let $\pi\colon C\to X$ be the projection. Let $Q$ be the universal quotient 
bundle on $\PP_\bullet(F)$ and let $f=e\cdot \dim V-1$ be the rank of $Q$. 

\begin{theorem}[Kempf--Laksov \cite{KL}]\label{t.KL}
We have
\[s_\lambda(E)=\pi_{*} c_f(Q|_C).\]
\end{theorem}

\begin{proof}
This is standard. We recall the role of the cone class $z(\hat C,F)$, which 
will be used later for products of Schur classes. We have 
$s_\lambda(E)=z(\hat C,F)$. See \cite[Lemma 3.3]{FL} for the deduction of 
this formula from \cite{KL}*{Theorem~10(ii)} (or \cite[Theorem 14.3, Remark 
14.3]{FultonIT}). (The assumption $\dim(V)=n+e$ in \cite[Lemma 3.3]{FL} can 
be dropped by \cite[(1.7)]{FL}.) Moreover, $z(\hat C,F)=\pi_{*} c_f(Q|_C)$ by 
\cite[(1.5)]{FL}. 
\end{proof}

The determinantal locus $\hat C$ admits a canonical resolution of
singularities $\hat \phi\colon \hat Z\to \hat C$, constructed as follows.
Let $\Fl$ be the flag variety associated to $(A_i)$: $\Fl$ represents the
functor carrying a $\C$-scheme $S$ to the set of flags
\[0\subseteq D_1\subseteq D_2\subseteq \dots \subseteq D_n\subseteq V_S\]
satisfying $D_i\subseteq (A_i)_S$ and $\rk(D_{i})=i$ for all $i$. Let 
$(\cD_i)$ be the universal flag on $\Fl$ and let $\hat 
Z=\Hom(p_1^*(V_{\Fl}/\cD_n),p_2^*E)$ be the vector bundle over $\Fl\times X$, 
where $p_1\colon \Fl\times X\to \Fl$ and $p_2\colon \Fl\times X\to X$ are the 
two projections from $\Fl\times X$. Let $\hat \phi\colon \hat Z\to \hat C$ be 
the obvious map induced by composition with the map $V_{\Fl}\to 
V_{\Fl}/\cD_n$. Let $\hat C^\circ\subseteq \hat C$ denote the open subvariety 
defined by $\dim(\ker(\sigma)\cap A_i)=i$ for all $i$. The restriction of 
$\hat \phi$ to $\hat C^\circ$ is an isomorphism. Thus $\hat \phi$ is a 
resolution of singularities. The projective version, $\phi\colon 
Z=\PP_\bullet(\hat Z)\to C$, is also a resolution of singularities. 

We now state the formula for products of Schur classes. Let $E_1,\dots,E_r$ 
be vector bundles of ranks $e_1,\dots,e_r>0$, respectively. Let 
$\lambda^1,\dots,\lambda^r$ be partitions satisfying $(\lambda^i)_1\le e_i$ 
for all $i$. Let $\hat C_i \subseteq F_i$ be the affine determinantal loci 
and let $\hat C=\hat C_1\times_X \dotsm \times_X \hat C_r \subseteq F_1\oplus 
\dots \oplus F_r= F$. Let $C=\PP_\bullet(\hat C)\subseteq \PP_\bullet(F)$ and 
let $\pi\colon C\to X$ be the projection. Let $Q$ be the universal quotient 
bundle on $\PP_\bullet(F)$ and let $f$ be the rank of $Q$. 

\begin{cor}\label{c.KL}
We have
\[s_{\lambda^1}(E_1)\dotsm s_{\lambda^r}(E_r)=\pi_*c_f(Q|_C).\]
\end{cor}

\begin{proof}
Indeed, $s_{\lambda^1}(E_1)\dotsm s_{\lambda^r}(E_r)=z(\hat C_1,F_1)\dotsm 
z(\hat C_r,F_r) = z(\hat C,F)=\pi_*c_f(Q|_C)$ by the proof of Theorem 
\ref{t.KL} and the multiplicativity of cone classes \cite[(3.8)]{FL}. 
\end{proof}

Let $\hat Z_i\to \hat C_i$ be the canonical resolution of singularities and
let $\hat Z=\hat Z_1\times_X \dotsm \times_X \hat Z_r$. Then
$Z=\PP_\bullet(\hat Z)\to C$ is a resolution of singularities.

\begin{lemma}\label{l.KL}
Assume that $H^{p-1,q-1}(X)=0$ with $p+q\le \dim(X)+3$. Then the map 
$H^{p,q}(X)\to H^{p,q}(Z)$ is an isomorphism. 
\end{lemma}

\begin{proof}
Since $Z$ is a projective bundle over $\Fl_1\times\dots\times \Fl_r\times X$ 
and each $\Fl_i$ is an iterated projective bundle over $\C$, $Z$ is an 
iterated projective bundle over $X$. Thus it suffices to apply Lemma 
\ref{l.Grass} below. 
\end{proof}

\begin{lemma}\label{l.Grass}
Let $f\colon Y\to X$ be an iterated Grassmannian bundle. Let $(p,q)$ be such 
that $H^{p-i,q-i}(X)=0$ for all $i\ge 1$. Then the map $f^*\colon 
H^{p,q}(X)\to H^{p,q}(Y)$ is a bijection. 
\end{lemma}

By the hard Lefschetz theorem, if $p+q\le \dim(X)+3$, then $H^{p-1,q-1}(X)=0$ 
implies $H^{p-i,q-i}(X)=0$ for all $i\ge 1$. 

\begin{proof}
First note that if $g\colon Z\to Y$ is another iterated Grassmannian bundle 
and the lemma holds for $f$, then the lemma holds for $g$ if and only if it 
holds for $fg$. Thus, by induction, it suffices to prove the lemma in the 
case where $f$ is a Grassmannian bundle. Let $Y=\Gr_r(E)$, where $E$ is a 
vector bundle on $X$ and let $F\subseteq f^*E$ be the universal subbundle on 
$Y$. Then the flag bundle $Z=\Fl(E)$ can be identified with $\Fl(F)\times_Y 
\Fl(f^*E/F)$. Since flag bundles are iterated projective bundles,  it then 
suffices, by the remark at the beginning of the proof, to prove the lemma 
under the assumption that $f$ is an iterated projective bundle. We then 
reduce readily to the case that $Y=\PP(E)$ is a projective bundle over $X$. 
This case follows from the projective bundle formula: $H^*(\PP(E),\C)$ is a 
free module over $H^*(X,\C)$ and $H^{p,q}(\PP(E))=\bigoplus_{i=0}^{r-1} 
H^{p-i,q-i}(X)\wedge\xi^{i}$, where $\xi=c_1(\cO_{\PP(E)}(1))$ and $r$ 
denotes the rank of $E$. 
\end{proof}

\subsection{Schubert classes}\label{s.Schubert}
Let $\be\colon 0=e_0<e_{1}<\dots<e_{k-1}<e_k=e$ be a sequence with $k>0$. Let 
$w=w(1)\dots w(e)$ be an $\be$-permutation, namely a sequence of distinct 
positive integers such that $w(j)<w(j+1)$ if $e_i<j<e_{i+1}$ for some $i$. 
Lascoux and Sch\"utzenberger \cite{LS} defined the Schubert polynomial 
$\fS_w(x_1,\dots,x_e)$, which is an $\be$-symmetric polynomial, namely 
symmetric in $x_j$ and $x_{j'}$ if $e_i<j<j'\le e_{i+1}$ for some~$i$. The 
degree of $\fS_w$ is the length $\ell(w)$ of $w$, which is the number of 
inversions in $w$ (when regarded as an element of $S_\infty$). Let $0= 
E_k\subseteq E_{k-1}\subseteq \dots \subseteq E_1\subseteq E_0=E$ be an 
$\be$-filtered vector bundle on $X$, namely a sequence of vector bundles such 
that $E/E_i$ has rank $e_i$. The Schubert class 
$\fS_w(E)=\fS_w(E_0/E_1,\dots,E_{k-1}/E_k)\in H^{\ell(w),\ell(w)}(X,\R)$ is 
the value of $\fS_w$ at the Chern roots $x_1,\dots,x_e$ of $E$, where 
$x_{e_{i-1}+1},\dots,x_{e_i}$ are the Chern roots of $E_{i-1}/E_{i}$. 

Let $m\ge \max\{w(1),\dots, w(e)\}$ be an integer. Let $V_1\subseteq 
V_2\subseteq \dots \subseteq V_m=V$ be a complete flag of a vector space $V$ 
of dimension $m$. Let $F=\Hom(V_X,E)$. We let $\hat C\subseteq F$ denote the 
affine degeneracy locus, which is the cone whose fiber at $x\in X$ consists 
of all $\sigma \colon V\to E_x$ such that $\rk(V_j\xrightarrow{\sigma}E_x\to 
(E/E_i)_x)\le r_{i,j}$. Here $r_{i,j}=\# \{a\le e_i\mid w(a)\le j\}$. In 
other words, $\dim(\sigma^{-1}((E_i)_x)\cap V_j)\ge j-r_{i,j}$. Let 
$C=\PP_\bullet(\hat C)\subseteq \PP_\bullet(F)$ be the corresponding 
projective degeneracy locus, which has codimension $\ell(w)$ in 
$\PP_\bullet(F)$. Locally $C$ is a product over $X$. 

Let $\pi\colon C\to X$ be the projection. Let $Q$ be the universal quotient 
bundle on $\PP_\bullet(F)$ and let $f=em-1$ be the rank of $Q$. 

\begin{theorem}[Fulton \cite{FultonFil}]\label{t.Fulton}
We have
\[\fS_w(E)=\pi_{*} c_f(Q|_C).\]
\end{theorem}

\begin{proof}
In \cite[page~630]{FultonFil}, it is deduced from \cite[Theorem 
8.2]{FultonSch} that $\fS_w(E)=z(\hat C,F)$, where $z(\hat C,F)$ denotes the 
cone class. (The assumption $m=\max\{w(1),\dots,w(e)\}$ in \cite{FultonFil} 
can be dropped by \cite[(1.7)]{FL}.) We have $z(\hat C,F)=\pi_{*} c_f(Q|_C)$ 
by \cite[(1.5)]{FL}. 
\end{proof}

The degeneracy locus $\hat C$ admits a canonical resolution of singularities 
$\hat \phi\colon \hat Z\to \hat C$, constructed as follows. For a moment we 
do not assume that $(V_j)$ is a complete flag. Given nonnegative integers 
$(d_{i,j})_{1\le i\le k,\ 1\le j\le m}$ satisfying $d_{i,j}\ge d_{i+1,j}$, 
$d_{i,j}\le d_{i,j+1}$ and $d_{1,j}\le \dim(V_j)$, let $\Fl$ be the flag 
variety associated to $(V_j)$: $\Fl$ represents the functor carrying a 
$\C$-scheme $S$ to the set of flags $(D_{i,j}\subseteq (V_j)_S)_{1\le i\le 
k,\ 1\le j\le m}$ satisfying $D_{i,j}\supseteq D_{i+1,j}$, $D_{i,j}\subseteq 
D_{i,j+1}$ and $\rk(D_{i,j})=d_{i,j}$. 

\begin{lemma}\label{l.grass}
$\Fl$ is an iterated Grassmannian bundle over $\C$.
\end{lemma}

\begin{proof}
We adopt the convention $D_{i,0}=0$ and $D_{0,j}=(V_j)_S$. Let 
$[n]=\{1,\dots, n\}$. Recall that a subset $T$ of a poset $P$ is called a 
\emph{downset} if $a\le b$ in $P$ with $b\in T$ implies $a\in T$. For any 
downset $T$ of $[k]\times [m]$, consider $\Fl_T$ representing flags 
$(D_{i,j}\subseteq V_j)_{(i,j)\in T}$. Then $\Fl_\emptyset$ is a point. If 
$T'=T\amalg \{(i,j)\}$ is also a downset, then $\Fl_{T'}$ is a Grassmannian 
bundle over $\Fl_T$ representing $D_{i,j-1}\subseteq D_{i,j}\subseteq 
D_{i-1,j}$. We conclude by the fact that there exists a chain 
\[\emptyset=T_0\subseteq T_1\subseteq \dots \subseteq T_{km}=[k]\times [m],\]
where each $T_l$ is a downset of $[k]\times [m]$ and $T_{l+1}$ is obtained 
from $T_l$ by adding one element. 
\end{proof}

Apply the above to the case where $(V_j)$ is a complete flag and 
$d_{i,j}=j-r_{i,j}$. Let $(\cD_{i,j})$ be the universal flag on $\Fl$. Let 
$p\colon \Fl\times X\to X$ be the projection and let $\hat Z\subseteq p^*F$ 
be the vector subbundle of $p^*F$ defined by the condition 
$\sigma(\cD_{i,m})\subseteq E_i$. Let $\hat \phi\colon \hat Z\to \hat C$ be 
the obvious map induced by $p$. Let $\hat C^\circ\subseteq \hat C$ denote the 
open subvariety defined by rank equalities. The restriction of $\hat \phi$ to 
$\hat C^\circ$ is an isomorphism. Thus $\hat \phi$ is a resolution of 
singularities. The projective version, $\phi\colon Z=\PP_\bullet(\hat Z)\to 
C$, is also a resolution of singularities. 

\begin{lemma}\label{l.Fulton}
Assume that $H^{p-1,q-1}(X)=0$ with $p+q\le \dim(X)+3$. Then the map 
$H^{p,q}(X)\to H^{p,q}(Z)$ is an isomorphism. 
\end{lemma}

\begin{proof}
Since $Z$ is a projective bundle over $\Fl\times X$ and $\Fl$ is an iterated 
Grassmannian bundle over $\C$ by Lemma \ref{l.grass}, $Z$ is an iterated 
Grassmannian bundle over $X$. Thus it suffices to apply Lemma \ref{l.Grass}. 
\end{proof}

\begin{remark}
Given a partition $\lambda=(\lambda_1,\dots,\lambda_n)$ satisfying 
$\lambda_1\le e$, consider the sequence $\be\colon 0=e_0<e_1=e$ and the 
$\be$-permutation $w_\lambda$ given by the elements of 
$\{1,2,\dots,e+n\}\backslash\{e+i-\lambda_i\mid 1\le i\le n\}$, arranged in 
increasing order. Then $s_\lambda=\fS_{w_\lambda}$. Moreover, for the same 
choice of $V$, the two constructions of the cone $C$ and its canonical 
resolution $\phi\colon Z\to C$ coincide.  
\end{remark}

\subsection{Derived classes of $\R$-twisted vector bundles}\label{s.2.3}
We extend the degeneracy formulas to derived classes of $\R$-twisted vector
bundles, following \cite{RT}. By an $\R$-twisted vector bundle on $X$ we mean
a vector bundle twisted by an element of $\NS(X)_\R$.

We refer to \cite[Sections 6.2, 8.1.A]{LazII} and \cite[Section 2.4]{RT} for 
basics on $\R$-twisted vector bundles. Recall that for any vector bundle $F$ 
on $X$ and $\delta\in \NS(X)_\R$, we identify $\PP_\bullet(F\langle 
\delta\rangle)$ with $\PP_\bullet(F)$ and the universal quotient bundle on 
$\PP_\bullet(F\langle \delta\rangle)$ is defined to be $Q\langle 
\pi^*\delta\rangle$, where $Q$ is the universal quotient bundle on 
$\PP_\bullet(F)$ and $\pi\colon \PP_\bullet(F)\to X$ is the projection. 

\begin{notation}
Given a polynomial $s(x_1,\dots,x_n)$, we write
\[s^{[i]}(x_1,\dots,x_n)=\frac{1}{i!}\cdot\left.\frac{d^i}{dt^i}s(x_1+t,\dots,x_n+t)\right|_{t=0}.\]
\end{notation}

Our notation differs slightly from the notation for derived Schur classes in
\cite[Definition 2.7]{RT}. We reserve the notation $s^{(i)}$ for the usual
$i$-th derivative.

For vector bundles $E$ on $X$ and $F$ on $Y$, we put $E\boxtimes
F=\pi_X^*E\otimes \pi_Y^*F$, where $\pi_X\colon X\times Y\to X$ and
$\pi_Y\colon X\times Y\to Y$ are the projections.

\begin{prop}
Let $\pi\colon C\to X$ be a proper morphism of algebraic varieties with $X$ 
smooth projective. Let $Q$ be a vector bundle on $C$ of rank $f$ and let $E$ 
be an $\be$-filtered vector bundle on $X$. Let $g\in \R[x_1,\dots,x_f]$ be a 
symmetric polynomial and let $s\in \R[x_1,\dots,x_e]$ be an $\be$-symmetric 
polynomial such that 
\begin{equation}\label{e.derive}
s(E\boxtimes L)=\pi_{Y*} g(Q\boxtimes L)
\end{equation}
for some smooth projective variety $Y$ and some line bundle $L$ on $Y$
satisfying $c_1(L)^{\max\{\deg(s),\deg(g)\}}\neq 0$, where
$\pi_Y=\pi\times\id_Y\colon C\times Y\to X\times Y$. Then we have
\[s^{[i]}(E\langle \delta\rangle)=\pi_{*} g^{[i]}(Q\langle \pi^*\delta \rangle) \]
for every $i\ge 0$ and every $\delta\in \NS(X)_\R$.
\end{prop}

\begin{proof}
This is essentially proved in \cite[Proposition 5.2]{RT}. We include a proof
for the sake of completeness. We have
\begin{gather*}
s(E\boxtimes L)=\sum_{i} s^{[i]}(E)|_{X\times Y}\cdot c_1(L)^i|_{X\times Y},\\
\pi_{Y*} g(Q\boxtimes L)=\pi_{Y*} (\sum_{i}  g^{[i]}(Q|_{C\times Y})c_1(L|_{C\times
Y})^i)
=\sum_i (\pi_* g^{[i]}(Q))|_{X\times Y}\cdot c_1(L)^i|_{X\times Y},
\end{gather*}
where we used the projection formula in the last equality. Thus 
\eqref{e.derive} and the K\"unneth formula imply $s^{[i]}(E)=\pi_* 
(g^{[i]}(Q))$ for all $i$. Therefore, 
\[s^{[i]}(E\langle \delta\rangle)=\sum_j \binom{i+j}{i} s^{[i+j]}(E)\delta^j=\sum_j \binom{i+j}{i} \pi_*g^{[i+j]}(Q)\delta^j=\pi_{*} g^{[i]}(Q\langle \pi^*\delta \rangle). \]
\end{proof}

Applying this to the formulas of Sections \ref{s.Schur} and \ref{s.Schubert}, 
we immediately get the following formulas, in which $Q$ denotes the 
$\R$-twisted universal quotient bundle on $\PP_\bullet(F)$. 

\begin{cor}[\cite{RT}*{Proposition 5.2}]\label{c.KLtwist}
Notation as in Theorem \ref{t.KL}. For every $\R$-twisted vector bundle $E$
of rank $e$, we have
\[s_\lambda^{[i]}(E)=\pi_{*} c_{f-i}(Q|_C).\]
\end{cor}

\begin{cor}\label{c.KLtwistprod}
Notation as in Corollary \ref{c.KL}. For $\R$-twisted vector bundles 
$E_1,\dots,E_r$ of ranks $e_1,\dots,e_r$, having the same $\R$-twist modulo 
$\NS(X)$, we have 
\[s_{\lambda^1}(E_1)\dotsm s_{\lambda^r}(E_r)=\pi_{*} c_{f}(Q|_C).\]
\end{cor}

\begin{cor}\label{c.Fultontwist}
Notation as in Theorem \ref{t.Fulton}. For every $\be$-filtered $\R$-twisted
vector bundle $E$, we have
\[\fS_w^{[i]}(E)=\pi_{*} c_{f-i}(Q|_C).\]
\end{cor}

\section{Positivity}\label{s.3}

In this section, we show that products of Schur classes of ample vector
bundles induce positive definite Hermitian forms on $H^{p,q}(X)$ under the
assumption $H^{p-1,q-1}(X)=0$ (Theorem \ref{t.Schurdef}). This generalizes
the case of smooth projective varieties over $\C$ of the positivity theorem
of Fulton and Lazarsfeld \cite[Theorem I]{FL}. We also prove a converse
(Proposition \ref{p.converse}), similar to \cite[Proposition 3.4]{FL}.

\subsection{Chern classes}

The following extension of the hard Lefschetz theorem is the starting point 
of our investigation. 

\begin{prop}[Bloch--Gieseker \cite{BG}*{Proposition 1.3}]\label{t.BG}
Let $E$ be an ample $\R$-twisted vector bundle of rank $e$ on a smooth
projective variety $X$ of dimension $d$. For $0\le k\le e$ and $p+q+k\le d$,
the map
\[-\wedge c_{k}(E)\colon H^{p,q}(X)\to H^{p+k,q+k}(X)\]
is injective on the intersection of the kernels of $-\wedge c_i(E)\colon
H^{p,q}(X)\to H^{p+i,q+i}(X)$, $k<i\le e$.
\end{prop}

\begin{proof}
As observed in \cite[Theorem 2.11]{RT}, the proof in \cite{BG} extends
without change to the case of an $\R$-twisted vector bundle. We recall the
argument for completeness.  We have
\[H^*(\PP(E),\C)\simeq H^*(X,\C)[\xi]/(\xi^e-c_1(E)\xi^{e-1} +\dots +(-1)^ec_e(E)),\]
where $\xi=c_1(\cO_{\PP(E)}(1))$. We may assume $k>0$. Let $\alpha\in
H^{p,q}(X)$ be such that $\alpha c_i(E)=0$ for all $k\le i\le e$. Consider
\[\eta=\xi^{k-1}-c_1(E)\xi^{k-2}+\dots +(-1)^{k-1} c_{k-1}(E).\]
Then $\alpha\eta\xi^{e-k+1}=0$. We have $\alpha\eta\in
H^{p+k-1,q+k-1}(\PP(E))$ and
\[(p+k-1+q+k-1)+(e-k+1)=p+q+k+e-1\le d+e-1=\dim(\PP(E)).\]
Thus, by the hard Lefschetz theorem, $\alpha\eta=0$. Looking at the
coefficient of $\xi^{k-1}$ in this equality, we get $\alpha=0$.
\end{proof}

\begin{notation}
For $\alpha\in H^*(X,\C)$, $\beta\in H^*(X,\C)$, and $\gamma\in H^*(X,\R)$,
we write
\[(\alpha,\beta)_\gamma=\int_X \alpha\wedge \bar\beta\wedge
\gamma.
\]
For $\alpha,\beta\in H^{p,q}(X)$ and $\gamma\in H^*(X,\R)$, we write
\[\langle \alpha,\beta\rangle_{\gamma} \colonequals
i^{q-p}(-1)^{\frac{(p+q)(p+q+1)}{2}}\int_X \alpha\wedge \bar\beta\wedge
\gamma.
\]
\end{notation}

\begin{cor}\label{c.BG}
Let $0\le k\le e$ and $p+q+k=d$. Assume that $-\wedge c_i(E)\colon
H^{p,q}(X)\to H^{p+i,q+i}(X)$ is zero for all $k<i\le e$. Then $\langle
-,-\rangle_{c_{k}(E)}$ is nondegenerate on $H^{p,q}(X)$.
\end{cor}

\begin{proof}
By Hodge symmetry and Serre duality, $\dim H^{p,q}(X)=\dim H^{p+k,q+k}(X)$. 
Thus Proposition \ref{t.BG} implies that $-\wedge c_k(E)\colon H^{p,q}(X)\to 
H^{p+k,q+k}(X)$ is a bijection. 
\end{proof}

\begin{cor}\label{c.multiBG}
Let $E_1,\dots,E_r$ be ample $\R$-twisted vector bundles of ranks 
$e_1,\dots,e_r$, respectively, on a smooth projective variety $X$ of 
dimension $d=p+q+e_1+\dots+e_r$. Assume that $E_1,\dots,E_r$ have the same 
$\R$-twist modulo $\NS(X)_\Q$. Then $\langle-,-\rangle_\gamma$ is 
nondegenerate on $H^{p,q}(X)$ for $\gamma=c_{e_1}(E_1)\dotsm c_{e_r}(E_r)$. 
\end{cor}

\begin{proof}
The following proof was suggested by Enhan Li. As above, since $\dim 
H^{p,q}(X)=\dim H^{p+e,q+e}(X)$, it suffices to show that the map 
\begin{equation}\label{e.multiBG} 
-\wedge \gamma\colon H^{p,q}(X)\to 
H^{p+e,q+e}(X)
\end{equation}
is an injection. Here $e=e_1+\dots+e_r$. By assumption, there exist 
$\delta_1,\dots,\delta_r\in \NS(X)_\Q$ such that $E_1\langle 
-\delta_1\rangle,\dots, E_r\langle-\delta_r\rangle$ have the same $\R$-twist 
modulo $\NS(X)$. By a refinement of the Bloch--Gieseker covering \cite[Lemma 
2.1]{BG} due to Koll\'ar and Mori \cite[Proposition 2.67]{KM}, there exists a 
finite dominant morphism $\pi\colon Y\to X$ with $Y$ a smooth variety such 
that $\pi^*\delta_1,\dots,\pi^*\delta_r\in \NS(Y)$. Then 
$\pi^*E_1,\dots,\pi^*E_r$ have the same $\R$-twist modulo $\NS(Y)$. Since 
$\pi^*\colon H^{p,q}(X)\to H^{p,q}(Y)$ is an injection, it suffices to show 
that $-\wedge \pi^*\gamma \colon H^{p,q}(Y)\to H^{p+e,q+e}(Y)$ is an 
injection. Thus it suffices to prove the injectivity of \eqref{e.multiBG} 
under the additional assumption that $E_1,\dots,E_r$ have the same $\R$-twist 
modulo $\NS(X)$. In this case, $\gamma=c_e(E_1\oplus \dots \oplus E_r)$ and 
the injectivity holds by Proposition \ref{t.BG}. 
\end{proof}

In the rest of this section, we concentrate on the case $H^{p-1,q-1}(X)=0$.

\begin{prop}\label{p.def}
Let $E$ be an ample $\R$-twisted vector bundle of rank $e$ on a smooth
projective variety $X$ of dimension $d$. Assume $H^{p-1,q-1}(X)=0$.  Then
$\langle -,-\rangle_{c_{k}(E)}$ is positive definite on $H^{p,q}(X)$ for
$0\le k\le e$ and $p+q+k=d$.
\end{prop}

\begin{proof}
By the hard Lefschetz theorem, since $p+q\le d\le d+3$ and 
$H^{p-1,q-1}(X)=0$, we have $H^{p-j,q-j}(X)=0$ for all $j\ge 1$. Thus, by 
Hodge symmetry and Serre duality, $H^{p+k+j,q+k+j}(X)=0$. Let $h$ be an ample 
class. We have $c_k(E\langle th \rangle)=\binom{e}{k}t^kh^k+O(t^{k-1})$ for 
$t\gg 0$. Thus $c_t\colonequals \frac{1}{(t+1)^k}c_k(E\langle th \rangle)\to 
\binom{e}{k} h^k$ as $t\to \infty$. Since $\langle -,-\rangle_{c_t}$ on 
$H^{p,q}(X)$ is nondegenerate for all $t\ge 0$ by Corollary \ref{c.BG} and 
positive definite for $t=\infty$ by the classical Hodge--Riemann relations, 
it is positive definite for all $t\ge 0$ by continuity. 
\end{proof}

\begin{cor}\label{c.def}
Let $E$ be an ample (resp.\ nef) $\R$-twisted vector bundle of rank $e$ on a 
smooth projective variety $X$ of dimension $d$. Let $h_1,\dots,h_l\in 
\NS(X)_\R$ be ample (resp.\ nef) classes, where $l\ge 0$. Assume 
$H^{p-1,q-1}(X)=0$.  Then $\langle -,-\rangle_{c_{k}(E)h_1\dotsm h_l}$ is 
positive definite (resp.\ positive semidefinite) on $H^{p,q}(X)$ for $0\le 
k\le e$ and $p+q+k+l=d$. 
\end{cor}

\begin{proof}
Ample case. We proceed by induction on $l$. The case $l=0$ is Proposition 
\ref{p.def}. Let $l\ge 1$. Since $h_l$ is a positive linear combination of 
ample classes in $\NS(X)$, we may assume $h_l\in \NS(X)$. By Bertini's 
theorem, up to multiplying $h_l$ by a positive number, we may assume that 
$h_l$ is very ample and represented by a smooth subvariety $Y$ of $X$ of 
dimension $d-1$. Then $\langle -,-\rangle_{c_{k}(E)h_1\dotsm h_l}=\langle 
\iota^*-,\iota^*-\rangle_{\iota^*(c_{k}(E)h_1\dotsm h_{l-1})}$, where 
$\iota\colon Y\hookrightarrow X$ is the embedding. Since $p+q\le d-1$, by the 
Lefschetz hyperplane theorem, $\iota^*\colon H^{p,q}(X)\to H^{p,q}(Y)$ is an 
injection and $0=H^{p-1,q-1}(X)\simeq H^{p-1,q-1}(Y)$. We conclude by the 
induction hypothesis. 

Nef case. Let $h\in N^1(X)$ be an ample class. Then $E\langle th\rangle$ and 
$h_i+th$ are ample for all $t>0$. By the ample case above, $\langle 
-,-\rangle_{c_k(E\langle th\rangle)(h_1+th)\dotsm (h_l+th)}$ is positive 
definite on $H^{p,q}(X)$ for all $t>0$. By continuity, it follows that 
$\langle -,-\rangle_{c_k(E)h_1\dotsm h_l}$ is positive semidefinite on 
$H^{p,q}(X)$. 
\end{proof}

\subsection{Cone classes}\label{s.3.2}

\begin{remark}\label{r.push}
Let $\pi\colon Z\to X$ be a morphism of smooth projective varieties and let
$\gamma\in H^*(Z,\R)$ be a class such that $\langle -,-\rangle_{\gamma}$ is
positive semidefinite on $H^{p,q}(Z)$. Then $\langle
-,-\rangle_{\pi_*\gamma}=\langle\pi^*-,\pi^*-\rangle_\gamma$ is positive
semidefinite on $H^{p,q}(X)$.
\end{remark}

By an \emph{alteration} we mean a surjective morphism between projective 
varieties that is generically finite. 

\begin{remark}\label{r.pullpush}
More generally, let $\pi\colon C\to X$ be a morphism of projective varieties 
with $X$ smooth and let $\phi\colon Z\to C$ be an alteration with $Z$ smooth 
projective. Let $\gamma\in H^*(C,\R)$ be a class such that $\langle 
-,-\rangle_{\phi^*\gamma}$ is positive semidefinite on $H^{p,q}(Z)$. Then 
$\langle -,-\rangle_{\pi_*\gamma}$ is positive semidefinite on $H^{p,q}(X)$. 
Indeed, $\pi_*\gamma =\frac{1}{m}(\pi\phi)_*\phi^*\gamma$, where $m$ denotes 
the generic degree of $\phi$. 
\end{remark}

\begin{theorem}\label{t.cone}
Let $X$ be a smooth projective variety of dimension $d$ and let $F$ be an 
$\R$-twisted vector bundle on $X$ of rank $r+1$. Let $C\subseteq 
\PP_\bullet(F)$ be a closed subvariety of dimension $d_{C}$ dominating $X$ 
and let $Q$ be the universal quotient bundle on $\PP_\bullet(F)$.  Assume 
that $H^{p-1,q-1}(X)=0$ and there exists an alteration $\phi\colon Z\to C$ 
with $Z$ smooth projective such that $H^{p-1,q-1}(Z)=0$. Let $\pi\colon C\to 
X$ be the projection. Let $d_C-d\le k\le r$, $p+q+k+l=d_C$, where $l\ge 0$. 
Let $h_1,\dots,h_{d-p-q}\in \NS(X)_\R$ be nef classes on $X$ such that 
$\langle-,-\rangle_{h_1\dotsm h_{d-p-q}}$ is positive definite on 
$H^{p,q}(X)$ and $F\langle -h_i\rangle$ is nef for all $l+1\le i\le d-p-q$. 
Then $\langle -,-\rangle_{\pi_*(c_{k}(Q|_C)) h_1\dots h_l}$ is positive 
definite on $H^{p,q}(X)$. 
\end{theorem}

\begin{proof}
We proceed by induction on $k$. Assume $k=d_C-d$. By generic flatness, there 
exists a dense open $U\subseteq X$ such that $\pi^{-1}(U)\to U$ is flat. 
Choose a closed point $x\in U$ and let $W=\pi^{-1}(x)$. Then we have 
\[\pi_*(c_k(Q|_C))=1_X\int_{W} c_{k} (Q|_W),
\]
where 
\[
\int_{W} c_{k}(Q|_W)=\int_{\PP_\bullet(F_{x})}c_1(\cO(1))^{k}\cap [W]>0.
\]
Moreover, by assumption, $\langle -,-\rangle_{h_1\dotsm h_l}$ is positive 
definite on $H^{p,q}(X)$. This finishes the proof of the case $k=d_C-d$. 

Assume now that $k>d_C-d$. Let $\alpha\in H^{p,q}(X)$ be nonzero. Consider 
the function 
\[f(t)=\langle
\alpha,\alpha\rangle_{
\pi_*(c_{k}(Q|_C\langle t\pi^*
h_{l+1}\rangle))h_1\dotsm h_l}.
\]
It suffices to show that $f(0)>0$. Since $F\langle th_{l+1}\rangle$ is nef 
for $\lvert t\rvert \le 1$, we have $f(t)\ge 0$  for such $t$ by Corollary 
\ref{c.def} and Remark \ref{r.pullpush}. Since 
$\left.\frac{d}{dt}c_{k}(Q|_C\langle 
t\pi^*h_{l+1}\rangle)\right|_{t=0}=(r-k+1)c_{k-1}(Q|_C)\pi^*h_{l+1}$, we have 
\[f'(0)=(r-k+1)\langle
\alpha,\alpha\rangle_{
\pi_*(c_{k-1}(Q|_C))h_1\dotsm h_{l+1}}>0.
\]
Here we used the induction hypothesis. Thus $f(0)>0$, as desired.
\end{proof}

\begin{cor}\label{c.Schurderdef}
Let $X$ be a smooth projective variety of dimension $d$. Let $E$ be an ample 
(resp.\ nef) $\R$-twisted vector bundle of rank $e$ on $X$. Let $\lambda$ be 
a partition with $\lambda_1\le e$. Assume $H^{p-1,q-1}(X)=0$. Then 
$\langle-,-\rangle_{s^{[i]}_\lambda(E)}$ is positive definite (resp.\ 
positive semidefinite) on $H^{p,q}(X)$ for $i\le \lvert \lambda \rvert$ and 
$p+q+\lvert \lambda\rvert-i=d$. 
\end{cor}

\begin{proof}
By the derived Kempf--Laksov formula (Corollary \ref{c.KLtwist}), 
$s^{[i]}_{\lambda}(E)=\pi_*c_{f-i}(Q|_C)$ in the notation of Theorem 
\ref{t.KL}. Let $\phi\colon Z\to C$ be the resolution constructed in Section 
\ref{s.Schur}. By Lemma \ref{l.KL}, $H^{p-1,q-1}(Z)=0$. The ample case of 
Corollary \ref{c.Schurderdef} then follows from Theorem \ref{t.cone} applied 
to $l=0$ and $h_1=\dots=h_{d-p-q}\in \NS(X)_\Q$ ample. The nef case follows 
by continuity. 
\end{proof}

\begin{theorem}\label{t.Schurdef}
Let $X$ be a smooth projective variety of dimension $d$. Let $E_1,\dots,E_r$ 
be ample (resp.\ nef) $\R$-twisted vector bundles on $X$ of ranks $e_1,\dots, 
e_r$, respectively. Assume that the $E_i$ have the same $\R$-twist modulo 
$\NS(X)$. For each $1\le i\le r$, let $\lambda^{i}$ be a partition with 
$(\lambda^{i})_1\le e_i$. Assume $H^{p-1,q-1}(X)=0$. Then 
$\langle-,-\rangle_\gamma$ is positive definite (resp.\ positive 
semidefinite) on $H^{p,q}(X)$ for 
\[\gamma=s_{\lambda^{1}}(E_1)\dotsm s_{\lambda^{r}}(E_r),\]
where $p+q+\lvert \lambda^{1}\rvert+\dots+\lvert \lambda^{r}\rvert=d$.
\end{theorem}

We will see in Corollary \ref{c.posfinal} that the assumption that the $E_i$ 
have the same $\R$-twist modulo $\NS(X)$ can be removed. 

\begin{proof}
As in the proof of Corollary \ref{c.Schurderdef}, the ample case follows from 
Theorem \ref{t.cone}, the formula $s_{\lambda^{1}}(E_1)\dotsm 
s_{\lambda^{r}}(E_r)=\pi_*c_f(Q|_C)$ (Corollary \ref{c.KLtwistprod}), and 
Lemma \ref{l.KL}. The nef case follows by continuity. 
\end{proof}

\begin{remark}
The nef cases of Corollary \ref{c.Schurderdef} and Theorem \ref{t.Schurdef} 
also follow more directly from Corollary \ref{c.def}, Remark 
\ref{r.pullpush}, and the pushforward formulas. 
\end{remark}

\begin{example}
Let $A$ be an abelian variety of dimension $d$. Consider the semipositive and 
nef cones $\mathrm{Semi}^k(A)\subseteq \Nef^k(A)$ of  $N^k(A)_\R\subseteq 
H^{k,k}(A,\R)$. By \cite[Theorems A, B]{DELV}, the inclusion 
$\mathrm{Semi}^k(A)\subseteq \Nef^k(A)$ can be strict. Now let $\lambda$ be a 
partition such that $\lvert \lambda\rvert=k$ and let $E$ be a nef vector 
bundle on $A$. Then \cite[Theorem~I]{FL} implies $s_\lambda(E)\in \Nef^k(A)$, 
whereas Theorem \ref{t.Schurdef} implies the stronger result $s_\lambda(E)\in 
\mathrm{Semi}^k(A)$. Indeed, by Theorem \ref{t.Schurdef}, 
$\langle-,-\rangle_{s_\lambda(E)}$ is positive semidefinite on 
$H^{d-k,0}(A)\simeq \bigwedge^{d-k} H^{1,0}(A)$. 
\end{example}

Theorem \ref{t.Schurdef} has the following converse, similar to 
\cite[Proposition 3.4, Lemma A.4]{FL}. Together with Theorem 
\ref{t.Schurdef}, this gives a geometric characterization of generalized 
Schur positivity. 

\begin{prop}\label{p.converse}
Let $g\in \cS^k_{e_1,\dots,e_r}$. Assume that one of the following conditions 
holds. 
\begin{enumerate}
\item For every smooth projective variety $X$ of dimension $k$ satisfying 
    $H^{p,q}(X)=0$ for all pairs $(p,q)$ satisfying $p\neq q$ and for all 
    globally generated vector bundles $E_1,\dots,E_r$ of ranks 
    $e_1,\dots,e_r$, respectively, we have 
\[
\int_X g(E_1,\dots,E_r)\ge 0.
\]

\item There exist integers $p,q\ge 0$ with $p\neq q$ such that for every 
    smooth projective variety $X$ of dimension $p+q+k$ satisfying 
    $H^{p-1,q-1}(X)=0$ and for all globally generated vector bundles 
    $E_1,\dots,E_r$ of ranks $e_1,\dots,e_r$, respectively, $\langle 
    -,-\rangle_{g(E_1,\dots,E_r)}$ is positive semidefinite on 
    $H^{p,q}(X)$. 
    
\item There exist integers $p,q$ with $\min(p,q)=0$ such that for every 
    smooth projective variety $X$ of dimension $p+q+k$ and for all ample 
    vector bundles $E_1,\dots,E_r$ of ranks $e_1,\dots,e_r$, respectively, 
    $\langle -,-\rangle_{g(E_1,\dots,E_r)}$ is positive semidefinite on 
    $H^{p,q}(X)$. 
\end{enumerate}
Then
\begin{equation}\label{e.g}
g=\sum_{\lambda^1,\dots,\lambda^r}
a_{\lambda^1,\dots,\lambda^r}s_{\lambda^1}(x_{1,1},\dots,x_{1,e_1})\dotsm
s_{\lambda^r}(x_{r,1},\dots,x_{r,e_r})
\end{equation}
where $\lambda^1,\dots,\lambda^r$ run through partitions with $\lvert
\lambda^1\rvert +\dots +\lvert \lambda^r\rvert =k$ and
$a_{\lambda^1,\dots,\lambda^r}\ge 0$.
\end{prop}

\begin{remark}\label{r.converse}
By continuity, in (a) and (b) above we may replace globally generated vector 
bundles by ample $\Q$-twisted vector bundles with the same $\Q$-twist modulo 
$\NS(X)$. 
\end{remark}

\begin{proof}[Proof of Proposition \ref{p.converse}]
Let us first prove that (c) implies (a) or (b). More precisely, we show that 
(a) holds if (c) holds for $p=q=0$ and (b) holds if (c) holds for some 
$(p,q)$ satisfying $(p,q)\neq (0,0)$, which, combined with the assumption 
$\min(p,q)=0$, implies $p\neq q$. By Remark \ref{r.converse}, we may assume 
that $E_1,\dots,E_r$ are ample $\Q$-twisted vector bundles. By the refined 
Bloch--Gieseker covering \cite[Proposition 2.67]{KM}, there exists a finite 
dominant morphism $\pi\colon Y\to X$ with $Y$ a smooth variety such that 
$\pi^*E_1,\dots,\pi^*E_r$ are non-twisted vector bundles. Thus $\langle 
-,-\rangle_{g(E_1,\dots,E_r)}=\frac{1}{\deg(\pi)}\langle\pi^*-,\pi^*-\rangle_{g(\pi^*E_1,\dots,\pi^*E_r)}$ 
is positive semidefinite. 

Next we show that (b) implies (a). Let $S\subseteq \PP^{p+q+1}$ be a smooth 
hypersurface of degree $\ge p+q+2$. Then $H^{p,q}(S)\neq 0$ and 
$H^{p-1,q-1}(S)=0$. See for example \cite[Corollary 17.5.4]{Arapura}. Let 
$X$, $E_1,\dots, E_r$ be as in (a). In particular, $H^{p',q'}(X)=0$ for all 
pairs $(p',q')$ satisfying $p'\neq q'$. Let $\pi_X\colon X\times S\to X$ and 
$\pi_S\colon X\times S\to S$ be the projections. By the K\"unneth formula, 
$H^{p-1,q-1}(X\times S)=0$.  By (b), 
\[\langle
\pi_S^*-,\pi_S^*-\rangle_{g(\pi_X^*E_1,\dots,\pi_X^*E_r)}=\int_X
g(E_1,\dots,E_r)\cdot \langle -,- \rangle_{1}
\]
is positive semidefinite on $H^{p,q}(S)$. By classical Hodge--Riemann
relations, $\langle -,- \rangle_{1}$ is positive definite on $H^{p,q}(S)$.
The inequality in (a) follows.

By the two preceding paragraphs, (c) also implies (a). It remains to show 
that the assertion of the proposition holds under condition (a). Write $g$ as 
in \eqref{e.g} with $a_{\lambda^1,\dots,\lambda^r}\in \R$. To avoid 
redundancy, we may assume that the partitions in \eqref{e.g} satisfy 
$(\lambda^i)_1\le e_i$ and we adjust trailing zeroes so that $\lambda^i$ has 
$\lvert \lambda^i\rvert$ parts. Assume that $a_{\mu^1,\dots,\mu^r}<0$ for 
some $r$-tuple of partitions $(\mu^1,\dots,\mu^r)$. As in the proof of 
\cite[Proposition 3.4]{FL}, for each $i$, let $Y_i\subseteq \Gr(\lvert 
\mu^i\rvert, \lvert \mu^i\rvert+e_i)$ be the Schubert variety dual to 
$\mu^i$, so that $\int_{Y_i} s_{\lambda}(Q_i)=\delta_{\lambda,\mu^i}$ for all 
partitions $\lambda$, where $Q_i$ is the restriction to $Y_i$ of the 
universal quotient bundle of rank $e_i$ on the Grassmannian $\Gr(\lvert 
\mu^i\rvert, \lvert \mu^i\rvert+e_i)$. Let $\phi_i\colon X_i\to Y_i$ be one 
of the resolutions of singularities constructed by Zelevinskii \cite{Zele} 
(as a special case of the Gelfand--MacPherson resolutions \cite[Section 
2.12]{GM}). Since $X_i$ is an iterated Grassmannian bundle over~$\C$, we have 
$H^{p,q}(X_i)=0$ for all $p\neq q$ by Lemma \ref{l.Grass}. Let 
$\phi=\phi_1\times \dots \times \phi_r\colon X\to Y$, where 
$X=X_1\times\dotsm \times X_r$ and $Y=Y_1\times\dotsm \times Y_r$. Let 
$\pi_i\colon Y\to Y_i$ be the projection. Then 
\begin{align*}
\int_X g(\phi^*\pi_1^*Q_1,\dots,\phi^*\pi_r^*Q_r)&=\int_Y g(\pi_1^*Q_1,\dots,\pi_r^* Q_r)=a_{\mu^1,\dots,\mu^r}<0,
\end{align*}
which contradicts the assumption (a).
\end{proof}

The use of hypersurfaces in the above proof was suggested by Shizhang Li.

As a consequence, we recover the Schur positivity of derived Schur
polynomials.

\begin{cor}\label{c.pos}
Let $\lambda$ be a partition with $\lambda_1\le e$. Then
$s_\lambda^{[i]}(x_1,\dots,x_e)$ is Schur positive and nonzero for all $0\le
i\le \lvert \lambda\rvert$. That is, $s_\lambda^{[i]}(x_1,\dots,x_e)$ is a
nonzero nonnegative linear combination of Schur polynomials in
$x_1,\dots,x_e$.
\end{cor}

\begin{proof}
The nonnegativity follows immediately from Corollary \ref{c.Schurderdef} and
Proposition \ref{p.converse}. (This argument was already given in
\cite[Remark 5.4]{RT2}.) Moreover, by Corollary \ref{c.Schurderdef},
$s_\lambda^{[i]}\neq 0$.
\end{proof}

\begin{remark}\label{r.pos}
Corollary \ref{c.pos} also follows from the explicit formula
\[s^{[1]}_\lambda(x_1,\dots,x_e)=\sum_{\mu}(e-\lambda_i+i)s_\mu,\]
where $\mu$ runs through partitions whose Young diagram can be obtained by 
removing one box from that of $\lambda$ and $i$ is such that 
$\mu_i=\lambda_i-1$. We refer to \cite[Theorems 1.1, 1.5]{CK} for a formula 
for $s_\lambda^{[i]}$.\footnote{The convention on Schur polynomials in 
\cite{CK} differs from ours by conjugation of the partition.} 
\end{remark}

\section{Hermitian forms and Hodge--Riemann pairs}\label{s.4}

In this section, we develop the linear algebra machine that will be used in 
the proof of the Hodge--Riemann relations. In spirit this extends the machine 
of Ross and Toma \cite[Section~3]{RT3} for Lorentzian forms to Hermitian 
forms in general, even though our axioms do not exactly match theirs. The end 
result is a criterion (Corollary \ref{c.iterate}) that, roughly speaking, 
allows one to deduce the Hodge--Riemann property for a Hermitian form in a 
one-parameter family from the same property for its derivatives. 

All vector spaces in this section are finite-dimensional complex vector
spaces. Unless otherwise stated, Hermitian spaces are not assumed to be
positive definite or nondegenerate.

\subsection{Definitions and first properties}

Let $V$ and $W$ be vector spaces. Let $\Herm(V)$ denote the space of
Hermitian forms on $V$ and let $\Sesq(W,V)$ denote the space of sesquilinear
maps $\Phi\colon W\times V\to \C$. We equip them with the usual topology. Let
$\Sesq^\circ(W,V)\subseteq \Sesq(W,V)$ denote the open subset consisting of
left nondegenerate $\Phi$, namely such that $\Phi(w,v)=0$ for all $v\in V$
implies $w=0$.

\begin{defn}
We call a pair $(H,\Phi)\in \Herm(V)\times \Sesq^\circ(W,V)$ such that $H$ is 
positive definite on $W^\perp_\Phi\colonequals \{v\in V\mid 
\Phi(w,v)=0,\forall w\in W\}$ a \emph{pre-Hodge--Riemann pair}. We let 
$\HR(V,W)$ denote the set of pre-Hodge--Riemann pairs and let 
$\overline{\HR}(V,W)$ denote its closure in $\Herm(V)\times \Sesq(W,V)$. We 
let $\HRw(V,W)\subseteq\Herm(V)\times \Sesq(W,V)$ denote the subset 
consisting of pairs $(H,\Phi)$ such that either $\Phi$ is left degenerate or 
$H$ is positive semidefinite on $W^\perp_\Phi$. 
\end{defn}

Our first task is to compare these subsets.

\begin{lemma}\label{l.cont0}
The map $r_{\pm}\colon \Herm(V)\times \Sesq^\circ(W,V)\to \N$ sending
$(H,\Phi)$ to the positive (resp.\ negative) index of inertia of $H$ on
$W_\Phi^\perp$ is lower semicontinuous. In other words, $r_+^{-1}(\N_{\le
a})$ and $r_-^{-1}(\N_{\le a})$ are closed for every $a\in \N$.
\end{lemma}

\begin{proof}
Indeed, the complement of $r_+^{-1}(\N_{\le a})$ (resp.\ $r_-^{-1}(\N_{\le
a})$) is open by the continuity of eigenvalues.
\end{proof}

We deduce the following consequences.

\begin{lemma}\label{l.cont1}
Let $S_1\subseteq \Herm(V)\times \Sesq^\circ(W,V)$ be the subset consisting 
of pairs $(H,\Phi)$ such that $H$ is positive semidefinite on $W^\perp_\Phi$. 
We have inclusions 
\[
\HR(V,W)\subseteq S_1\subseteq \overline{\HR}(V,W)\subseteq \HRw(V,W),
\]
where $\HR(V,W)$ is open and $\HRw(V,W)$ is closed in $\Herm(V)\times
\Sesq(W,V)$.
\end{lemma}

\begin{proof}
The openness of $\HR(V,W)$ follows from the fact that it is the complement of
the closed subset $r_+^{-1}(\N_{\le c-1})$ in $\Herm(V)\times
\Sesq^\circ(W,V)$, where $c=\dim(V)-\dim(W)$. Moreover, $S_1=r_-^{-1}(\N_{\le
0})$ is a closed subset of $\Herm(V)\times \Sesq^\circ(W,V)$. It follows that
$\HRw(V,W)$ is a closed subset of $\Herm(V)\times \Sesq(W,V)$, which implies
$\overline{\HR}(V,W)\subseteq \HRw(V,W)$. The inclusion $\HR(V,W)\subseteq
S_1$ is trivial. To see $S_1\subseteq \overline{\HR}(V,W)$, note that for
every $(H,\Phi)\in S_1$ and every positive definite $I\in \Herm(V)$, we have
$(H+tI,\Phi)\in \HR(V,W)$ for all $t>0$.
\end{proof}

One way to get pairs $(H,\Phi)$ is via the map
\[\Comp\colon \Herm(V)\times
\Map(W,V)\to \Herm(V)\times \Sesq(W,V)
\]
sending $(H,\iota)$ to $(H,H(\iota-,-))$.

\begin{defn}
A pair $(H,\iota)\in \Herm(V)\times \Map(W,V)$ is called a 
\emph{Hodge--Riemann pair} if $H(\iota-,\iota-)$ is negative definite and $H$ 
is positive definite on $(\iota W)_H^\perp$. We let $\fHR(V,W)$ denote the 
set of Hodge--Riemann pairs and let $\overline{\fHR}(V,W)$ denote the closure 
of $\fHR(V,W)$ in $\Herm(V)\times \Map(W,V)$. We let $\fHRw(V,W)$ (resp.\ 
$\fHRvw(V,W)$) denote the set of pairs $(H,\iota)\in \Herm(V)\times 
\Map(W,V)$ such that $H(\iota-,\iota-)$ is negative semidefinite and either 
$H(\iota-,-)$ is left degenerate (resp.\ $H(\iota-,\iota-)$ is degenerate) or 
$H$ is positive semidefinite on $(\iota W)^\perp_H$. 
\end{defn}

\begin{remark}\label{r.comp}
We have $\fHR(V,W)\subseteq \Comp^{-1}(\HR(V,W))$. For any $(H,\iota)\in 
\Comp^{-1}(\HR(V,W))$, $\iota$ is injective and $H(\iota -,\iota-)$ and $H$ 
are nondegenerate by Lemma \ref{l.2} below. Since negative definiteness 
(resp.\ negative semidefiniteness) is an open (resp.\ closed) condition, it 
follows that $\fHR(V,W)\subseteq \Comp^{-1}(\HR(V,W))$ is an open and closed 
subset. If $\dim W>\dim V$, both $\HR(V,W)$ and $\fHR(V,W)$ are empty. If 
$\dim W\le \dim V$, one can show that $\fHR(V,W)$ is a connected component of 
$\Comp^{-1}(\HR(V,W))$. 
\end{remark}

\begin{lemma}\label{l.2}
Let $(V,H)$ be a Hermitian space and let $W\subseteq V$ be a subspace such 
that $H$ is nondegenerate on $W^\perp_H$. Then $H$ is nondegenerate on $W$ 
and $V= W\oplus  W^\perp_H$. 
\end{lemma}

\begin{proof}
Indeed, $W\cap W^\perp_H\subseteq (W^\perp_H)^\perp_H\cap W^\perp_H=0$. Thus 
$H$ is nondegenerate on $W$ and the decomposition follows. 
\end{proof}

\begin{lemma}\label{l.1}
Let $(V,H)$ be a Hermitian space and let $W\subseteq V$ be a nondegenerate 
subspace of negative index of inertia $s$. Let $v_1,\dots,v_{d_W}$ be a basis 
of $W$. Then the following conditions are equivalent. 
\begin{enumerate}
\item $H$ is positive semidefinite (resp.\ positive definite) on
    $W^\perp_H$.
\item The negative index of inertia of $H$ is $s$ (resp.\ and $H$ is 
    nondegenerate). 
\item For all $v_0\in V$, $(-1)^{s}\det(H(v_i,v_j))_{0\le i,j\le d_W}\ge 0$ 
    (resp.\ and equality holds if and only if $v_0\in W$). 
\end{enumerate}
\end{lemma}

In the case where $H$ is negative definite on $W$, another way to state 
condition (b) above is that $W$ is a maximal negative definite (resp.\ 
semidefinite) subspace of $V$. 

\begin{proof}
The equivalence of (a) and (b) follows easily from the fact that for any 
nondegenerate subspace $W\subseteq V$, we have $V=W\oplus W^\perp_H$. 

(a)$\implies$(c). If $v_0\in W$, then the determinant is clearly zero. Assume 
that $v_0\notin W$. Let $V_0=W+\C v_0$. Then the negative index of inertia of 
$H$ on $V_0$ is $s$. Thus $(-1)^{s}\det(H(v_i,v_j))_{0\le i,j\le d_W}\ge 0$. 
If, moreover, $H$ is positive definite on $W^\perp_H$, then $H|_{V_0\times 
V_0}$ is nondegenerate by Lemma \ref{l.2} and the determinant is nonzero. 

(c)$\implies$(a). Let $v_0\in W^\perp_H$ be a nonzero vector. Then
$(H(v_i,v_j))_{0\le i,j\le d_W}$ has the form
\[\begin{pmatrix}H(v_0,v_0)&0\\0&M_W\end{pmatrix}\]
where $M_W$ has negative index of inertia $s$. Thus $H(v_0,v_0)\ge 0$ (resp.\ 
$H(v_0,v_0)>0$). 
\end{proof}

Without assuming $W$ nondegenerate in Lemma \ref{l.1}, we still have 
(b)$\implies$(a). Indeed, this follows from Lemma \ref{l.1} applied to a 
maximal negative definite subspace of $W$. 

The inequality in Lemma \ref{l.1} for $s=d_W$ extends to the case where $H$ 
is only assumed to be negative semidefinite on $W$ as follows. 

\begin{lemma}\label{l.ineq}
Let $(V,H)$ be a Hermitian space and let $\iota\colon W\to V$ be a linear map 
such that $H(\iota-,\iota-)$ is negative semidefinite. Let 
$w_1,\dots,w_{d_W}$ be a basis of $W$ and let $v_i=\iota w_i$ for $1\le i\le 
d_W$. Then $(-1)^{d_W}\det(H(v_i,v_j))_{0\le i,j\le d_W}\ge 0$ holds for all 
$v_0\in V$ if and only if $(H,\iota)\in \fHRvw(V,W)$. 
\end{lemma}

\begin{proof}
By Lemma \ref{l.1}, it remains to show that the inequality holds in the case
that $H(\iota-,\iota-)$ is degenerate. In this case, we may assume that
$v_{d_W}\in (\iota W)^\perp_H$. Then
\begin{align*}
&(-1)^{d_W}\det(H(v_i,v_j))_{0\le i,j\le
d_W}
\\=&(-1)^{d_W-1}\lvert H(v_0,v_{d_W})\rvert^2 \det(H(v_i,v_j))_{1\le i,j\le
d_W-1}\ge 0.
\end{align*}
\end{proof}

\begin{lemma}\label{l.fHR}
Let $W$ and $V$ be vector spaces. Let $Z_1$ denote the set of pairs 
$(H,\iota)\in \Herm(V)\times \Map(W,V)$ such that $H(\iota-,\iota-)$ is 
negative definite and $H$ is positive semidefinite on $(\iota 
    W)^\perp_H$.
We have inclusions
\[\fHR(V,W)\subseteq Z_1\subseteq \overline{\fHR}(V,W)\subseteq \fHRw(V,W)\subseteq \fHRvw(V,W),\]
where $\fHR(V,W)$ is open and $\fHRw(V,W)$ and $\fHRvw(V,W)$ are closed in
$\Herm(V)\times \Map(W,V)$.
\end{lemma}

\begin{proof}
By Lemma \ref{l.cont1} and Remark \ref{r.comp}, $\fHR(V,W)$ is open. By Lemma 
\ref{l.ineq} and the fact that negative semidefiniteness is a closed 
condition, $\fHRvw(V,W)$ is closed. Moreover, $\fHRw(V,W)$ is a closed subset 
of $\Comp^{-1}(\HRw(V,W))$, which is closed by Lemma \ref{l.cont1}. This 
implies $\overline{\fHR}(V,W)\subseteq \fHRw(V,W)$. To see $Z_1\subseteq 
\overline{\fHR}(V,W)$, note that for every $(H,\iota)\in Z_1$ and every 
positive definite $I\in \Herm(V)$, we have $(H+tI,\iota)\in \fHR(V,W)$ for 
$t>0$ sufficiently small. The inclusions $\fHR(V,W)\subseteq Z_1$ and 
$\fHRw(V,W)\subseteq \fHRvw(V,W)$ are trivial. 
\end{proof}

\begin{remark}
The first three inclusions in the above lemma are all strict if $\dim V>\dim 
W>0$. We have $\fHRw(V,W)=\fHRvw(V,W)$ if $\dim V=1+\dim W$ and 
$\fHRw(V,W)\subsetneq \fHRvw(V,W)$ if $\dim V>1+\dim W>1$. 
\end{remark}

\begin{remark}\label{r.summand}
Let $(H,\iota)\in \Herm(V)\times \Map(W,V)$ and $(H',\iota')\in
\Herm(V')\times \Map(W',V')$. Then
\begin{enumerate}
\item $(H\oplus H',\iota\oplus \iota')\in \fHR(V\oplus V',W\oplus W')$ if
    and only if $(H,\iota)\in \fHR(V,W)$ and  $(H',\iota')\in \fHR(V',W')$.
\item Assume that $H(\iota-,\iota-)$ and $H'(\iota'-,\iota'-)$ are negative 
    definite. Then $(H\oplus H',\iota\oplus \iota')\in \fHRw(V\oplus 
    V',W\oplus W')$ if and only if $(H,\iota)\in \fHRw(V,W)$ and 
    $(H',\iota')\in \fHRw(V',W')$. The same holds for $\fHRvw$. Indeed, 
    $Z_1$ is compatible with direct sums. 
\end{enumerate}
\end{remark}

Next we discuss functoriality of the Hodge--Riemann property.

\begin{lemma}\label{l.func0}
Let $W$ be a vector space and let $f\colon V'\to V$ be a linear map. 
\begin{enumerate}
\item For $(H,\Phi)\in \HR(V,W)$ such that $\Phi(-,f-)$ is left 
    nondegenerate, we have $(H(f-,f-),\Phi(-,f-))\in \overline{\HR}(V',W)$. 
    Moreover, if $f$ is an injection, then $(H(f-,f-),\Phi(-,f-))\in 
    \HR(V',W)$. 

\item For any $H\in \Herm(V)$ and any linear map $\iota\colon W\to V'$ such 
    that $(H,f\iota)\in \fHR(V,W)$, we have $(H(f-,f-),\iota)\in 
    \overline{\fHR}(V',W)$. Moreover, if $f$ is an injection, then 
    $(H(f-,f-),\iota)\in \fHR(V',W)$. 
\end{enumerate} 
\end{lemma}

\begin{proof}
(a) For $v\in V'$ satisfying $\Phi(w,fv)=0$ for all $w\in W$, we have 
$H(fv,fv)\ge 0$ and equality holds if and only if $fv=0$. Thus 
$(H(f-,f-),\Phi(-,f-))\in S_1(V',W)$ and, if $f$ is an injection,  
$(H(f-,f-),\Phi(-,f-))\in \HR(V',W)$. 

(b) By assumption, $H(f\iota-,f\iota-)$ is negative definite. Moreover, for 
any $v\in V'$ satisfying $H(f\iota w,fv)=0$ for all $w\in W$, we have 
$H(fv,fv)\ge 0$ and equality holds if and only if $fv=0$. Thus 
$(H(f-,f-),\iota)\in Z_1(V',W)$ and, if $f$ is injective,  
$(H(f-,f-),\iota)\in \fHR(V',W)$. 
\end{proof}

The subsets $\HRw$ and $\fHRw$ enjoy better functoriality.

\begin{lemma}\label{l.func}
Let $W$ be a vector space and let $f\colon V'\to V$ be a linear map.
\begin{enumerate}
\item For $(H,\Phi)\in \HRw(V,W)$, we have $(H(f-,f-),\Phi(-,f-))\in
    \HRw(V',W)$.
\item For any $H\in \Herm(V)$ and any linear map $\iota\colon W\to V'$ such
    that $(H,f\iota)\in \fHRw(V,W)$, we have $(H(f-,f-),\iota)\in
    \fHRw(V',W)$.
\end{enumerate}
\end{lemma}

The subset $\fHRvw$ enjoys the same functoriality as in (b).

\begin{proof}
(a) We may assume that $\Phi(-,f-)$ is left nondegenerate. Then $\Phi$ is
left nondegenerate. Thus $H$ is positive semidefinite on $W^\perp_\Phi$. It
follows that $H(f-,f-)$ is positive semidefinite on $W^\perp_{\Phi(-,f-)}$.

(b) By assumption, $H(f\iota-,f\iota-)$ is negative semidefinite. Moreover,
$\Comp (H(f-,f-),\iota)=(H(f-,f-),H(f\iota-,f-))\in \HRw(V',W)$ by (a).
\end{proof}

\begin{lemma}\label{l.trivial}
Let $(V,H)$ be a Hermitian space, and $\iota\colon W\hookrightarrow V$ an
injective $\C$-linear map. Let $\Phi\colon W\times V\to \C$ be a
sesquilinear map. We extend $H$ to a Hermitian form
\[G=\begin{pmatrix}H & \Phi^*\\ \Phi & 0\end{pmatrix}\]
on $V\oplus W$. In other words,
\[G(v'+w',v+w)=H(v',v)+\Phi(w',v)+\overline{\Phi(w,v')},\]
for $v,v'\in V$, $w,w'\in W$. Let
\[W^\perp_\Phi\colonequals \{v\in V\mid
    \Phi(w,v)=0,\forall w\in W\}.\]
Assume that $\Phi|_{W\times \iota W}$ is nondegenerate. Then the projection
$V\oplus W\to V$ induces an isomorphism of Hermitian spaces
\begin{equation}\label{e.tr0}
((\iota W\oplus W)^\perp_G, G|_{(\iota W\oplus W)^\perp_G})\simto (W^\perp_\Phi, H|_{W^\perp_\Phi}),
\end{equation}
which restricts to an isomorphism
\[(V\oplus W)^\perp_G \simto W^\perp_\Phi\cap (W^\perp_\Phi)^\perp_H.\]
In particular, $G$ is nondegenerate if and only if $H|_{W^\perp_\Phi \times
W^\perp_\Phi}$ is nondegenerate.
\end{lemma}

\begin{proof}
For $v\in V$ and $w\in W$, $v+w\in (\iota W\oplus W)^\perp_G$ if and only if
it satisfies the following conditions:
\begin{gather}
\label{e.tr1} 0=G(w',v+w)=\Phi(w',v)\quad \forall w'\in W,\\
\label{e.tr2} 0=G(\iota w'',v+w)=H(\iota w'',v)+\overline{\Phi(w,\iota w'')}\quad \forall w''\in W.
\end{gather}
Condition \eqref{e.tr1} is $v\in W^\perp_\Phi$. For every $v$, there exists
a unique $w\in W$ satisfying \eqref{e.tr2} by the nondegeneracy of
$\Phi|_{W\times \iota W}$. Moreover, for $v\in W^\perp_\Phi$,
$G(v+w,v+w)=H(v,v)$. This finishes the proof that \eqref{e.tr0} is an
isomorphism of Hermitian spaces.

For every subspace $U\subseteq W^\perp_\Phi$, \eqref{e.tr0} restricts to an
isomorphism
\[((\iota W+U)\oplus W)^\perp_G \simto W^\perp_\Phi\cap U^\perp_H.\]
It then suffices to take $U=W^\perp_\Phi$. Indeed, $V=\iota W\oplus
W^\perp_\Phi$ by the nondegeneracy of $\Phi|_{W\times \iota W}$.
\end{proof}

\begin{lemma}\label{l.block}
Let $(V,H)$ be a Hermitian space, and $\iota\colon W\hookrightarrow V$ an
injective $\C$-linear map. Let $\Phi\colon W\times V\to \C$ be a
sesquilinear map. We extend $H$ to a Hermitian form
\[G=\begin{pmatrix}H & \Phi^*\\ \Phi & 0\end{pmatrix}\]
on $V\oplus W$ as in Lemma \ref{l.trivial}. Assume that $H$ is negative
definite on $\iota W$. Let $d_V=\dim(V)$, $d_W=\dim(W)$. Let $w_1,\dots,
w_{d_W}$ be a basis of $W$. Consider the following conditions:
\begin{enumerate}
\item[(a)] $H$ is positive semidefinite on $(\iota W)^\perp_H$.
\item[(a${}_+$)] $H$ is positive definite on $(\iota W)^\perp_H$.
\item[(b)] $H$ is positive semidefinite on $W^\perp_\Phi$.
\item[(b${}_+$)] $H$ is positive definite on $W^\perp_\Phi$.
\item[(c)] $G$ is positive semidefinite on $(\iota W)^\perp_G$ and
    $\Phi|_{W\times \iota W}$ is nondegenerate.
\item[(c${}_+$)] $G$ is positive definite on $(\iota W)^\perp_G$.
\item[(d)] The matrix $M=(\Phi(w_i,\iota w_j))_{1\le i,j\le d_W}$ is
    invertible and for every $v\in V$, we have $H(v,v)\ge 2\Rea
    (\bx_v^*M^{-1}\by_v)$, where $\bx_v=(H(\iota w_i,v))_{1\le i\le d_W}$
    and $\by_v=(\Phi(w_i,v))_{1\le i\le d_W}$ are column vectors.
\item[(d${}_+$)] With the notation of (d), the matrix $M$ is invertible
    and for every $v\in V$, we have $H(v,v)\ge 2\Rea (\bx_v^*M^{-1}\by_v)$
    and equality holds if and only if $v=0$.
\end{enumerate}
Then the following implications hold:
\[\begin{tikzcd}\text{(a${}_+$)}\ar[r,Leftarrow]\ar[d,Rightarrow] & \text{(b${}_+$)}\ar[r,Leftrightarrow]\ar[d,Rightarrow] & \text{(c${}_+$)}\ar[d,Rightarrow]\ar[r,Leftrightarrow] & \text{(d${}_+$)}\ar[d,Rightarrow]\\
\text{(a)}\ar[r,Leftarrow] & \text{(b)}\ar[r,Leftrightarrow] & \text{(c)}\ar[r,Leftrightarrow] & \text{(d)}.
\end{tikzcd}
\]
\end{lemma}

\begin{proof}
(a${}_+$)$\implies$(a) and (b${}_+$)$\implies$(b). Trivial.

(b)$\implies$(a). This follows from Lemma \ref{l.1}, because
$\dim(W^\perp_\Phi)\ge d_V-d_W$.

(b${}_+$)$\implies$(a${}_+$). Similarly, this follows from Lemma \ref{l.1}.

(b)$\implies$(c). Since $H$ is positive semidefinite on $W^\perp_\Phi$ and
negative definite on $\iota W$, we have $\iota W\cap W^\perp_\Phi=0$. In
other words, $\Phi|_{W\times \iota W}$ is nondegenerate. Since $G$ is
positive semidefinite on $W^\perp_\Phi\oplus W$ of dimension $d_V$, $G$ is
positive semidefinite on $(\iota W)^\perp_G$ by Lemma \ref{l.1}.

(b${}_+$)$\implies$(c${}_+$). We have already seen that $\Phi|_{W\times \iota 
W}$ is nondegenerate, which is equivalent to $V=\iota W\oplus W^\perp_\Phi$. 
By (b${}_+$) and Lemma \ref{l.trivial}, $G$ is nondegenerate. By (b${}_+$), 
$G$ is positive definite on $W^\perp_\Phi$. Consider the decomposition 
\[V\oplus W=W^\perp_\Phi\oplus (W^\perp_\Phi)^\perp_G.\]
We have $\dim(W^\perp_\Phi)=d_V-d_W$ and thus
$\dim((W^\perp_\Phi)^\perp_G)=2d_W$. By definition, $W\subseteq
(W^\perp_\Phi)^\perp_G$. Since $G$ is totally isotropic on $W$, the
signature of $G$ on $(W^\perp_\Phi)^\perp_G$ is $(d_W,d_W)$. Therefore, the
signature of $G$ on $V\oplus W$ is $(d_V,d_W)$, which is equivalent to
(c${}_+$) by Lemma \ref{l.1}.

(c)$\implies$(d). Consider the block matrix
\[B_v=\begin{pmatrix}
H(v,v)&\bx_v^*&\by_v^*\\
\bx_v&N&M^*\\
\by_v&M&0
\end{pmatrix},
\]
where $N=(H(\iota w_i,\iota w_j))_{1\le i,j\le d_W}$. An elementary
computation gives
\begin{gather*}
\begin{pmatrix}
N & M^*\\ M & 0
\end{pmatrix}^{-1}=\begin{pmatrix}
  0&M^{-1}\\(M^*)^{-1} & -(M^{*})^{-1}NM^{-1}
\end{pmatrix},\\
\det(B_v)=(-1)^{d_W}\lvert
\det(M)\rvert^2(H(v,v)-2\Rea(\bx_v^*M^{-1}\by_v)+(M^{-1}\by_v)^*N(M^{-1}\by_v)).
\end{gather*}
Since $N$ is negative definite, the negative index of inertia of $B_v$ is at 
least $d_W$. We claim that the negative index of inertia of $B_v$ is exactly 
$d_W$. In the case $v\in \iota W$, $B_v$ is totally isotropic on a subspace 
of dimension $d_W+1$ and the claim follows. In the case $v\notin \iota W$, 
$B_v$ is the matrix of the restriction of $G$ to $\C v\oplus \iota W \oplus 
W$ and the claim follows from (c). It follows from the claim that 
$(-1)^{d_W}\det(B_v)\ge 0$. Thus
\begin{equation}\label{e.ineqRe}
H(v,v)\ge 2\Rea(\bx_v^*M^{-1}\by_v)-(M^{-1}\by_v)^*N(M^{-1}\by_v)\ge 2\Rea(\bx_v^*M^{-1}\by_v).
\end{equation}

(c${}_+$)$\implies$(d${}_+$). Note first that, by (c${}_+$), $W\cap (\iota
W)^\perp_G=0$, which means that $\Phi|_{W\times \iota W}$ is nondegenerate.
Let $v\in V$ be nonzero. If $v\notin\iota W$, then, by (c${}_+$), the
signature of $B_v$ is $(d_W+1,d_W)$, which implies $(-1)^{d_W}\det(B_v)> 0$
and hence the first inequality in \eqref{e.ineqRe} is strict. If $v\in \iota
W$ is nonzero, then $\by_v\neq 0$ and the second inequality in
\eqref{e.ineqRe} is strict.

(d)$\implies$(b). For $v\in W^\perp_{\Phi}$, we have $\by_v=0$ and
$H(v,v)\ge 2\Rea(\bx_v^*M^{-1}\by_v) =0$ by (d).

(d${}_+$)$\implies$(b${}_+$). Similarly, for nonzero $v\in W^\perp_{\Phi}$,
we have $H(v,v)> 2\Rea(\bx_v^*M^{-1}\by_v) =0$ by (d${}_+$).
\end{proof}

\begin{remark}
The conditions (d) and (d${}_+$) above are analogues of \cite[(2.2)]{RT}. In
fact, we do not need them to prove the equivalences (b)$\iff$(c) and
(b${}_+$)$\iff$(c${}_+$). Indeed, (c)$\implies$(b) and
(c${}_+$)$\implies$(b${}_+$) follow directly from Lemma \ref{l.trivial} and
$(\iota W)^\perp_G\supseteq (\iota W\oplus W)^\perp_G$.
\end{remark}

\subsection{A differential criterion}

\begin{lemma}\label{l.der0}
Let $V$ be a $\C$-vector space and let $H_t$, $t\in I$ be a family of 
positive semidefinite Hermitian forms on $V$, where $I\subseteq \R$ is an 
open interval containing $0$. Assume that $H'_0\colonequals \frac{d 
H_t}{dt}|_{t=0}$ exists. Then any $v_0\in V$ satisfying $H_0(v_0,v_0)=0$ also 
satisfies $H'_0(v_0,v_0)=0$. In particular, if $H'_0$ is definite, then $H_0$ 
is positive definite. 
\end{lemma}

\begin{proof}
Consider the function $f(t)=H_t(v_0,v_0)\ge 0$. The assumption $f(0)=0$ 
implies $0=f'(0)=H'_0(v_0,v_0)$.
\end{proof}

We need the following generalization of Lemma \ref{l.der0}.

\begin{lemma}\label{l.derivative}
Let $V$ be a $\C$-vector space, $\iota\colon W\hookrightarrow V$ an
injective $\C$-linear map. Let $H_t$, $t\in I$ be a family of Hermitian
forms on $V$, where $I\subseteq \R$ is an open interval containing $0$.
Assume that for all $t\in I$, $\iota W\subseteq V$ is a nondegenerate
subspace with respect to $H_t$, and that $H_t$ is positive semidefinite on
$(\iota W)^\perp_{H_t}$. Assume moreover that $H'_0\colonequals \frac{d
H_t}{dt}|_{t=0}$ exists. Then any $v_0\in (\iota W)^\perp_{H_0}$ satisfying
$H_0(v_0,v_0)=0$ also satisfies $H'_0(v_0,v_0)=0$. In particular, if $H'_0$
is definite on $(\iota W)^\perp_{H_0}$, then $H_0$ is positive definite on
$(\iota W)^\perp_{H_0}$.
\end{lemma}

\begin{proof}
Let $v_0\in (\iota W)^\perp_{H_0}$ such that $H_0(v_0,v_0)=0$. Let $p_t$ be 
the projection of $V=\iota W\oplus (\iota W)^\perp_{H_t}$ onto $(\iota 
W)^\perp_{H_t}$ and let $v_t=p_t(v_0)$. Since $v_t-v_0\in \iota W$ for all 
$t\in I$, we have $v'_0\in \iota W$. Consider the function 
$f(t)=H_t(v_t,v_t)\ge 0$. The assumption $f(0)=0$ implies 
\[0=f'(0)=H'_0(v_0,v_0)+2\re H_0(v'_0,v_0)=H'_0(v_0,v_0).\]
\end{proof}

Under suitable conditions, we can iterate Lemma \ref{l.derivative} as
follows.

\begin{theorem}\label{p.iterate}
Let $V$ be a $\C$-vector space, $\iota\colon W\hookrightarrow V$ an injective 
$\C$-linear map. Let $k\ge 1$ be an integer and $I\subseteq \R$ an open 
interval containing $0$. Let $H_t$, $t\in I$ be a family of Hermitian forms 
on $V$ such that $H^{(i)}_t\colonequals \frac{d^iH_t}{dt^i}$ exists for 
$i=k-1$ and all $t\in I$ and $H^{(k)}_0$ exists. Assume that there exist a 
$\C$-linear map $f\colon W\to V$ and constants 
$\kappa_0,\dots,\kappa_{k-1}\in \C^\times $ such that 
\[
H^{(i+1)}_0(f w, v)=\kappa_i H_0^{(i)}(\iota w,v)
\]
for all $w\in W$, $v\in V$, $0\le i\le k-1$. We make the following
assumptions:
\begin{enumerate}
\item[(a)] $H^{(i)}_t$ and $H^{(k)}_0$ are nondegenerate on $\iota W$ for
    all $t\in I$ and $0\le i\le k-1$.
\item[(b)] $H^{(i)}_t$ is positive semidefinite on $(\iota
    W)^\perp_{H^{(i)}_t}$ for all $t\in I$ and $0\le i \le k-1$.
\item[(c)] $H^{(i+1)}_0$ is semidefinite on $(\iota W)^\perp_{H^{(i)}_0}=(f
    W)^\perp_{H^{(i+1)}_0}$ for all $0\le i\le k-1$.
\end{enumerate}
Consider the following conditions:
\begin{enumerate}
\item[(z${}^{(i)}$)] $H^{(i)}_0$ is positive definite on $(\iota
    W)^\perp_{H_0^{(i)}}$.
\item[(z${}^{\mathrm{w}}$)] $H_0(v,v)>0$ for every $v\in (\iota
    W)^\perp_{H_0}\backslash fW$.
\end{enumerate}
Then the following implications hold:
\[\text{(z${}^{(k)}$)}\implies \text{(z${}^{(k-1)}$)}\implies \dots \implies \text{(z${}^{(2)}$)}\implies \text{(z${}^{(1)}$)} \implies \text{(z${}^{\mathrm{w}}$)}.\]
\end{theorem}

\begin{proof}
Let us first show $\text{(z${}^{(1)}$)} \implies
\text{(z${}^{\mathrm{w}}$)}$. Let $v\in (\iota W)^\perp_{H_0}$ such that
$H_0(v,v)=0$. By (a), (b), and Lemma \ref{l.derivative}, $H'_0(v,v)=0$. Since
$H'_0$ is semidefinite on $(\iota W)^\perp_{H_0}$ by (c), it follows that
$H'_0(v',v)=0$ for all $v'\in (\iota W)^\perp_{H_0}$. Let $\Phi_0\colon
W\times V\to \C$ be given by $\Phi_0(w,v)=H_0(\iota w,v)$ and let $G_0$ be
the Hermitian form on $V\oplus W$ given by
\[G_0=\begin{pmatrix}
  H'_0& \Phi_0^*\\
  \Phi_0 & 0
\end{pmatrix}.
\]
By Lemma \ref{l.trivial}, there exists a (unique) $w\in W$ such that $v+w\in
(V\oplus W)^\perp_{G_0}$. In particular, for every $v'\in V$,
\[0=G_0(v',v+w)=H'_0(v',v)+H_0(v',\iota w)=H'_0(v',v)+\overline{\kappa_0^{-1}}H'_0(v',f w)=H'_0(v',v+\kappa_0^{-1}f w).\]
By (a) and (z${}^{(1)}$), $H'_0$ is nondegenerate. It follows that
$v+\kappa_0^{-1}f w=0$. Thus $v\in fW$, as desired.

It remains to show $\text{(z${}^{(i+1)}$)}\implies\text{(z${}^{(i)}$)}$ for
$1\le i\le k-1$. Let $v\in (\iota
    W)^\perp_{H_0^{(i)}}$ such that $H_0^{(i)}(v,v)=0$. By $\text{(z${}^{(1)}$)}
\implies \text{(z${}^{\mathrm{w}}$)}$ applied to $H^{(i)}_t$, we have $v=fw$
for some $w\in W$. Then
\[0=H_0^{(i)}(\iota w',fw)=\overline{\kappa_{i-1}}H_0^{(i-1)}(\iota w',\iota w)\]
for all $w'\in W$. Therefore, $w=0$ by (a).
\end{proof}

The condition (z${}^{(k)}$) can be checked using the following lemma.

\begin{lemma}\label{l.init}
Let $V$ be a $\C$-vector space, $\iota\colon W\hookrightarrow V$ an injective 
$\C$-linear map. Let $k\ge 1$ be an integer and $I\subseteq \R$ an open 
interval containing $0$. Let $H_t$, $t\in I$ be a family of Hermitian forms 
on $V$ such that $H^{(i)}_t\colonequals \frac{d^iH_t}{dt^i}$ exists for $i=k$ 
and all $t\in I$ and $H^{(k+1)}_0$ exists. Assume that there exist a 
$\C$-linear map $f\colon W\to V$, a $\C$-linear subspace $U\subseteq V$, and 
constants $\kappa_0,\dots,\kappa_{k}\in \R^\times $ such that $V=U+fW$ and
\[
H^{(i+1)}_0(f w, v)=\kappa_i H_0^{(i)}(\iota w,v)
\]
for all $w\in W$, $v\in V$, $k-1\le i\le k$. We make the following 
assumptions: 
\begin{enumerate}
\item[(a$'$)] $H_t^{(k)}$ is nondegenerate on $\iota W$ for all $t\in I$.
\item[(a$''$)] $\kappa_{k}\kappa_{k-1}H_0^{(k-1)}$ is negative definite on
    $\iota W$.
\item[(b$'$)] $H^{(k)}_t$ is positive semidefinite on $(\iota
    W)^\perp_{H^{(k)}_t}$ for all $t\in I$.
\item[(d)] $H_0^{(k)}$ is positive definite on $(\iota
    W)^\perp_{H_0^{(k)}}\cap U$.
\item[(e)] $H_0^{(k+1)}$ is positive semidefinite on $U$.
\end{enumerate}
Then (z${}^{(k)}$) holds.
\end{lemma}

\begin{proof}
By (d), it suffices to show $H_0^{(k)}(v,v)>0$ for all $v\in (\iota
W)^\perp_{H_0^{(k)}}\backslash U$. Let $v\in (\iota
W)^\perp_{H_0^{(k)}}\backslash U$ such that $H_0^{(k)}(v,v)=0$. We have
$v=u+fw$ for $u\in U$ and $0\neq w\in W$. Then
\[0=H_0^{(k)}(\iota w,v)=H_0^{(k)}(\iota w,u+fw)=H_0^{(k)}(\iota w,u)+H_0^{(k)}(\iota w,fw).\]
Thus
\begin{equation}\label{e.iterateneg}
H_0^{(k)}(\iota w,u)=-H_0^{(k)}(\iota w,fw).
\end{equation}
By (a$'$), (b$'$), and Lemma \ref{l.derivative},
\begin{align*}
0 &=H_0^{(k+1)}(v,v)&&\\
&=H_0^{(k+1)}(u,u)+2\re H_0^{(k+1)}(fw,u)+H_0^{(k+1)}(fw,fw) &&\\
&=H_0^{(k+1)}(u,u)+2\kappa_{k}\re H_0^{(k)}(\iota w,u)+\kappa_{k}H_0^{(k)}(\iota w,fw) &&\\
&=H_0^{(k+1)}(u,u)-\kappa_{k}H_0^{(k)}(\iota w,fw)&&\text{by \eqref{e.iterateneg}}\\
&=H_0^{(k+1)}(u,u)-\kappa_{k}\kappa_{k-1}H_0^{(k-1)}(\iota w,\iota w)>0&&\text{by (e) and (a$''$).}
\end{align*}
Contradiction.
\end{proof}

\begin{cor}\label{c.iterate}
Let $V$ be a $\C$-vector space, $\iota\colon W\hookrightarrow V$ an injective 
$\C$-linear map. Let $k\ge 0$ be an integer and $I\subseteq \R$ an open 
interval containing $0$. Let $H_t$, $t\in I$ be a family of Hermitian forms 
on $V$ such that $H^{(i)}_t\colonequals \frac{d^iH_t}{dt^i}$ exists for $i=k$ 
and all $t\in I$ and $H^{(k+1)}_0$ exists. Assume that there exist a 
$\C$-linear map $f\colon W\to V$, a $\C$-linear subspace $U\subseteq V$ and 
constants $\kappa_0,\dots,\kappa_{k}\in \R_{>0}$ such that $V=U\oplus fW$, 
$\iota W\subseteq U$, and 
\begin{equation}\label{e.f}
H^{(i+1)}_0(f w, v)=\kappa_i H_0^{(i)}(\iota w,v)
\end{equation}
for all $w\in W$, $v\in V$, $0\le i\le k$. We make the following assumptions:
\begin{enumerate}[(A)]
\item $(H_t^{(k)}|_{U\times U},\iota)\in \fHR(U,W)$ for all $t\in I$;
\item $(H_t^{(i)},\iota)\in \fHRvw(V,W)$ for all $t\in I$ and $0\le i\le 
    k$. 
\item $(H'_0,H_0(\iota-,-))\in \HRw(V,W)$  if $k\ge 1$. 
\item[(e)] $H_0^{(k+1)}$ is positive semidefinite on $U$.
\end{enumerate}
Then $(H_0|_{U\times U},\iota)\in \fHR(U,W)$. Moreover, $(H^{(i)}_0,\iota)\in
\fHR(V,W)$ for all $1\le i\le k$.
\end{cor}

In our applications, we have in fact $H_0^{(k+1)}|_{U\times U}=0$.

\begin{proof}
For $k=0$, it suffices to apply (A). For $k\ge 1$, we apply Theorem 
\ref{p.iterate} and Lemma \ref{l.init}. Indeed, by (A), (d) holds and 
$H_t^{(k)}$ is negative definite on $\iota W$. By (B), $H_t^{(i)}$ is 
negative semidefinite on $\iota W$ for all $0\le i\le k-1$. We show by 
descending induction that $H_t^{(i)}$ is negative definite on $\iota W $ for 
all $0\le i\le k$. We have already seen the case $i=k$. For $0\le i<k$ such 
that $H_t^{(i+1)}$ is negative definite on $\iota W $, $H_t^{(i)}$ is 
negative definite on $\iota W $ by Lemma \ref{l.der0}. In particular, we have 
(a), (a$'$), (a$''$). Then (B) implies (b), (b$'$), and the negative index of 
inertia of $H_t^{(i)}$ is $d_W=\dim(W)$ by Lemma \ref{l.1}. For $1\le i\le 
k-1$, $H_0^{(i+1)}(f-,f-)=\kappa_i \kappa_{i-1}H_0^{(i-1)}(\iota-,\iota-)$ is 
negative definite, which implies the case $i\ge 1$ of (c) by Lemma \ref{l.1}. 
Moreover, since $H'_0(f-,-)=\kappa_0 H_0(\iota-,-)$ is left nondegenerate, 
(C) implies the case $i=0$ of (c). By Theorem \ref{p.iterate} and Lemma 
\ref{l.init}, (z${}^{\mathrm{w}}$) holds, which implies $(H_0|_{U\times 
U},\iota)\in \fHR(U,W)$. Moreover, for $1\le i\le k$, (z$^{(i)}$) holds, 
which implies $(H^{(i)}_0,\iota)\in \fHR(V,W)$. 
\end{proof}

\begin{remark}\label{r.Cp}
The proof shows that assumption (C) can be replaced by either of the 
following: 
\begin{enumerate}
\item[(C$'$)] $(H'_0,f)\in \fHRw(V,W)$ if $k \ge 1$; 
\item[(C$''$)] $H'_0(f-,f-)$ is negative definite if $k\ge 1$. 
\end{enumerate}
\end{remark}

\section{Schur classes and Schubert classes}\label{s.5}

In this section, we prove the Hodge--Riemann property for Schur classes and 
Schubert classes. We first discuss the case of top Chern classes (Theorem 
\ref{t.top}), which is a consequence of the Bloch--Gieseker theorem 
(Proposition \ref{t.BG}). Using the linear algebra machine, we then prove a 
general cone theorem (Theorem \ref{t.cone0}). This implies the Hodge--Riemann 
property for products of derived Schur classes and derived Schubert classes 
(Corollaries \ref{c.final} and \ref{c.Schubertprod}). 

For $\gamma\in H^{k,k}(X,\R)$, we let $H^{p,q}(X)_{\gamma\hprim}$ denote the
kernel of
\[-\wedge \gamma\colon H^{p,q}(X)\to H^{p+k,q+k}(X).\]

\subsection{Top Chern classes}\label{s.5.1}

In this subsection we do not assume any vanishing of Hodge numbers. For an 
$\R$-twisted vector bundle $E$ of rank $e$, we call $c_e(E)$ its \emph{top} 
Chern class. The following theorem on products of top Chern classes will be 
deduced from Corollary \ref{c.multiBG}, which follows from the 
Bloch--Gieseker theorem (Proposition \ref{t.BG}). 

\begin{theorem}\label{t.top}
Let $E_1,\dots,E_r$ be ample $\R$-twisted vector bundles of ranks 
$e_1,\dots,e_r$, respectively, on a smooth projective variety $X$ of 
dimension $d=p+q+e_1+\dots+e_r+l$ with $l\ge 0$. Assume that $E_1,\dots,E_r$ 
have the same $\R$-twist modulo $\NS(X)_\Q$. Let $h_1,\dots,h_l,h\in 
\NS(X)_\R$ be ample classes. Let $\gamma=c_{e_1}(E_1)\dotsm 
c_{e_r}(E_r)h_1\dotsm h_l$. 
\begin{enumerate}
\item (Hard Lefschetz) The map
\[-\wedge \gamma\colon H^{p,q}(X)\to H^{d-q,d-p}(X)\]
is a bijection.

\item (Lefschetz decomposition) We have
\[H^{p,q}(X)=H^{p-1,q-1}(X)\wedge h\oplus H^{p,q}(X)_{\gamma h\hprim}.\]

\item (Hodge--Riemann relations) $\langle -,-\rangle_{\gamma}$ is positive 
    definite on $H^{p,q}(X)_{\gamma h\hprim}$. 
\end{enumerate}
\end{theorem}

\begin{remark}
Taking $E_1=\dots=E_r=0$ in Theorem \ref{t.top}, we recover the algebraic 
case of the mixed Hodge--Riemann relations of Gromov \cite[Theorem 
2.4.B]{Gromov} and Dinh--Nguy\^en \cite[Theorem A]{DN} (see also Timorin 
\cite{Timorin} and Cattani \cite{Cattani}).  
\end{remark}

Before giving the proof, we introduce some notation. 

\begin{defn}
Let $p+q+k=d$. We write $H^{p,q}$ for $H^{p,q}(X)$. Consider the map
\[H^{k,k}(X,\R)\times H^{k+1,k+1}(X,\R)\to \Herm(H^{p,q})\times \Sesq(H^{p-1,q-1},H^{p,q})\]
carrying $(\gamma,\delta)$ to $(H,\Phi)=(\langle-,-\rangle_\gamma, 
(-,-)_\delta)$. In other words,
\[H(\alpha,\beta)=i^{q-p}(-1)^{\frac{(p+q)(p+q+1)}{2}}\int_X \alpha\wedge \bar \beta\wedge\gamma,\quad \Phi(\alpha',\beta)=\int_X \alpha'\wedge \bar\beta\wedge \delta\]
We say that $(\gamma,\delta)$ is a \emph{pre-Hodge--Riemann pair} on 
$H^{p,q}(X)$ if $(H,\Phi)$ is a pre-Hodge--Riemann pair. We let 
$\HR_{p,q}(X)$ denote the set of pre-Hodge--Riemann pairs on $H^{p,q}(X)$ and 
let $\overline{\HR}_{p,q}(X)$ denote the closure of $\HR_{p,q}(X)$ in 
$H^{k,k}(X,\R)\times H^{k+1,k+1}(X,\R)$. We let $\HRw_{p,q}(X)\subseteq 
H^{k,k}(X,\R)\times H^{k+1,k+1}(X,\R)$ denote the inverse image of 
$\HRw(H^{p,q},H^{p-1,q-1})$. 
\end{defn}

We will see in Corollary \ref{c.cpenult} that for any ample $\R$-twisted 
vector bundle $E$ of rank $e\ge 1$ on a smooth projective variety $X$ of 
dimension $p+q+e-1$, we have $(c_{e-1}(E),c_e(E))\in \HR_{p,q}(X)$. 

We have $\HR_{p,q}(X)\subseteq \overline{\HR}_{p,q}(X)\subseteq
\HRw_{p,q}(X)$. By definition, $(\gamma,\delta)$ is a pre-Hodge--Riemann pair
on $H^{p,q}(X)$ if and only if $\dim
H^{p,q}(X)_{\delta\hprim}=h^{p,q}-h^{p-1,q-1}$ and $\langle-,-\rangle_\gamma$
is positive definite on $H^{p,q}(X)_{\delta\hprim}$. Here $h^{p,q}=\dim
H^{p,q}(X)$.

\begin{remark}\label{r.indep}
Let $C\subseteq H^{1,1}(X,\R)$ be a connected subset. Assume that $\langle 
h'-,h'-\rangle_{\gamma}=-\langle -,-\rangle_{\gamma h'^2}$ is nondegenerate 
on $H^{p-1,q-1}(X)$ for all $h'\in C$. Then the signature $(r,s)$ of $\langle 
h'-,h'-\rangle_{\gamma}$ on $H^{p-1,q-1}(X)$ is independent of $h'\in C$. In 
this case, the assertion $(\gamma,\gamma h')\in \HRw_{p,q}(X)$ (resp.\ $\in 
\HR_{p,q}(X)$) is equivalent to the assertion that $\langle-,-\rangle_\gamma$ 
on $H^{p,q}(X)$ has negative index of inertia~$s$ (resp.\ and is 
nondegenerate), which does not depend on the choice of $h'\in C$. 
\end{remark}

\begin{lemma}\label{l.gammader}
Let $X$ be a smooth projective variety of dimension $d=p+q+k+1$ and let 
$\gamma\in H^{k,k}(X,\R)$. Let $S\subseteq C\subseteq H^{1,1}(X,\R)$ be 
subsets with $C$ connected and open in $\spa_\R C$ and contained in the 
closure of $\R_{>0}\cdot S$.  Assume the following. 
\begin{enumerate}
\item $\langle -,-\rangle_{\gamma hh'^2}$ is nondegenerate on
    $H^{p-1,q-1}(X)$ for all $h,h'\in C$.
\item For every $h\in S$, there exists $h'\in C$ such that $(\gamma
    h,\gamma h h')\in \HR_{p,q}(X)$.
\end{enumerate}
Then $(\gamma h,\gamma h h')\in \HR_{p,q}(X)$ for all $h,h'\in C$. 
\end{lemma}

\begin{proof}
By (a) and Remark \ref{r.indep}, the assertion $(\gamma h,\gamma h h')\in 
\HR_{p,q}(X)$ does not depend on $h'\in C$. Thus, by (b), $(\gamma h,\gamma h 
h')\in \HR_{p,q}(X)$ for all $h\in \R_{>0}\cdot S$ and $h'\in C$. By 
continuity, $(\gamma h,\gamma h h')\in \overline{\HR}_{p,q}(X)$ for all 
$h,h'\in C$. 

To prove the lemma, we may assume $h'\in S$. We apply Lemma 
\ref{l.derivative} to $W=H^{p-1,q-1}(X)$, $V=H^{p,q}(X)$, $\iota=-\wedge h'$, 
and $H_t=\langle -,-\rangle_{\gamma(h+th')}$. Let $I\subseteq \R$ be an open 
interval containing $0$ such that $h+th'\in C$ for all $t\in I$. By Lemma 
\ref{l.cont1}, $H_t$ is positive semidefinite on $(\iota W)_{H_t}^\perp$. 
Moreover, $H'_0=\langle -,-\rangle_{\gamma h'}$ is positive definite on 
$(\iota W)_{H_0}^\perp$ since $(\gamma h', \gamma h'h)\in \HR_{p,q}(X)$ by 
the preceding paragraph. By Lemma \ref{l.derivative}, $H_0=\langle 
-,-\rangle_{\gamma h}$ is positive definite on $(\iota W)_{H_0}^\perp$. 
\end{proof}

\begin{prop}\label{p.HL}
Let $X$ be a smooth projective variety of dimension $d=p+q+k+1$ with $k\ge 
0$. Let $\gamma\in H^{k,k}(X,\R)$, $\delta\in H^{k+1,k+1}(X,\R)$. 
\begin{enumerate}[(i)]
\item Assume that for every irreducible smooth ample divisor $Z$ on $X$, 
    $\langle -,-\rangle_{\gamma|_{Z}}$ on $H^{p,q}(Z)$ is positive definite 
    (resp.\ positive semidefinite, resp.\ $(\gamma|_{Z},\delta|_{Z})\in 
    \HRw_{p,q}(Z)$). Then $\langle-,-\rangle_{\gamma h}$ on $H^{p,q}(X)$ is 
    positive definite 
(resp.\ positive semidefinite, resp.\ $(\gamma h,\delta h)\in 
\HRw_{p,q}(X)$) for all  $h\in \Amp(X)$. 

\item Assume that the following conditions hold.
\begin{enumerate}
\item For all $h,h'\in \Amp(X)$, $\langle -,-\rangle_{\gamma hh'^2}$ is 
    nondegenerate on $H^{p-1,q-1}(X)$. 
\item For every irreducible smooth ample divisor $Z$ on $X$, 
    $(\gamma|_{Z},(\gamma h')|_{Z})\in \HR_{p,q}(Z)$ for all $h'\in 
    \Amp(X)$. 
\end{enumerate}
Then $(\gamma h,\gamma hh')\in \HR_{p,q}(X)$ for all  $h,h'\in \Amp(X)$. 
\end{enumerate} 
\end{prop}

\begin{proof}
We first reduce to the case $h\in \Amp(X)\cap \NS(X)_\Q$. This follows by 
continuity for the semidefinite case and the case of $\HRw$. For the case of 
$\HR$, we apply Lemma \ref{l.gammader} to $S=\Amp(X)\cap \NS(X)_\Q$ and 
$C=\Amp(X)$. For the definite case, it suffices to apply Lemma 
\ref{l.derivative} to $W=0$ and $\langle-,- \rangle_{\gamma(h+th')}$ for 
$h'\in S$. 

Assume $h\in \Amp(X)\cap \NS(X)_\Q$. By Bertini's theorem, up to replacing 
$h$ by a positive multiple, we may assume that $h$ is very ample and 
represented by a smooth closed subvariety $Z$ of $X$ of codimension $1$. Then 
\[
\langle
-,-\rangle_{\gamma h}=\langle
\iota^*-,\iota^*-\rangle_{\iota^*\gamma},\quad  (
-,-)_{\delta  h}=(
\iota^*-,\iota^*-)_{\iota^*\delta},
\]
where $\iota\colon Z\hookrightarrow X$ is the embedding. Since $p+q\le d-1$, 
$\iota^*\colon H^{p-i,q-i}(X)\to H^{p-i,q-i}(Z)$ is an injection for $i=0$ 
and a bijection for $i=1$ by the Lefschetz hyperplane theorem. The definite 
and semidefinite cases are then clear and for the other two cases we conclude 
by Lemmas \ref{l.func0} and \ref{l.func}. 
\end{proof}

\begin{proof}[Proof of Theorem \ref{t.top}]
It suffices to prove the following assertions.
\begin{enumerate}
\item $\langle -,-\rangle_{\gamma}$ is nondegenerate on $H^{p,q}(X)$. 
\item $\langle h-,h-\rangle_{\gamma}$ is nondegenerate on $H^{p-1,q-1}(X)$. 
\item $(\gamma,\gamma h)\in \HR_{p,q}(X)$.
\end{enumerate} 
We proceed by induction on $p$ and $q$. The case $p<0$ or $q<0$ is trivial. 
Assume $p,q\ge 0$. (a) follows from (b) and (c). Since $\langle 
h-,h-\rangle_{\gamma}=-\langle -,-\rangle_{\gamma h^2}$, (b) follows from (a) 
for $(p-1,q-1)$, which holds by the induction hypothesis. To prove (c), we 
may assume $h\in \Amp(X)\cap \NS(X)$ by Remark \ref{r.indep} and the 
induction hypothesis. For this case, we proceed by induction on $l$.  

Case $l=0$. In this case, (a) holds by Corollary \ref{c.multiBG}. By Remark 
\ref{r.indep} and the induction hypothesis on $(p,q)$, to prove (c) in this 
case, it suffices to show that the signature of $\langle -,-\rangle_{\gamma}$ 
on $H^{p,q}(X)$ is $(a_0,a_1)$, where $a_j=\sum_{i\ge 
0}(b_{j+2i}-b_{j+2i+1})$, $b_j=\dim H^{p-j,q-j}(X)$. We have 
$c_{e_i}(E_i\langle th \rangle)=t^{e_i}h^{e_i}+O(t^{e_i-1})$ for $t\gg 0$. 
Thus $\gamma_t\colonequals \frac{1}{(1+t)^e}c_{e_1}(E_1\langle th\rangle 
)\dotsm c_{e_r}(E_r\langle th\rangle)\to h^e$ as $t\to \infty$, where 
$e=e_1+\dots+e_r$. By continuity, we are thus reduced to proving the result 
for $h^e$, which is classical. 

Case $l>0$. We apply Proposition  \ref{p.HL}(ii) to 
$\gamma'=c_{e_1}(E_1)\dotsm c_{e_r}(E_r)h_1\dotsm h_{l-1}$. Condition (a) 
holds by the induction hypothesis on $(p,q)$. Condition (b) holds by the 
induction hypothesis on $l$. By Proposition \ref{p.HL}(ii), $(\gamma,\gamma 
h)\in \HR_{p,q}(X)$. 
\end{proof} 

\begin{cor}\label{c.cpenultm}
Let $E_1,\dots,E_r$ be ample $\R$-twisted vector bundles of ranks 
$e_1,\dots,e_r$, respectively, on a smooth projective variety $X$ of 
dimension $d=p+q+e_1+\dots+e_r-1$. Assume that $E_1,\dots,E_r$ have the same 
$\R$-twist modulo $\NS(X)_\Q$. Let 
\[\gamma=\sum_{i=1}^r a_ic_{e_i-1}(E_i)\prod_{\substack{j=1\\j\neq i}}^r 
c_{e_j}(E_j),
\] 
where $a_1,\dots,a_r\in \R_{>0}$. Assume that $\{a_i-a_j\mid 1\le i,j\le r\}$ 
spans a $\Q$-vector space of dimension $\le 1$. Then 
$(\gamma,c_{e_1}(E_1)\dotsm c_{e_r}(E_r))\in \HR_{p,q}(X)$. 
\end{cor}

Note that the assumption on $\{a_i-a_j\mid 1\le i,j\le r\}$ is automatic if 
$r\le 2$. 

\begin{proof}
We may assume $a_i-a_j\in \Q$ for all $1\le i,j\le r$. Let us first prove 
that the pairing $(-,-)_{c_{e_1}(E_1)\dotsm c_{e_r}(E_r)}$ on 
$H^{p-1,q-1}(X)\times H^{p,q}(X)$ is left nondegenerate and 
$\langle-,-\rangle_{\gamma}$ is nondegenerate on 
$H^{p,q}(X)_{c_{e_1}(E_1)\dotsm c_{e_r}(E_r)\hprim}$. The first assertion 
follows from the fact that $\langle -,-\rangle_{c_{e_1}(E_1)\dotsm 
c_{e_r}(E_r)h}$ is nondegenerate on $H^{p-1,q-1}(X)$, which is a special case 
of Theorem \ref{t.top}. Here $h\in \NS(X)$ is an ample class.  For the second 
assertion, we apply the K\"unneth formula $H^{p,q}(X\times \PP^1)= 
\pi^*H^{p,q}(X)\oplus \xi\wedge \pi^*H^{p-1,q-1}(X)$, where $\pi\colon 
X\times \PP^1\to X$ is the projection and $\xi$ is the pullback of 
$c_1(\cO_{\PP^1}(1))$ to $X\times \PP^1$. Let $F_i=E_i\boxtimes 
\cO_{\PP^1}(a_i)$ on $X\times \PP^1$. Then $F_1,\dots,F_r$ have the same 
$\R$-twist modulo $N^1(X\times \PP^1)_\Q$. Moreover, 
$c_{e_i}(F_i)=\pi^*c_{e_i}(E_i) + a_i\xi\wedge \pi^*c_{e_i-1}(E_i)$, so that 
$\langle-,-\rangle_{c_{e_1}(F_1)\dotsm c_{e_r}(F_r)}$ has the form 
\[G=\begin{pmatrix}
H&\Phi^*\\\Phi& 0
\end{pmatrix},\]
where
\begin{gather*}
H=\langle-,-\rangle_{\gamma}\in \Herm(H^{p,q}(X)),\quad \Phi=\epsilon (-,-)_{c_{e_1}(E_1)\dotsm c_{e_r}(E_r)}\in \Sesq(H^{p-1,q-1}(X),H^{p,q}(X))
\end{gather*}
for some $\epsilon$ satisfying $\epsilon^4=1$. By Theorem \ref{t.top}, $G$ 
and $\Phi|_{H^{p-1,q-1}(X)\times hH^{p-1,q-1}(X)}$ are nondegenerate. Thus, 
by Lemma \ref{l.trivial}, $H$ is nondegenerate on 
$H^{p,q}(X)_{c_{e_1}(E_1)\dotsm c_{e_r}(E_r)\hprim}$. 

For all $t\ge 0$, by the preceding paragraph applied to the ample 
$\R$-twisted vector bundles $E_1\langle th\rangle,\dots, E_r\langle 
th\rangle$, the pairing $(-,-)_{\delta_t}$ on $H^{p-1,q-1}(X)\times 
H^{p,q}(X)$ is left nondegenerate and $\langle-,-\rangle_{\gamma_t}$ is 
nondegenerate on $H^{p,q}(X)_{\delta_t\hprim}$, where 
\[\gamma_t=\sum_{i=1}^r a_ic_{e_i-1}(E_i\langle th\rangle)\prod_{\substack{j=1\\j\neq i}}^r 
c_{e_j}(E_j\langle th\rangle),\quad \delta_t=c_{e_1}(E_1\langle 
th\rangle)\dotsm c_{e_r}(E_r\langle th\rangle).
\] 
Moreover, 
\[\left(\frac{\gamma_t}{(1+t)^{e_1+\dotsm+e_r-1}},\frac{\delta_t}{(1+t)^{e_1+\dots+e_r}}\right)\to \left(\left(\sum_{i=1}^r a_ie_i\right)h^{e_1+\dots+e_r-1},h^{e_1+\dots+e_r}\right)\in 
\HR_{p,q}(X)
\] 
as $t\to \infty$. Thus, by the continuity of eigenvalues, 
$(\gamma,c_{e_1}(E_1)\dotsm c_{e_r}(E_r))=(\gamma_0,\delta_0)\in 
\HR_{p,q}(X)$. 
\end{proof}

Taking $r=1$, we get the following special case.

\begin{cor}\label{c.cpenult}
Let $E$ be an ample $\R$-twisted vector bundle of rank $e\ge 1$ on a smooth 
projective variety $X$ of dimension $d=p+q+e-1$. Then $(c_{e-1}(E),c_e(E))\in 
\HR_{p,q}(X)$. 
\end{cor}

In Corollary \ref{c.cpenult}, $c_{e-1}(E)$ does not satisfy the hard 
Lefschetz theorem on $H^{p,q}(X)$ in general. See \cite[Example 9.2]{RT} for 
a counterexample in the case $p=q=2$. 

\begin{cor}\label{c.cpenultw}
Let $E_1,\dots,E_r$ be nef $\R$-twisted vector bundles of ranks 
$e_1,\dots,e_r$, respectively, on a smooth projective variety $X$ of 
dimension $d=p+q+e_1+\dots+e_r-1$. Then 
\[(c_{e_1-1}(E_1)c_{e_2}(E_2)\dotsm 
c_{e_r}(E_r), c_{e_1}(E_1)\dotsm c_{e_r}(E_r))\in \overline{\HR}_{p,q}(X).
\] 
\end{cor}

\begin{proof}
This follows from Corollary \ref{c.cpenultm} by continuity.
\end{proof}

\subsection{Hodge--Riemann pairs}

In the rest of this paper we will mainly focus on the case 
$H^{p-2,q-2}(X)=0$. 

\begin{defn}
Let $p+q+k=d$. We write $H^{p,q}$ for $H^{p,q}(X)$. Consider the map
\[H^{k,k}(X,\R)\times H^{1,1}(X,\R)\to \Herm(H^{p,q})\times \Map(H^{p-1,q-1},H^{p,q})\]
carrying $(\gamma,h)$ to $(H,\iota)=(\langle-,-\rangle_\gamma,-\wedge h)$. We 
say that $(\gamma,h)$ is a \emph{Hodge--Riemann pair} on $H^{p,q}(X)$ if 
$(H,\iota)$ is a Hodge--Riemann pair. We let $\fHR_{p,q}(X)$ denote the set 
of Hodge--Riemann pairs on $H^{p,q}(X)$ and let $\overline{\fHR}_{p,q}(X)$ 
denote the closure of $\fHR_{p,q}(X)$ in $H^{k,k}(X,\R)\times H^{1,1}(X,\R)$. 
We let $\fHRw_{p,q}(X)$ and $\fHRvw_{p,q}(X)$ denote the inverse images of 
$\fHRw(H^{p,q},H^{p-1,q-1})$ and $\fHRvw(H^{p,q},H^{p-1,q-1})$ in 
$H^{k,k}(X,\R)\times H^{1,1}(X,\R)$, respectively. 
\end{defn}

By definition, $(\gamma,h)$ is a Hodge--Riemann pair on $H^{p,q}(X)$ if and
only if $\langle h-,h-\rangle_\gamma=-\langle-,-\rangle_{h^2\gamma}$ is
negative definite on $H^{p-1,q-1}(X)$ and $\langle-,-\rangle_\gamma$ is
positive definite on $H^{p,q}(X)_{\gamma h\hprim}$. This is equivalent to the
condition that $\langle-,-\rangle_{h^2\gamma}$ is positive definite on
$H^{p-1,q-1}(X)$ and $\langle -,-\rangle_{\gamma}$ on $H^{p,q}(X)$ has
signature $(h^{p,q}-h^{p-1,q-1},h^{p-1,q-1})$, by Lemma \ref{l.1}. Moreover,
under the assumption that $\langle-,-\rangle_{h^2\gamma}$ is positive
definite on $H^{p-1,q-1}(X)$, $(\gamma,h)\in \fHRw_{p,q}(X)$ if and only if
the negative index of inertia of $\langle -,-\rangle_{\gamma}$ on
$H^{p,q}(X)$ is $h^{p-1,q-1}$, by Lemma \ref{l.1}.

\begin{remark}\label{r.changeh}
It follows from the above that if $(\gamma,h)\in \fHR_{p,q}(X)$ (resp.\ $\in 
\fHRw_{p,q}(X)$ such that $\langle -,-\rangle_{h^2\gamma}$ is positive 
definite on $H^{p-1,q-1}(X)$), then $(\gamma,h')\in \fHR_{p,q}(X)$ (resp.\ 
$\fHRw_{p,q}(X)$) for all $h'\in H^{1,1}(X,\R)$ such that $\langle 
-,-\rangle_{h'^2\gamma}$ is positive definite on $H^{p-1,q-1}(X)$. 
\end{remark}

We have 
\begin{equation}\label{e.incl}
\fHR_{p,q}(X)\subseteq \overline{\fHR}_{p,q}(X)\subseteq
\fHRw_{p,q}(X)\subseteq \fHRvw_{p,q}(X).
\end{equation}
By Hodge symmetry, $\fHR_{p,q}(X)=\fHR_{q,p}(X)$ and similarly for the other 
sets in \eqref{e.incl}. 

\begin{remark}\label{r.ineq}
Let $(\gamma,h)\in \fHRvw_{p,q}(X)$. Let $\alpha_1,\dots,\alpha_m$ be a basis 
of $H^{p-1,q-1}(X)$ and $\beta_i=h\alpha_i$. By Lemma \ref{l.ineq}, for every 
$\beta_0\in H^{p,q}(X)$, we have 
\[(-1)^m\det(\langle
\beta_i,\beta_j\rangle_{\gamma})_{0\le i,j\le m}\ge 0.
\]
By Lemma \ref{l.1}, if $(\gamma,h)\in \fHR_{p,q}(X)$, then equality holds if
and only if $\beta_0\in H^{p-1,q-1}(X)\wedge h$.
\end{remark}

Next we discuss functoriality of the Hodge--Riemann property, which will be
used later.

\begin{lemma}\label{l.funcsum}
Let $\pi\colon Z\to X$ be an alteration of smooth projective varieties. For 
$\gamma\in H^{k,k}(X,\R)$ and $h\in H^{1,1}(X,\R)$ such that 
$(\pi^*\gamma,\pi^*h)\in \fHR_{p,q}(Z)$ (resp.\ $(\pi^*\gamma,\pi^*h)\in 
\fHRw_{p,q}(Z)$ and $\langle-,-\rangle_{\pi^*(\gamma h^2)}$ is positive 
definite on $H^{p-1,q-1}(Z)$), we have $(\gamma,h)\in 
    \fHR_{p,q}(X)$ (resp.\ $(\gamma,h)\in
    \fHRw_{p,q}(X)$).
\end{lemma}

\begin{proof}
Let $m$ be the generic degree of $\pi$. Since $\pi_*\pi^*$ is multiplication
by $m$, we have
\begin{equation}\label{e.decomp}
H^{p,q}(Z)=\pi^*H^{p,q}(X)\oplus V',\quad
H^{p-1,q-1}(Z)=\pi^*H^{p-1,q-1}(X)\oplus W',
\end{equation}
where $V'$ and $W'$ are the kernels of $\pi_*$ on $H^{p,q}$ and 
$H^{p-1,q-1}$, respectively. For $\alpha,\alpha'\in H^{p,q}(X)$ and $\beta\in 
V'$, we have 
\[
\langle \pi^*\alpha,\pi^*\alpha'\rangle_{\pi^*\gamma}=m\langle
\alpha,\alpha'\rangle_{\gamma},\quad
\langle
\pi^*\alpha,\beta\rangle_{\pi^*\gamma}=\langle
\alpha,\pi_*\beta\rangle_\gamma =0.
\]
In other words, the first decomposition in \eqref{e.decomp} identifies 
$(H^{p,q}(X),\langle -,-\rangle_{m\gamma})$ with a direct summand of 
$(H^{p,q}(Z),\langle -,-\rangle_{\pi^*\gamma})$. Moreover, $-\wedge \pi^* h$  
preserves the decompositions \eqref{e.decomp}. We conclude by Remark 
\ref{r.summand}. 
\end{proof}

\begin{lemma}\label{l.func2}
Let $\pi\colon Z\to X$ be a morphism of smooth projective varieties such that
$\pi^*\colon H^{p-1,q-1}(X)\to H^{p-1,q-1}(Z)$ is a bijection. Let $h\in
H^{1,1}(X,\R)$.
\begin{enumerate}
\item For $(\gamma,\pi^*h)\in \fHR_{p,q}(Z)$, we have $(\pi_*\gamma,h)\in
    \fHR_{p,q}(X)$ if $\pi^*\colon H^{p,q}(X)\to H^{p,q}(Z)$ is an
    injection.
\item For $(\gamma,\pi^*h)\in \fHRw_{p,q}(Z)$, we
    have $(\pi_*\gamma,h)\in \fHRw_{p,q}(X)$.
\item For $(\gamma,\delta)\in \HRw_{p,q}(Z)$, we have 
    $(\pi_*\gamma,\pi_*\delta)\in \HRw_{p,q}(X)$. 
\end{enumerate}
\end{lemma}

\begin{proof}
This follows from Lemmas \ref{l.func0} and \ref{l.func} applied to
$W=H^{p-1,q-1}(X)\simeq H^{p-1,q-1}(Z)$ and $f=\pi^*\colon H^{p,q}(X)\to
H^{p,q}(Z)$. Indeed, $\langle \pi^*-,\pi^*-\rangle_{\gamma}=\langle
-,-\rangle_{\pi_*\gamma}$.
\end{proof}

We will need the following variants of Lemma \ref{l.func2}.

\begin{lemma}\label{l.funcprod}
Let $X$ and $Y$ be smooth projective varieties satisfying $H^{p-2,q-2}(X)=0$. 
Let $\pi\colon X\times Y \to X$ be the projection. Let $h\in H^{1,1}(X,\R)$ 
and let $\gamma\in \bigoplus_{i=0}^{d_Y} H^{k-i,k-i}(X)\otimes H^{i,i}(Y)$, 
where $k\ge d_Y=\dim(Y)$. 
\begin{enumerate}
\item For $(\gamma,\pi^*h)\in \fHR_{p,q}(X\times Y)$, we have 
    $(\pi_*\gamma,h)\in \fHR_{p,q}(X)$. 
\item For $(\gamma,\pi^*h)\in \fHRw_{p,q}(X\times Y)$ such that 
    $\langle-,-\rangle_{\gamma\pi^*h^2}$ is positive definite on 
    $H^{p-1,q-1}(X\times Y)$, we have $(\pi_*\gamma,h)\in \fHRw_{p,q}(X)$. 
\end{enumerate}
\end{lemma}

\begin{proof}
By the K\"unneth formula, $H^{p-1,q-1}(X\times Y)=W\oplus W'$, 
$H^{p,q}(X\times Y)=V\oplus V'$, where 
\[W=H^{p-1,q-1}(X)\otimes H^{0,0}(Y),\quad
V=H^{p,q}(X)\otimes H^{0,0}(Y)\oplus H^{p-1,q-1}(X)\otimes H^{1,1}(Y)
\] 
In fact, $V$ and $V'$ are direct summands of $(H^{p,q}(X\times Y),\langle 
-,-\rangle_{\gamma})$ and $-\wedge\pi^*h$ sends $W$ into $V$ and $W'$ into 
$V'$. By Remark \ref{r.summand}, $(\gamma,\pi^*h)$ induces an element of 
$\fHR(V,W)$ (resp.\ $\fHRw(V,W)$) in case (a) (resp.\ (b)). We then conclude 
by Lemmas \ref{l.func0} and \ref{l.func} as in the proof of Lemma 
\ref{l.func2}. 
\end{proof}

\begin{lemma}\label{l.func3}
Let $\pi\colon C\to X$ be a morphism of projective varieties with $X$ smooth
and let $\phi\colon Z\to C$ be an alteration with $Z$ smooth projective such
that $(\pi\phi)^*\colon H^{p-1,q-1}(X)\to H^{p-1,q-1}(Z)$ is a bijection. For
$\gamma\in H^{2k}(C,\R)$ and $h\in H^{1,1}(X,\R)$ such that
 $(\phi^*\gamma,(\pi\phi)^*h)\in
    \fHRw_{p,q}(Z)$, we have $(\pi_*\gamma,h)\in \fHRw_{p,q}(X)$. Moreover, 
    for $\gamma\in H^{2k}(C,\R)$ and $\delta\in H^{2(k+1)}(C,\R)$ such that 
    $(\phi^*\gamma,\phi^*\delta)\in \HRw_{p,q}(Z)$, we have $(\pi_*\gamma,\pi_*\delta)\in 
    \HRw_{p,q}(X)$.
\end{lemma}

\begin{proof}
Let $m$ be the generic degree of $\phi$. Then 
$\pi_*\gamma=\frac{1}{m}(\pi\phi)_*\phi^*\gamma$ and 
$\pi_*\delta=\frac{1}{m}(\pi\phi)_*\phi^*\delta$. Thus it suffices to apply 
Lemma \ref{l.func2} to the composite morphism $\pi\phi$. 
\end{proof}

Next we discuss Hodge--Riemann pairs of the form $(\gamma h,h')$.

\begin{prop}\label{p.HL2}
Let $X$ be a smooth projective variety of dimension $d=p+q+k+1$ with $k\ge 
0$. Let $\gamma\in H^{k,k}(X,\R)$. 
\begin{enumerate}
\item Let $h'\in H^{1,1}(X,\R)$. Assume that for every irreducible smooth 
    ample divisor $Z$ on $X$, $(\gamma|_{Z},h'|_Z)\in \fHRw_{p,q}(Z)$. Then 
    $(\gamma h,h')\in \fHRw_{p,q}(X)$ for all $h\in \Amp(X)$. 

\item Assume that for every irreducible smooth ample divisor $Z$ on $X$, 
    $(\gamma|_{Z},h'|_Z)\in \fHR_{p,q}(Z)$ for all $h'\in \Amp(X)$. Then 
    $(\gamma h,h')\in \fHR_{p,q}(X)$ for all  $h,h'\in 
\Amp(X)$. 
\end{enumerate} 
\end{prop}

\begin{proof}
(a) This follows from the semidefinite and $\HRw$ cases of Proposition 
\ref{p.HL}.

(b) By the $\HR$ case of Proposition \ref{p.HL}, it suffices to show that 
$\langle-,- \rangle_{\gamma h h'^2}$ is positive definite on $H^{p-1,q-1}(X)$ 
for all $h,h'\in \Amp(X)$. This follows from the definite case of Proposition 
\ref{p.HL}, since by assumption $\langle-,-\rangle_{(\gamma h'^2)|_Z}$ is 
positive definite on $H^{p-1,q-1}(Z)$.
\end{proof}

The proposition can be iterated as follows.

\begin{cor}\label{c.multiHL}
Let $X$ be a smooth projective variety of dimension $d=p+q+k+l$ with $k,l\ge 
0$. Let $\gamma\in H^{k,k}(X,\R)$. 
\begin{enumerate}
\item Let $h\in H^{1,1}(X,\R)$. Assume that for every chain $Z=Z_0\subseteq 
    \dots \subseteq Z_l=X$, where each $Z_i$ is an irreducible smooth ample 
    divisor of $Z_{i+1}$, we have $(\gamma|_{Z},h|_Z)\in \fHRw_{p,q}(Z)$. 
    Then $(\gamma h_1\dotsm h_l,h)\in \fHRw_{p,q}(X)$  for 
all $h_1,\dots,h_l\in \Amp(X)$. 
\item Assume that for every chain $Z=Z_0\subseteq \dots \subseteq Z_l=X$ as 
    in (a), we have $(\gamma|_{Z},h_Z)\in \fHR_{p,q}(Z)$ for all $h_Z\in 
    \Amp(Z)$. Then $(\gamma h_1\dotsm h_l,h)\in \fHR_{p,q}(X)$  for all 
    $h_1,\dots,h_l,h\in \Amp(X)$. 
\end{enumerate}
\end{cor}

Note that, for $Z$ as above, the map $H^{p-j,q-j}(X)\to H^{p-j,q-j}(Z)$ is a 
bijection for all $j\ge 1$ by the Lefschetz hyperplane theorem. 

\begin{proof}
This follows from Proposition \ref{p.HL2} by induction on $l$.
\end{proof}

\subsection{Cone classes}

\begin{theorem}\label{t.cone0}
Let $X$ be a smooth projective variety of dimension $d$ and let $F$ be an 
$\R$-twisted vector bundle on $X$ of rank $r+1$. Let $\tau\colon 
C\hookrightarrow P=\PP_\bullet(F)$ be the inclusion of a closed subvariety of 
dimension $d_{C}\ge d+2$ dominating $X$ and let $Q$ be the universal quotient 
bundle on $P$.  Assume that $H^{p-2,q-2}(X)=0$. Let $\pi_P\colon P\to X$ be 
the projection and $\pi=\pi_P\tau$. Let $d_C-d\le j\le r$ and $p+q+j+l=d_C$, 
where $l\ge 0$. Let $\gamma\in H^{l,l}(X,\R)$, $h\in H^{1,1}(X,\R)$. Let 
$k=j-d_C+d$. We make the following assumptions. 
\begin{enumerate}[(A)]
\item $(\gamma h^k,h)\in \fHR_{p,q}(X)$.
\item $(\delta_i(t),\pi_P^*h)\in \fHRvw_{p,q}(P)$ for all $0\le i\le k$ and
    all $t$ in an open interval $I\subseteq \R$ containing $0$. Here
    \[\delta_i(t)=\tau_*(c_{j-i}(Q|_C\langle t\pi^*
    h\rangle))\pi_P^*(h^i \gamma).\]
\item $(\delta_{1}(0),\delta_0(0)\pi_P^*h)\in \HRw_{p,q}(P)$ if $k\ge 1$. 
\end{enumerate}
Then $(\pi_{*}(c_{j}(Q|_C)) \gamma,h)$ is a Hodge--Riemann pair on 
$H^{p,q}(X)$. Moreover, 
\[(\tau_{*}(c_{j-i}(Q|_C)) \pi_P^*(h^i\gamma),\pi_P^*h)\]
is a Hodge--Riemann pair on $H^{p,q}(P)$ for all $1\le i\le k$.
\end{theorem}

\begin{lemma}\label{l.zeta}
Let $\pi_P\colon P\to X$, $d$, and $Q$ be as in Theorem \ref{t.cone0}. Let 
$\zeta=c_1(\cO_P(1))$. Then, for all $k$, we have 
\[c_k(Q)=c_{k-1}(Q)\zeta+\pi_P^*c_k(F).\]
Moreover, for all $\alpha\in H^{p,q}(P)$, $\beta\in H^{p',q'}(X)$ with 
$p+p'+k\ge d+2$, we have 
\begin{equation}\label{e.vanish}
c_k(Q)\alpha\wedge \pi_P^*\beta=c_{k-1}(Q)\zeta \alpha\wedge \pi_P^*\beta.
\end{equation}
\end{lemma}

\begin{proof}
The first statement is standard \cite[Lemma 4.17]{RT}. Indeed, it follows
from the short exact sequence $0\to \cO_{P}(-1)\to \pi_P^*F\to Q\to 0$ that
\[\pi_P^*c(F)=c(\cO_P(-1))c(Q)=(1-\zeta)c(Q),\]
which implies
\begin{equation}\label{e.Chern}
\pi_P^*c_k(F)=c_k(Q)-\zeta c_{k-1}(Q).
\end{equation}

We have
\[H^{p,q}(P)=\pi_P^*H^{p,q}(X)\oplus \pi_P^*H^{p-1,q-1}(X)\zeta. \]
Thus $\alpha=\pi_P^*\alpha_0+\zeta\pi_P^*\alpha_1$ for $\alpha_i\in
H^{p-i,q-i}(X)$, $i=0,1$. Since $p-i+p'+k>d$, we have $c_k(F)\alpha_i\wedge
\beta=0$, which implies that
\begin{equation}\label{e.vanish2}
\pi_P^*c_k(F)\alpha\wedge \pi_P^*\beta=0
\end{equation}
\eqref{e.vanish} follows from \eqref{e.Chern} and \eqref{e.vanish2}.
\end{proof}

\begin{proof}[Proof of Theorem \ref{t.cone0}]
Let $\zeta=c_1(\cO_P(1))$. We apply Corollary \ref{c.iterate} to
\[V=H^{p,q}(P)=U\oplus fW,\quad W=H^{p-1,q-1}(X),\quad U=\pi_P^* H^{p,q}(X),\]
$f=\zeta\wedge \pi_P^*-\colon W\to V$, $\iota=\pi_P^*(h\wedge -)\colon W\to 
V$, and $H_t(-,-)=\langle-,- \rangle_{\delta_0(t)}$. We check that the 
assumptions are satisfied.  Clearly $\iota W\subseteq U$. Moreover, by (A), 
$\iota$ is an injection. We have 
\begin{equation}\label{e.scalar}
H^{(i)}_t(-,-)=\langle-,-\rangle_{\delta_0^{(i)}(t)},\quad \delta^{(i)}_0(t)=\frac{(r-j+i)!}{(r-j)!}\delta_i(t).
\end{equation}
Since $p+q+j+l=d_C\ge d+2$, \eqref{e.f} holds for $0\le i\le k$ by Lemma 
\ref{l.zeta}, with $\kappa_i=r-j+i+1>0$. 

Assumption (A) of Corollary \ref{c.iterate} follows from (A) here, since 
$\pi_{P*}(\delta_k(t))=\gamma h^k\pi_*c_{j-k}(Q|_C\langle t\pi^*h \rangle)$, 
where $\pi_*c_{j-k}(Q|_C\langle t\pi^*h \rangle)$ is a positive multiple of 
$1_X\in H^0(X,\R)$ as in the proof of Theorem \ref{t.cone}. Assumptions (B) 
and (C) of Corollary \ref{c.iterate} are equivalent to (B) and (C) here by 
\eqref{e.scalar}. Finally, since $\pi_*(c_{j-k-1}(Q|_C))=0$, we have 
$H_0^{(k+1)}|_{U\times U}=0$, which implies (e). 

By Corollary \ref{c.iterate}, we have $(\pi_{P*}\delta_0(0),h)\in
\fHR_{p,q}(X)$, where $\pi_{P*}\delta_0(0)=\pi_*(c_j(Q|_C))\gamma$, and
$(\delta_i(0),\pi_P^*h)\in \fHR_{p,q}(P)$ for $1\le i \le k$.
\end{proof}

\begin{lemma}\label{l.SchurHRw}
Let $\lambda^1,\dots,\lambda^r$ be partitions and let $j_1,\dots,j_r\ge 0$. 
Let $X$ be a smooth projective variety of dimension $d=p+q+(\lvert 
\lambda^{1}\rvert-j_1)+\dots + (\lvert \lambda^{r}\rvert-j_r)$ with 
$H^{p-2,q-2}(X)=0$. Let $E_1,\dots,E_r$ be nef $\R$-twisted vector bundles on 
$X$ of ranks $e_1,\dots,e_r$, respectively. 
\begin{enumerate}
\item For any nef class $h\in \NS(X)_\R$,   
    $(s_{\lambda^1}^{[j_1]}(E_1)\dotsm s_{\lambda^r}^{[j_r]}(E_r),h)\in 
    \fHRw_{p,q}(X)$.
\item If $j_1=1$, then  
    $(s_{\lambda^1}^{[1]}(E_1)\gamma_1,s_{\lambda^1}(E_1)\gamma_1)\in 
    \HRw_{p,q}(X)$, where $\gamma_1=s_{\lambda^2}^{[j_2]}(E_2)\dotsm 
    s_{\lambda^r}^{[j_r]}(E_r)$. 
\end{enumerate}
\end{lemma}

\begin{proof}
We may assume $(\lambda^i)_1\le e_i$ and $j_i\le \lvert \lambda^i\rvert$ for 
all $1\le i\le r$. Let $\gamma=s_{\lambda^1}^{[j_1]}(E_1)\dotsm 
s_{\lambda^r}^{[j_r]}(E_r)$ and
\[\gamma_{i}=\prod_{\substack{i'=1\\i'\neq i}}^{r} s_{\lambda^{i'}}^{[j_{i'}]}(E_{i'}),\quad \gamma_{1,i}=\prod_{\substack{i'=2\\i'\neq i}}^{r} s_{\lambda^{i'}}^{[j_{i'}]}(E_{i'}).\]

We first prove the case $j_1=\dots=j_r=0$ of (a) (resp.\ the case 
$j_2=\dots=j_r=0$ of (b)) by induction on $\#I$, where $I=\{i\mid 
s_{\lambda^i}\neq c_{e_i}\}\subseteq \{1,\dots,r\}$. In the case $\#I=0$, we 
may assume by continuity that $E_1,\dots,E_r$ are ample $\Q$-twisted vector 
bundles and $h$ is ample, which implies $(\gamma,h)\in \fHR_{p,q}(X)$ by 
Theorem \ref{t.top} (resp.\ $(c_{e_1-1}(E_1)\gamma_1,c_{e_1}(E_1)\gamma_1)\in 
\HRw_{p,q}(X)$ by Corollary \ref{c.cpenultw}). Assume $\#I>0$. Let $i\in I$. 
By the Kempf--Laksov formula (Corollary \ref{c.KLtwist}), 
$s_{\lambda^i}(E_i)=\pi_*c_f(Q|_C)$ (resp.\ and 
$s^{[j_i]}_{\lambda^i}(E_i)=\pi_*c_{f-j_i}(Q|_C)$) in the notation of Theorem 
\ref{t.KL}. Let $\phi\colon Z\to C$ be the resolution constructed in Section 
\ref{s.Schur}. By Lemma \ref{l.KL}, $(\pi\phi)^*\colon H^{p-1,q-1}(X)\to 
H^{p-1,q-1}(Z)$ is a bijection. By the induction hypothesis, 
$(\phi^*(c_f(Q|_C)\pi^*\gamma_{i}),(\pi\phi)^*h)\in \fHRw_{p,q}(Z)$ (resp.\ 
$(\phi^*(c_{f-1}(Q|_C)\pi^*\gamma_{1}),\phi^*(c_f(Q|_C)\pi^*\gamma_{1}))\in 
\HRw_{p,q}(Z)$ if $i=1$ and 
\[(\phi^*(c_{f}(Q|_C)\pi^*(s_{\lambda^1}^{[1]}(E_1)\gamma_{1,i})),\phi^*(c_f(Q|_C)\pi^*(s_{\lambda^1}(E_1)\gamma_{1,i})))\in 
\HRw_{p,q}(Z)\] if $i>1$). Thus, by Lemma \ref{l.func3}, 
$(\gamma,h)=(\pi_*(c_f(Q|_C)\pi^*\gamma_{i}),h)\in \fHRw_{p,q}(X)$ (resp.\ 
$(s_{\lambda^1}^{[1]}(E_1)\gamma_1,s_{\lambda^1}(E_1)\gamma_1)\in 
\HRw_{p,q}(X)$). 

For the general case, we proceed by induction on $n=\#J$ in case (a) (resp.\ 
on $n=\#(J\backslash \{1\})$ in case (b)), where $J=\{i\mid j_i>0\}\subseteq 
\{1,\dots,r\}$. The case $n=0$ was already proven. Assume $n>0$. Let $i\in J$ 
(resp.\ $i\in J\backslash\{1\}$). By the derived Kempf--Laksov formula 
(Corollary \ref{c.KLtwist}), 
$s_{\lambda^i}^{[j_i]}(E_i)=\pi_*c_{f-j_i}(Q|_C)$ in the notation of Theorem 
\ref{t.KL}. For $\phi$ as above,  
$(\phi^*(c_{f-j_i}(Q|_C)\pi^*\gamma_{i}),(\pi\phi)^*h)\in \fHRw_{p,q}(Z)$ 
(resp.\ 
$(\phi^*(c_{f-j_i}(Q|_C)\pi^*(s_{\lambda^1}^{[1]}(E_1)\gamma_{1,i})),\phi^*(c_{f-j_i}(Q|_C)\pi^*(s_{\lambda^1}(E_1)\gamma_{1,i})))\in 
\HRw_{p,q}(Z)$) by the induction hypothesis.  Thus, by Lemma \ref{l.func3}, 
$(\gamma,h)=(\pi_*(c_{f-j_i}(Q|_C)\pi^*\gamma_{i}),h)\in \fHRw_{p,q}(X)$  
(resp.\ $(s_{\lambda^1}^{[1]}(E_1)\gamma_1,s_{\lambda^1}(E_1)\gamma_1)\in 
\HRw_{p,q}(X)$). 
\end{proof}

Next we deduce Hodge--Riemann relations for products of Schur and Schubert 
classes from the cone theorem (Theorem \ref{t.cone0}) and the weak 
Hodge--Riemann relations of Lemma \ref{l.SchurHRw}. Note that we do 
\emph{not} assume that $E_1,\dots,E_r$ have the same $\R$-twists modulo 
$\NS(X)_\Q$, in contrast to Theorem \ref{t.top} on products of top Chern 
classes. On the other hand, here we need the assumption $H^{p-2,q-2}(X)=0$, 
which is already used in Theorem \ref{t.cone0}. 

\begin{theorem}\label{p.refined}
Let $\lambda^1,\dots,\lambda^r$ be partitions and, for each $1\le i\le r$, 
let $0\le j_i\le \lvert \lambda^i\rvert$. Let $X$ be a smooth projective 
variety of dimension $d=p+q+(\lvert \lambda^{1}\rvert-j_1)+\dots + (\lvert 
\lambda^{r}\rvert-j_r)$ with $H^{p-2,q-2}(X)=0$. Let $E_1,\dots,E_r$ be nef 
$\R$-twisted vector bundles on $X$ of ranks $e_1,\dots,e_r$, respectively, 
satisfying $e_i\ge (\lambda^i)_1$ for all $1\le i \le r$. Let 
$h_0,\dots,h_r\in \NS(X)_\R$ be nef. Assume that 
\begin{equation}\label{e.refined}
(h_1^{\lvert \lambda^1\rvert -j_1}\dotsm h_r^{\lvert \lambda^r\rvert 
-j_r},h_i)\in \fHR_{p,q}(X) 
\end{equation}
for all $0\le i\le r$. Then  
\[(s_{\lambda^1}^{[j_1]}(E_1\langle h_1\rangle)\dotsm 
s_{\lambda^r}^{[j_r]}(E_r\langle h_r\rangle),h_0)\in \fHR_{p,q}(X).\] 
\end{theorem}

\begin{proof}
We proceed by induction on $p$ and $q$. The case $p<0$ or $q<0$ is trivial. 
Assume $p,q\ge 0$. For $0\le i\le r$, let 
\[\alpha_i=s_{\lambda^1}^{[j_1]}(E_1\langle h_1\rangle)\dotsm 
s_{\lambda^i}^{[j_i]}(E_i\langle h_i\rangle),\quad \beta_i=h_{i+1}^{\lvert 
\lambda^{i+1}\rvert -j_{i+1}}\dotsm h_r^{\lvert \lambda^r\rvert -j_r}.
\]  
We proceed by induction on $i$ to show that $(\alpha_i\beta_i,h_i)\in 
\fHR_{p,q}(X)$ for $0\le i\le r$. The case $i=0$ holds by \eqref{e.refined}. 
Assume $i\ge 1$. We apply the derived Kempf--Laksov formula (Corollary 
\ref{c.KLtwist}) to $s_{\lambda^i}^{[j_i]}(E_i\langle h_i \rangle)$ and the 
cone theorem (Theorem \ref{t.cone0}) to $F=\Hom(V,E_i\langle h_i\rangle)$ 
(with $V$ of dimension large enough so that $\dim(C)\ge d+2$), $k=\lvert 
\lambda^i\rvert -j_i$, and $\gamma=\alpha_{i-1}\beta_i$. By the induction 
hypothesis on $i$, $(\gamma 
h_i^k,h_{i-1})=(\alpha_{i-1}\beta_{i-1},h_{i-1})\in \fHR_{p,q}(X)$. By 
\eqref{e.refined}, $\langle-,-\rangle_{\beta_0h_i^2}$ is positive definite on 
$H^{p-1,q-1}(X)$. By the induction hypothesis on $p$ and $q$, this implies 
that $\langle -,-\rangle_{\gamma h_i^{k+2}}$ is positive definite on 
$H^{p-1,q-1}(X)$. Thus, by Remark \ref{r.changeh}, $(\gamma h_i^k,h_i)\in 
\fHR_{p,q}(X)$. This proves assumption (A) of the cone theorem. Let 
$\phi\colon Z\to C$ be the resolution constructed in Section \ref{s.Schur}. 
By Lemma \ref{l.KL}, $(\pi\phi)^*\colon H^{p-1,q-1}(X)\to H^{p-1,q-1}(Z)$ is 
a bijection. Assumptions (B) and (C) then follow from Lemmas \ref{l.func3} 
and \ref{l.SchurHRw}. Thus, by the cone theorem, $(\gamma 
s^{[j_i]}_{\lambda^i}(E_i\langle h_i\rangle),h_i)=(\alpha_i\beta_i,h_i)\in 
\fHR_{p,q}(X)$. By induction, $(\alpha_r,h_r)\in \fHR_{p,q}(X)$. By 
\eqref{e.refined}, $\langle-,-\rangle_{\beta_0h_0^2}$ is positive definite on 
$H^{p-1,q-1}(X)$. By the induction hypothesis on $p$ and $q$, this implies 
that $\langle -,-\rangle_{\alpha_r h_0^{2}}$ is positive definite on 
$H^{p-1,q-1}(X)$. Thus, by Remark \ref{r.changeh}, $(\alpha_r,h_0)\in 
\fHR_{p,q}(X)$. 
\end{proof}

\begin{cor}\label{c.final}
Let $\lambda^1,\dots,\lambda^r$ be partitions and, for each $1\le i\le r$, 
let $0\le j_i\le \lvert \lambda^i\rvert$. Let $X$ be a smooth projective 
variety of dimension $d=p+q+(\lvert \lambda^{1}\rvert-j_1)+\dots + (\lvert 
\lambda^{r}\rvert-j_r)$ with $H^{p-2,q-2}(X)=0$. Let $E_1,\dots,E_r$ be ample 
(resp.\ nef) $\R$-twisted vector bundles on $X$ of ranks $e_1,\dots,e_r$, 
respectively, satisfying $e_i\ge (\lambda^i)_1$ for all $1\le i \le r$. Let 
$h\in \NS(X)_\R$ be ample (resp.\ nef). Then  
\[(s_{\lambda^1}^{[j_1]}(E_1)\dotsm 
s_{\lambda^r}^{[j_r]}(E_r),h) 
\] 
belongs to $\fHR_{p,q}(X)$ (resp.\ $\overline{\fHR}_{p,q}(X)$). 
\end{cor}

\begin{proof}
The nef case follows from the ample case by continuity. For the ample case, 
let $t>0$ such that $E_1\langle -th\rangle, \dots, E_r\langle -th\rangle$ are 
ample. Since $(h^{d-p-q},h)\in \fHR_{p,q}(X)$, it suffices to apply Theorem 
\ref{p.refined} to $E_1\langle -th\rangle, \dots, E_r\langle -th\rangle$ and 
$h_0=\dots=h_r=th$. 
\end{proof}

Taking $r=1$ and $j_1=0$ in Corollary \ref{c.final}, we get Theorem 
\ref{t.main}. Corollary \ref{c.main} follows by Remark \ref{r.ineq}. 

In the case $H^{p-1,q-1}(X)=0$, Corollary \ref{c.final} implies the following 
promised generalization of Theorem \ref{t.Schurdef}. 

\begin{cor}\label{c.posfinal}
Let $\lambda^1,\dots,\lambda^r$ be partitions and let $X$ be a smooth 
projective variety of dimension $d=p+q+\lvert \lambda^{1}\rvert+\dots + 
\lvert \lambda^{r}\rvert$ with $H^{p-1,q-1}(X)=0$. Let $E_1,\dots,E_r$ be 
ample (resp.\ nef) $\R$-twisted vector bundles on $X$ of ranks 
$e_1,\dots,e_r$, respectively, satisfying $e_i\ge (\lambda^i)_1$ for all 
$1\le i \le r$. Then  $\langle -,-\rangle_{\gamma}$ is positive definite 
(resp.\ positive semidefinite) on $H^{p,q}(X)$ for 
\[\gamma=s_{\lambda^{1}}(E_1)\dotsm s_{\lambda^{r}}(E_r).\]
\end{cor}

Similar results hold for Schubert classes.

\begin{lemma}\label{l.Schubertw}
Let $X$ be a smooth projective variety of dimension $d$ with 
$H^{p-2,q-2}(X)=0$. For each $1\le i\le r$, let $E_i$ be an $\be^i$-filtered 
nef $\R$-twisted vector bundle on $X$, let $w_i$ be an $\be^i$-permutation, 
and let $j_i\ge 0$ such that $d=p+q+(\ell(w_1)-j_1)+\dots + (\ell(w_r)-j_r)$. 
\begin{enumerate}
\item For any nef class $h\in \NS(X)_\R$, $(\fS_{w_1}^{[j_1]}(E_1)\dotsm 
    \fS_{w_r}^{[j_r]}(E_r),h)\in \fHRw_{p,q}(X)$.
\item For any partition $\lambda$ satisfying $\lvert \lambda 
    \rvert-1=\ell(w_1)-j_1$, 
    $(s^{[1]}_\lambda(E_1)\gamma_1,s_\lambda(E_1)\gamma_1)\in 
    \HRw_{p,q}(X)$, where $\gamma_1=\fS_{w_2}^{[j_2]}(E_2)\dotsm 
    \fS_{w_r}^{[j_r]}(E_r)$. 
\end{enumerate}
\end{lemma}

\begin{proof}
We proceed by induction on $n=\#J$ in case (a) (resp.\ on 
$n=\#(J\backslash\{1\})$ in case (b)), where $J=\{i\mid \fS_{w_i}^{[j_i]}\neq 
c_{\ell(w_i)-j_i}\}\subseteq \{1,\dots,r\}$. The case $n=0$ is a special case 
of Lemma \ref{l.SchurHRw}. For the case $n>0$, it suffices to repeat the last 
paragraph of the proof of Lemma \ref{l.SchurHRw}, replacing the derived 
Kempf--Laksov formula by the derived Fulton formula (Corollary 
\ref{c.Fultontwist}). 
\end{proof}

\begin{theorem}
Let $X$ be a smooth projective variety of dimension $d$ with 
$H^{p-2,q-2}(X)=0$. For each $1\le i\le r$, let $E_i$ be an $\be^i$-filtered 
nef $\R$-twisted vector bundle on $X$, let $w_i$ be an $\be^i$-permutation, 
and let $0\le j_i\le \ell(w_i)$ such that $d=p+q+(\ell(w_1)-j_1)+\dots + 
(\ell(w_r)-j_r)$. Let $h_0,\dots,h_r\in \NS(X)_\R$ be nef. Assume that 
\[
(h_1^{\ell(w_1) -j_1}\dotsm h_r^{\ell(w_r) 
-j_r},h_i)\in \fHR_{p,q}(X) 
\]
for all $0\le i\le r$. Then 
\[(\fS_{w_1}^{[j_1]}(E_1\langle h_1\rangle )\dotsm \fS_{w_r}^{[j_r]}(E_r\langle h_r\rangle),h_0)\in \fHR_{p,q}(X).\]
\end{theorem}

\begin{proof}
The proof is very similar to that of Theorem \ref{p.refined}, with the 
derived Kempf--Laksov formula replaced by the derived Fulton formula and 
Lemma \ref{l.SchurHRw} replaced by Lemma \ref{l.Schubertw}. 
\end{proof}

\begin{cor}\label{c.Schubertprod}
Let $X$ be a smooth projective variety of dimension $d$ with 
$H^{p-2,q-2}(X)=0$. For each $1\le i\le r$, let $E_i$ be an $\be^i$-filtered 
ample (resp.\ nef) $\R$-twisted vector bundle on $X$, let $w_i$ be an 
$\be^i$-permutation, and let $0\le j_i\le \ell(w_i)$ such that 
$d=p+q+(\ell(w_1)-j_1)+\dots + (\ell(w_r)-j_r)$. Let $h\in \NS(X)_\R$ be 
ample (resp.\ nef). Then 
\[(\fS_{w_1}^{[j_1]}(E_1)\dotsm \fS_{w_r}^{[j_r]}(E_r),h)\]
belongs to $\fHR_{p,q}(X)$ (resp.\ $\overline{\fHR}_{p,q}(X)$). 
\end{cor}

\section{Hodge--Riemann polynomials}\label{s.6}

As mentioned in the introduction, Ross and Toma \cite[Section 9.2]{RT} showed
that classes satisfying the Hodge--Riemann property are not stable under
positive linear combinations. It is natural to ask which linear combinations
of Schur classes, or more generally, of products of Schur classes of ample
vector bundles satisfy the Hodge--Riemann property. In this section, we study
polynomials that produce such linear combinations when evaluated at Chern
roots of ample $\R$-twisted vector bundles, which we call Hodge--Riemann
polynomials.

In Section \ref{s.6.1}, we give the definitions and some preliminaries. In 
Section \ref{s.6.2}, we explore relations with Lorentzian polynomials and 
dually Lorentzian polynomials.  In Section \ref{s.6.3}, we show that 
Hodge--Riemann polynomials are stable under differential operators given by 
volume polynomials of ample divisors (Theorem \ref{t.cHR}) and deduce Theorem 
\ref{t.Polya}. In Section \ref{s.6.5}, we discuss stability of the 
Hodge--Riemann property under multiplication by derived Schur polynomials. In 
Section \ref{s.6.6}, we return to Lorentzian polynomials and deduce Schur 
log-concavity of derived sequences. 

\subsection{Definitions and first properties}\label{s.6.1}

\begin{defn}\label{d.cHR}
Let $\cS_{e_1,\dots,e_r}=\bigoplus_{k=0}^\infty \cS_{e_1,\dots,e_r}^k 
\subseteq 
\R[x_{1,1},\dots,x_{1,e_1};\dots;x_{r,1},\dots,x_{r,e_r}]=\R[\underline{x}]$ 
denote the ring of invariant polynomials under the action of 
$\Sigma_{e_1}\times \dots \times \Sigma_{e_r}$, graded by the degree $k$. 
Here $\Sigma_{e_i}$ acts by permuting $x_{i,1},\dots,x_{i,e_i}$. Consider the 
following cones in $\cS_{e_1,\dots,e_r}$. 
\begin{enumerate}
\item (Generalization of the Schur positive cone) $\cP_{e_1,\dots,e_r}$ 
    consists of polynomials of the form 
    \[g=\sum_{\lambda^1,\dots,\lambda^r}
    a_{\lambda^1,\dots,\lambda^r}s_{\lambda^1}(x_{1,1},\dots,x_{1,e_1})\dotsm
    s_{\lambda^r}(x_{r,1},\dots,x_{r,e_r})
\]
where $\lambda^1,\dots,\lambda^r$ run through partitions and
$a_{\lambda^1,\dots,\lambda^r}\ge 0$. Put
$\cP_{e_1,\dots,e_r}^k=\cP_{e_1,\dots,e_r}\cap \cS^k_{e_1,\dots,e_r}$.

\item For integers $p,q$, $\cHR{p,q}^k_{e_1,\dots,e_r}$ (resp.\ 
    $\cHRw{p,q}{k}_{e_1,\dots,e_r}$) consists of $g\in 
    \cS^k_{e_1,\dots,e_r}$ such that for every smooth projective variety 
    $X$ of dimension $d=k+p+q$ satisfying $H^{p-2,q-2}(X)=0$, for all ample 
    $\R$-twisted vector bundles $E_1,\dots,E_r$ of ranks $e_1,\dots, e_r$, 
    for every ample class $h\in \NS(X)_\R$, we have 
    $(g(E_1,\dots,E_r),h)\in \fHR_{p,q}(X)$ (resp.\ $\fHRvw_{p,q}(X)$). We 
    call such polynomials \emph{Hodge--Riemann polynomials} (resp.\ 
    \emph{weakly Hodge--Riemann polynomials}) in bidegree $(p,q)$. 

\item $\cHRp{p,q}^k_{e_1,\dots,e_r}$ (resp.\ 
    $\cHRpw{p,q}{k}_{e_1,\dots,e_r}$) is defined in the same way as 
    $\cHR{p,q}^k_{e_1,\dots,e_r}$ (resp.\ $\cHRw{p,q}{k}_{e_1,\dots,e_r}$) 
    except that the $\R$-twisted vector bundles $E_1,\dots,E_r$ in the 
    definition are assumed to have the same $\R$-twist modulo $\NS(X)$. 

\item We write
    \[\cHR{}^k_{e_1,\dots,e_r}=\bigcap_{p,q\ge 0}
    \cHR{p,q}^k_{e_1,\dots,e_r},\quad
    \cHRp{}^k_{e_1,\dots,e_r}=\bigcap_{p,q\ge 0}
    \cHRp{p,q}^k_{e_1,\dots,e_r}.
    \]
\end{enumerate}
\end{defn}

Note that $\cHRw{p,q}{k}_{e_1,\dots,e_r}$ and
$\cHRpw{p,q}{k}_{e_1,\dots,e_r}$ are closed subsets of
$\cS^k_{e_1,\dots,e_r}$.

\begin{lemma}\label{l.triv} Let $p,q\ge 0$.
\begin{enumerate}
\item $\cHRww{p,q}{k}_{e_1,\dots,e_r}=\cHRww{q,p}{k}_{e_1,\dots,e_r}$, 
    $\cHRpww{p,q}{k}_{e_1,\dots,e_r}=\cHRpww{q,p}{k}_{e_1,\dots,e_r}$. 
\item $\cHRpp{p,q}^k_{e_1,\dots,e_r}\subseteq
    \cHRppw{p,q}{k}_{e_1,\dots,e_r}$.
\item $\cHRww{p,q}{k}_{e_1,\dots,e_r}\subseteq 
    \cHRpww{p,q}{k}_{e_1,\dots,e_r}$ and equality holds if $r=1$. 
\item    
    $\cHRww{p,p}{k}_{e_1,\dots,e_r}=\cHRpww{p,p}{k}_{e_1,\dots,e_r}=\cS^k_{e_1,\dots,e_r}$ 
    for $p\ge 2$. 
\item $\cHRp{p,q}^k_{e_1,\dots,e_r}\subseteq
    \cP^k_{e_1,\dots,e_r}\backslash \{0\}$ and
    $\cHRpw{p,q}{k}_{e_1,\dots,e_r}\subseteq  \cP^k_{e_1,\dots,e_r}$ for
    $p\neq q$ or $p=q\le 1$.
\item $\cHRpp{p,0}^k_{e_1,\dots,e_r}=\cP^k_{e_1,\dots,e_r}\backslash \{0\}$ 
    and $\cHRppw{p,0}{k}_{e_1,\dots,e_r}=\cP^k_{e_1,\dots,e_r}$. 
\item $\cHRpp{p+1,1}^k_{e_1,\dots,e_r}\subseteq
    \cHRpp{p,1}^k_{e_1,\dots,e_r}$ and
    $\cHRppw{p+1,1}{k}_{e_1,\dots,e_r}\subseteq
    \cHRppw{p,1}{k}_{e_1,\dots,e_r}$.
\end{enumerate}
\end{lemma}

\begin{proof}
(a) This follows from Hodge symmetry.

(b),(c) These follow from the definition.

(d) For $p\ge 2$, being Hodge--Riemann in bidegree $(p,p)$ is a vacuous 
condition, because the assumption $H^{p-2,p-2}(X)=0$ is never satisfied. 

(e) There exists a smooth projective variety $X$ of dimension $p+q+k$ such 
that $H^{p,q}(X)\neq 0$ and $H^{p-2,q-2}(X)=0$. Indeed, for $p=q\le 1$ this 
is clear and for $p\neq q$ it suffices to use the product variety in the 
proof of Proposition \ref{p.converse}. It follows that $0\notin 
\cHRp{p,q}^k_{e_1,\dots,e_r}$. It remains to prove the second inclusion. For 
$p\neq q$ or $p=q=0$, the assertion follows from Proposition \ref{p.converse} 
and Remark \ref{r.converse}. For $p=q=1$, let $g\in 
\cHRpw{1,1}{k}_{e_1,\dots,e_r}$. Let $X$ be a smooth projective variety of 
dimension $k$ and let $E_1,\dots,E_r$ be ample vector bundles of ranks 
$e_1,\dots,e_r$, respectively. Then $(g(F_1,\dots,F_r),\xi)\in 
\fHRvw_{1,1}(X\times \PP^2)$, where $F_i=E_i\boxtimes \cO_{\PP^2}(1)$ on 
$X\times \PP^2$ and $\pi\colon X\times \PP^2\to X$ is the projection and 
$\xi$ is the pullback of $c_1(\cO_{\PP^2}(1))$ to $X\times\PP^2$. Thus 
\[\int_X g(E_1,\dots,E_r)=\int_{X\times \PP^2} g(\pi^* E_1,\dots,\pi^*E_r)\xi^2=\int_{X\times \PP^2} g(F_1,\dots,F_r)\xi^2\ge 0,\]
Therefore, $g\in \cP^k_{e_1,\dots,e_r}$ by Proposition \ref{p.converse}. 

(f) By definition, 
$\cHR{p,0}^k_{e_1,\dots,e_r}\subseteq\cHRp{p,0}^k_{e_1,\dots,e_r}$. By (e), 
$\cHRp{p,0}^k_{e_1,\dots,e_r}\subseteq\cP^k_{e_1,\dots,e_r}\backslash 
    \{0\}$. Moreover, $\cP^k_{e_1,\dots,e_r}\backslash
    \{0\}\subseteq \cHR{p,0}^k_{e_1,\dots,e_r}$ by Corollary \ref{c.posfinal}.
Thus $\cHRpp{p,0}^k_{e_1,\dots,e_r}=\cP^k_{e_1,\dots,e_r}\backslash 
    \{0\}$. It follows that
    \[\cHRw{p,0}{k}_{e_1,\dots,e_r}\subseteq\cHRpw{p,0}{k}_{e_1,\dots,e_r}\subseteq\cP^k_{e_1,\dots,e_r}=\cHR{p,0}^k_{e_1,\dots,e_r}\cup \{0\}\subseteq
    \cHRw{p,0}{k}_{e_1,\dots,e_r},\]
where we used (e) in the second inclusion.

(g) Let $g\in\cHRpp{p+1,1}^k_{e_1,\dots,e_r}$ (resp.\ $g\in 
\cHRppw{p+1,1}{k}_{e_1,\dots,e_r}$). Let $X$ be a smooth projective variety 
of dimension $d=p+1+k$ and let $h$, $E_1,\dots, E_r$ be as in the definition 
of $\cHRpp{p,1}^k_{e_1,\dots,e_r}$ (resp.\ 
$\cHRppw{p,1}{k}_{e_1,\dots,e_r}$). Let $C$ be an elliptic curve. Let $h'\in 
\NS(C)$ be an ample class and let $\omega\in H^{1,0}(C)$ such that $\langle 
\omega,\omega\rangle_1=1$. Let $Y=X\times C$ and let $\pi_X\colon Y\to X$ and 
$\pi_C\colon Y\to C$ be the projections. Let $F_i=E_i\boxtimes \cO_C(h')$ and 
$h''=\pi_X^*h+\pi_C^*h'$. Then $-\wedge \omega$ identifies the Hermitian 
space $(H^{p,1}(X),\langle-,-\rangle_{g(E_1,\dots,E_r)})$ with a direct 
summand of $(H^{p+1,1}(Y),\langle-,-\rangle_{g(F_1,\dots,F_r)})$ and 
identifies $H^{p-1,0}(X) \xrightarrow{-\wedge h} H^{p,1}(X)$ with a direct 
summand of $H^{p,0}(Y) \xrightarrow{-\wedge h''} H^{p+1,1}(Y)$. Thus, 
$(g(F_1,\dots,F_r),h'')\in \fHR_{p+1,1}(Y)$ (resp.\ $\in \fHRvw_{p+1,1}(Y)$) 
implies $(g(E_1,\dots,E_r),h)\in \fHR_{p,1}(X)$ (resp.\ $ \in 
\fHRvw_{p,1}(X)$) by Remark \ref{r.summand} (resp.\ and the fact that, for 
$g$ nonzero, $\langle -,-\rangle_{g(F_1,\dots,F_r)h''^2}$ is positive 
definite on $H^{p,0}(Y)$ by (e) and Corollary \ref{c.posfinal}). Therefore, 
$g\in\cHRpp{p,1}^k_{e_1,\dots,e_r}$ (resp.\ $g\in 
\cHRppw{p,1}{k}_{e_1,\dots,e_r})$. 
\end{proof}

In particular, we have the following inclusions:
\[
\begin{tikzcd}
\cHR{}^k_{e_1,\dots,e_r}\ar[r,phantom,"\subseteq"]\ar[d,phantom,sloped,"\subseteq"]  & \cHR{1,1}^k_{e_1,\dots,e_r} \ar[r,phantom,"\subseteq"]\ar[d,phantom,sloped,"\subseteq"]& \cHR{0,0}^k_{e_1,\dots,e_r}\ar[d,equal]\\
  \cHRp{}^k_{e_1,\dots,e_r} \ar[r,phantom,"\subseteq"]& \cHRp{1,1}^k_{e_1,\dots,e_r}\ar[r,phantom,"\subseteq"] & \cHRp{0,0}^k_{e_1,\dots,e_r}\ar[r,equal] & \cP^k_{e_1,\dots,e_r}\backslash\{0\}.
\end{tikzcd}
\]

Next we look at alternative ways to define $\cHRpp{p,q}^k_{e_1,\dots,e_r}$
and  $\cHRppw{p,q}{k}_{e_1,\dots,e_r}$. By continuity, in the definition of
$\cHRppw{p,q}{k}_{e_1,\dots,e_r}$, we may replace ``ample'' by ``nef''.

\begin{lemma}\label{l.alt}
\begin{enumerate}
\item In the definition of $\cHRppw{p,q}{k}_{e_1,\dots,e_r}$, we may 
    replace $\fHRvw_{p,q}(X)$ by $\fHRw_{p,q}(X)$ or
    $\overline{\fHR}_{p,q}(X)$. 
\item In the definition of $\cHRpp{p,q}^{k}_{e_1,\dots,e_r}$ and
    $\cHRppw{p,q}{k}_{e_1,\dots,e_r}$, we may replace $h\in\NS(X)_\R$ by
    $h\in\NS(X)$.
\end{enumerate}
\end{lemma}

By (b), the definition of $\cHR{p,q}^k_{e_1,\dots,e_r}$
    in Definition \ref{d.cHR} coincides with Definition \ref{d.intro}.

\begin{proof}
(a) Since $\overline{\fHR}_{p,q}(X)\subseteq \fHRw_{p,q}(X)\subseteq 
\fHRvw_{p,q}(X)$, it suffices to prove the case of 
$\overline{\fHR}_{p,q}(X)$. The proof is similar to the end of the proof of 
\cite[Theorem 7.2]{RT2}. Let $g\in \cHRppw{p,q}{k}_{e_1,\dots,e_r}$ and let 
$X$, $E_1,\dots,E_r$, $h$ be as in the definition. Let us show 
$(g(E_1,\dots,E_r),h)\in \overline{\fHR}_{p,q}(X)$. For this, we may assume 
$g\neq 0$. By Lemma \ref{l.triv}(e), $g\in \cP^k_{e_1,\dots,e_r}$. In 
particular, $g$ is monomial-positive. Consider 
$\gamma_t=\frac{1}{(1+t)^k}g(E_1\langle th\rangle,\dots,E_r\langle 
th\rangle)$. Then $\gamma_t\to ah^k$ as $t\to \infty$, where $a>0$ is the sum 
of the coefficients of $g$. Since $(ah^k,h)\in \fHR_{p,q}(X)$, we have 
$(\gamma_t,h)\in \fHR_{p,q}(X)$ for $t\gg 0$. For all $t\ge 0$, 
$(\gamma_t,h)\in \fHRvw_{p,q}(X)$. Moreover, $\langle-,-\rangle_{\gamma_t 
h^2}$ is positive definite on $H^{p-1,q-1}(X)$. Thus, for $t\ge 0$, 
$(\gamma_t,h)\in \fHR_{p,q}(X)$ if and only if $\langle-,-\rangle_{\gamma_t}$ 
on $H^{p,q}(X)$ is nondegenerate. The set $T=\{t\in \R_{\ge 0}\mid 
(\gamma_t,h)\notin \fHR_{p,q}(X)\}$ is thus the intersection of $\R_{\ge 0}$ 
with the vanishing locus of a polynomial. Since $T\neq \R_{\ge 0}$, $T$ must 
be finite. It follows that $0$ is contained in the closure of $\R_{\ge 
0}\backslash T$. In other words, $(g(E_1,\dots,E_r),h)=(\gamma_0,h)\in 
\overline{\fHR}_{p,q}(X)$. 

(b) For $\cHRppw{p,q}{k}_{e_1,\dots,e_r}$ one reduces to $\NS(X)_\Q$ by 
continuity and then to $\NS(X)$ by scalar multiplication. Let us prove the 
case of $\cHRpp{p,q}^{k}_{e_1,\dots,e_r}$. Except in the trivial case 
$p=q>1$, for $g\neq 0$ satisfying the definition with $\NS(X)_\R$ replaced by 
$\NS(X)$, we have $g\in \cHRppw{p,q}{k}_{e_1,\dots,e_r}\subseteq 
\cP^k_{e_1,\dots,e_r}$ by the previous case and Lemma  \ref{l.triv}(e). Then 
$\langle-,-\rangle_{g(E_1,\dots,E_r)h'^2}$ is positive definite on 
$H^{p-1,q-1}(X)$ for every ample $h'\in \NS(X)_\R$ by Corollary 
\ref{c.posfinal} and we conclude by Remark \ref{r.changeh}. 
\end{proof}

\begin{lemma}\label{l.extract}
Assume $\min(p,q)\le 1$ or $e_1=\dots=e_r=1$. 
\begin{enumerate}
\item In the definition of $\cHRp{p,q}^{k}_{e_1,\dots,e_r}$, we may replace 
    ``modulo $N^1(X)$'' by ``modulo $N^1(X)_\Q$''. 
\item In the definition of $\cHRw{p,q}{k}_{e_1,\dots,e_r}$, we may restrict 
    to non-twisted vector bundles. In particular, 
    $\cHRw{p,q}{k}_{e_1,\dots,e_r}=\cHRpw{p,q}{k}_{e_1,\dots,e_r}$.
\end{enumerate}
\end{lemma}

\begin{proof}
The case $e_1=\dots=e_r=1$ follows from scalar multiplication and continuity. 
Assume $\min(p,q)\le 1$. 

(a) Let $g\in \cHRp{p,q}^k_{e_1,\dots,e_r}$. Let $X$ and $h$ be as in the 
definition and let $E_1,\dots,E_r$ be ample $\R$-twisted vector bundles of 
ranks $e_1,\dots,e_r$ having the same $\R$-twist modulo $N^1(X)_\Q$. By the 
refined Bloch--Gieseker covering \cite[Proposition 2.67]{KM}, there exists a 
finite dominant morphism $\pi\colon Z\to X$ with $Z$ a smooth variety such 
that $\pi^*E_1,\dots,\pi^*E_r$ have the same $\R$-twist modulo $N^1(Z)$.  The 
assumption $\min(p,q)\le 1$ ensures that $H^{p-2,q-2}(Z)= 0$. Thus $(\pi^* 
g(E_1,\dots,E_r),\pi^* h)\in \fHR_{p,q}(Z)$. It follows that 
$(g(E_1,\dots,E_r), h)\in \fHR_{p,q}(X)$ by Lemma \ref{l.funcsum}. 

(b) For $g$ satisfying the definition restricted to non-twisted vector 
bundles, we show $g\in \cHRw{p,q}{k}_{e_1,\dots,e_r}$. By the proof of Lemma 
\ref{l.triv}(e), we have $g\in \cP^k_{e_1,\dots,e_r}$. We may assume $g\neq 
0$. Let $X,E_1,\dots,E_r,h$ be as in the definition of 
$\cHRw{p,q}{k}_{e_1,\dots,e_r}$. By continuity, we may assume that the $E_i$ 
are ample $\Q$-twisted vector bundles. The proof is then the same as in (a), 
except that to apply Lemma \ref{l.funcsum} we need the fact that $\langle 
-,-\rangle_{\pi^*(g(E_1,\dots,E_r)h^2)}$ is positive definite on 
$H^{p-1,q-1}(Z)$, which follows from Corollary \ref{c.posfinal}. 
\end{proof}

\begin{remark}\label{r.triv}\leavevmode
\begin{enumerate}
\item A symmetric polynomial in $\cHR{p,q}^k_{\bone^e}$ does not 
    necessarily belong to $\cHR{p,q}^k_{e}$. Indeed, 
    $\cHR{0,0}^k_{\bone^e}=\cP^k_{\bone^e}\backslash\{0\}$ is the set of 
    nonzero monomial-positive polynomials, while 
    $\cHR{0,0}^k_{e}=\cP^k_e\backslash\{0\}$ is the set of nonzero Schur 
    positive symmetric polynomials. For this reason, it is important to 
    keep the subscripts $e_1,\dots,e_r$. 
\item By definition, the map $\cS^k_{e_0,\dots,e_0,e_1,\dots,e_r}\to
    \cS^k_{e_0,e_1,\dots,e_r}$ sending $g$ to
    $g(\underline{x_0},\dots,\underline{x_0},\underline{x_1},\dots,\underline{x_r})$
    carries $\cHRpp{p,q}^{k}_{e_0,\dots,e_0,e_1,\dots,e_r}$ into
    $\cHRpp{p,q}^{k}_{e_0,e_1,\dots,e_r}$ and
    $\cHRppw{p,q}{k}_{e_0,\dots,e_0,e_1,\dots,e_r}$ into
    $\cHRppw{p,q}{k}_{e_0,e_1,\dots,e_r}$.
\item For $1\le s< r$, the inclusion map 
    $\cS^k_{e_1,\dots,e_{s-1},e_{s}+\dots+e_r}\to \cS^k_{e_1,\dots,e_r}$ 
    carries $\cHRpww{p,q}{k}_{e_1,\dots,e_{s-1},e_{s}+\dots+e_r}$ into 
    $\cHRpww{p,q}{k}_{e_1,\dots,e_r}$. Indeed, it suffices to take a direct 
    sum of $\R$-twisted vector bundles in the definition. In particular, by 
    Lemma \ref{l.triv}(f), the inclusion map carries  
    $\cP^{k}_{e_1,\dots,e_{s-1},e_{s}+\dots+e_r}$ into 
    $\cP^{k}_{e_1,\dots,e_r}$. 
\end{enumerate}
\end{remark}

\begin{example}\leavevmode\label{e.Sch}
\begin{enumerate}
\item Let $\lambda^1,\dots,\lambda^r$ be partitions and, for each $1\le 
    i\le r$, let $0\le j_i\le \lvert \lambda^i\rvert$ and $e_i\ge 
    (\lambda^i)_1$. By Corollary \ref{c.final}, we have 
\[
s_{\lambda^1}^{[j_1]}(x_{1,1},\dots,x_{1,e_1})\dotsm s_{\lambda^r}^{[j_r]}(x_{r,1},\dots,x_{r,e_r})\in \cHR{}^{(\lvert \lambda^1\rvert-j_1)+\dots+(\lvert \lambda^r\rvert-j_r)}_{e_1,\dots,e_r}.
\]

\item For any sequence $\be\colon 0=e_0<\dots<e_k=e$, let 
    $\diff(\be)=(e_1-e_0,\dots,e_k-e_{k-1})$. For each $1\le i\le r$, let  
    $\be^i$ be a sequence, $w_i$ an $\be^i$-permutation, and $0\le j_i\le 
    \ell(w_i)$. By Corollary \ref{c.Schubertprod} applied to direct sums of 
    ample $\R$-twisted vector bundles, we have 
\[\fS_{w_1}^{[j_1]}(\underline{x_1})\dotsm \fS_{w_r}^{[j_r]}(\underline{x_r})\in \cHRp{}^{(\ell(w_1)-j_1)+\dots+(\ell(w_r)-j_r)}_{\diff(\be^1),\dots,\diff(\be^r)}.\] 
\end{enumerate}
\end{example}

\begin{lemma}\label{l.HL}
Let $g(\underline{x})\in \cHRpp{p,q}^k_{e_1,\dots,e_r}$ (resp.\
$\cHRppw{p,q}{k}_{e_1,\dots,e_r}$). Then $g(\underline{x})y_1\dotsm y_l\in
\cHRpp{p,q}^{k+l}_{e_1,\dots,e_r,\bone^l}$ (resp.\
$\cHRppw{p,q}{k+l}_{e_1,\dots,e_r,\bone^l}$).
\end{lemma}

\begin{proof}
This follows from Corollary \ref{c.multiHL}.
\end{proof}

\subsection{Lorentzian polynomials and dually Lorentzian
polynomials}\label{s.6.2}

We first review the definitions of Lorentzian polynomials due to Br\"and\'en 
and Huh \cite{BH}, Lorentzian polynomials on cones due to Br\"and\'en and 
Leake \cite{BL}, and dually Lorentzian polynomials due to Ross, S\"u\ss, and 
Wannerer \cite{RSW}. 

\begin{defn}\label{l.Lor}
Let $f\in \R[x_1,\dots,x_n]$ be a homogeneous polynomial of degree $k$.
\begin{enumerate}
\item Let $K\subseteq \R^n$ be a convex cone. We say that $f$ is 
    \emph{$K$-Lorentzian} (resp.\ strictly $K$-Lorentzian) if for all 
    $\bv_1,\dots,\bv_k\in K\backslash\{0\}$, $D_{\bv_1}\dotsm D_{\bv_k}f\ge 
    0$ (resp.\ $>0$) and, if $k\ge 2$, the symmetric bilinear form on 
    $\R^n$ 
    \begin{equation}\label{e.Lordef}
(\bx,\by)\mapsto D_{\bx}D_{\by}D_{\bv_3}\dotsm D_{\bv_k}f
\end{equation}
has at most one positive eigenvalue (resp.\ and is nondegenerate). 
\item We say that $f$ is \emph{(strictly) Lorentzian} if it is (strictly)
    $\R_{\ge 0}^n$-Lorentzian.
\item \cite[Definition 1.1]{RSW} For $f$ of multidegree at most 
    $(\kappa_1,\dots,\kappa_n)$, we say that $f$ is \emph{dually 
    Lorentzian} if 
    \[f\spcheck_\kappa(x_1,\dots,x_n)=N(x_1^{\kappa_1}\dotsm x_n^{\kappa_n} f(x_1^{-1},\dots,x_n^{-1})) \]
    is Lorentzian, where $N$ is the $\R$-linear operator satisfying 
    $N(x_1^{\alpha_1}\dotsm x_n^{\alpha_n})=\frac{x_1^{\alpha_1}\dotsm 
    x_n^{\alpha_n}}{\alpha_1!\dotsm \alpha_n!}$. 
\end{enumerate}
\end{defn}

We have $\sL^k_n(K)\subseteq \rL^k_n(K)\subseteq \cS^k_{\bone^n}$, where 
$\rL^k_n(K)$ (resp.\ $\sL^k_n(K)$) denotes the set of $K$-Lorentzian (resp.\ 
strictly $K$-Lorentzian) polynomials in $\cS^k_{\bone^n}$. By continuity, 
$\rL^k_n(K)=\rL^k_n(\overline{K})\subseteq \cS^k_{\bone^n}$ is closed. 
Moreover, $\sL^k_n(K)=\emptyset$ if $K$ contains a line and $k\ge 1$. For 
$K\subseteq \{0\}$ and $k\ge 1$, $\sL^k_n(K)=\rL^k_n(K)=\cS^k_{\bone^n}$. 

Our definitions of Lorentzian polynomials and $C$-Lorentzian polynomials, 
where $C\subseteq \R^n$ is an open convex cone, are equivalent to the 
original ones by \cite[Remark 2.5, Proposition 8.2]{BL}. In particular, by 
\cite[Remark 2.5]{BL}, for $f$ nonzero and $C$-Lorentzian and 
$\bv_1,\dots,\bv_k\in C$, $D_{\bv_1}\dotsm D_{\bv_k}f> 0$ and, if $k\ge 2$, 
\eqref{e.Lordef} has exactly one positive eigenvalue. Our definition of 
strictly Lorentzian polynomials is equivalent to the original one by 
\cite[Lemma 6.1]{RSW} and Lemma \ref{l.sLor} below. Our definition of 
strictly $C$-Lorentzian polynomials follows \cite[Definition 3.4]{HX2} and 
differs from \cite[Definition 7.1]{RSW}. 

\begin{lemma}\label{l.sLor}
Let $K\subseteq \R^n$ be a closed convex cone. Then $\sL^k_n(K)$ is the 
interior of $\rL^k_n(K)\subseteq \cS^k_{\bone^n}$.
\end{lemma}

\begin{proof}
In Definition \ref{l.Lor}, we may restrict to the case where 
$\bv_1,\dots,\bv_k$ are unit vectors in $K$. Since the set of such vectors is 
compact, $\sL^k_n(K)\subseteq \cS^k_{\bone^n}$ is open. Thus $\sL^k_n(K)$ is 
contained in the interior $\rL^k_n(K)^\circ$ of $\rL^k_n(K)\subseteq 
\cS^k_{\bone^n}$. 

It remains to show $\rL^k_n(K)^\circ\subseteq \sL^k_n(K)$. We may assume that 
$K\backslash\{0\}$ is nonempty. For $k\ge 1$ and $\bv\in K\backslash \{0\}$, 
we have $D_{\bv}(\rL^k_n(K))\subseteq \rL^{k-1}_n(K)$. Since $D_{\bv} \colon 
\cS^k_{\bone^n}\to \cS^{k-1}_{\bone^n}$ is surjective and hence open, we have 
$D_{\bv}(\rL^k_n(K)^\circ )\subseteq \rL^{k-1}_n(K)^\circ$. Thus we may 
assume $k=0$ or $k=2$. The case $k=0$ is trivial. Let $f\in 
\rL^2_n(K)^\circ$. For $\bv_1,\bv_2\in K\backslash\{0\}$, $D_{\bv_1} 
D_{\bv_2} f>0$ by the previous case. In particular, $f$ has one positive 
eigenvalue. If $f$ is degenerate, then we may perturb it into $g\in 
\rL^2_n(K)^\circ$ with more than one positive eigenvalue, a contradiction. 
Thus $f$ is nondegenerate and $f\in \sL^2_n(K)$. 
\end{proof}

\begin{lemma}\label{l.BL26}
Let $A\colon \R^m\to \R^n$ be a linear map. Let $K'\subseteq \R^m$ and 
$K\subseteq \R^n$ be convex cones such that $A(K')\subseteq K\neq \{0\}$. If 
$f\in \R[x_1,\dots,x_n]$ is $K$-Lorentzian (resp.\ strictly $K$-Lorentzian 
and $A$ is an injection), then $fA\in \R[y_1,\dots,y_m]$ is $K'$-Lorentzian 
(resp.\ strictly $K'$-Lorentzian). 
\end{lemma}

\begin{proof}
As in \cite[Proposition 2.6]{BL}, by the formula $D_{\bv_1}\dotsm 
D_{\bv_k}(fA)=(D_{A\bv_1}\dotsm D_{A\bv_k}f)A$ for all $\bv_1,\dots,\bv_k\in 
\R^m$, we may assume that the degree $k$ of $f$ is $0$ or $2$. The case $k=0$ 
is trivial. In the case $k=2$, it suffices to observe that the positive index 
of inertia of $fA$ is at most the positive index of inertia of $f$ (resp.\ 
and any subspace of a nondegenerate quadratic space $V=(\R^n,f)$ containing a 
maximal positive definite subspace of $V$ is nondegenerate). 
\end{proof}

\begin{lemma}\label{l.RSW}
Let $f\in \R[x_1,\dots,x_m]$ be dually Lorentzian of degree~$k$. Let $K=C$ or 
$K=\overline{C}$, where $C\subseteq \R^n$ is an open convex cone. Let 
$\bv_1,\dots,\bv_m\in K$. Then $f(D_{\bv_1},\dots,D_{\bv_m})$ preserves 
$K$-Lorentzian polynomials. Moreover, if $\bv_1,\dots,\bv_m$ and $f$ are 
nonzero, then $f(D_{\bv_1},\dots,D_{\bv_m})$ sends strictly $K$-Lorentzian 
polynomials of degree $\ge k$ to strictly $K$-Lorentzian polynomials. 
\end{lemma}

\begin{proof}
For the preservation of $K$-Lorentzian polynomials, the case $K=C$ is 
\cite[Theorem 7.2]{RSW} and the case $K=\overline{C}$ follows by continuity. 
For the assertion on strictly $K$-Lorentzian polynomials, we again separate 
the two cases. 

Case $K=\overline{C}$. We may assume that $\overline{C}$ contains no line. 
Then the convex cone spanned by $\{\bv_1,\dots,\bv_m\}$ contains no line. 
Thus there exists $\bw\in \R^n$ such that $\bw^T \bv_i>0$ for all $1\le i\le 
m$. Let $A$ be the $n\times m$ matrix whose column vectors are 
$\bv_1,\dots,\bv_m$. Then $\bw^T A =(\bw^T \bv_1,\dots,\bw^T\bv_m)\in 
\R_{>0}^m$ and thus $f(A^T\bw)>0$. In particular, $fA^T\neq 0$. It follows 
that $f(D_{\bv_1},\dots,D_{\bv_m})=fA^T(\partial_1,\dots,\partial_n)\neq 0$. 
The assertion then follows from the fact that 
$f(D_{\bv_1},\dots,D_{\bv_m})\colon \cS^{d}_{\bone^n}\to \cS^{d-k}_{\bone^n}$ 
is a surjective linear map \cite[Lemma 6.3]{RSW} and hence open. 

Case $K=C$. Let $g\in \sL^d_n(C)$ with $d\ge k$. It suffices to show that 
$f(D_{\bv_1},\dots,D_{\bv_m})g$ is strictly $K'$-Lorentzian for all 
polyhedral cones $K'\subseteq C\cup \{0\}$. We may assume 
$\bv_1,\dots,\bv_m\in K'$. The assertion then follows from the previous case.
\end{proof}

\begin{defn}
Let $X$  be a smooth projective variety of dimension $k+n$, $\gamma\in 
H^{k,k}(X,\R)$, and $\xi_1,\dots,\xi_r\in \NS(X)_\R$. We define the 
\emph{generalized volume polynomials} to be 
\[
\vol_\gamma(\xi)=\frac{1}{n!}\int_X \gamma \xi^n,\quad \vol_{\gamma; \xi_1,\dots,\xi_r}(y_1,\dots,y_r)=\vol_\gamma(y_1\xi_1+\dots+y_r\xi_r)
\]
for $\xi\in\NS(X)_\R$ and $y_1,\dots,y_r\in \R$.
\end{defn}

\begin{theorem}\label{t.vLor}
Let $X$ be a smooth projective variety of dimension $k+n$. Let $\gamma\in 
H^{k,k}(X,\R)$. Assume that the following conditions hold: 
\begin{enumerate}
\item For every smooth closed subvariety $Z$ of $X$ of codimension $n$, we 
    have $\int_{Z} \gamma\ge 0$ (resp.\ $>0$). 
\item For every smooth closed subvariety $Z$ of $X$ of codimension $n-2$ 
    and every $h_Z\in \Amp(Z)$, we have $(\gamma|_{Z},h_Z)\in 
    \fHRw_{1,1}(Z)$ (resp.\ $\in \fHR_{1,1}(Z)$). 
\end{enumerate}
Then the polynomial
\[v(\xi) =\vol_{\gamma}(\xi)=\frac{1}{n!}\int_{X}\gamma \xi^n\]
on $\NS(X)_\R$ is $\Amp(X)$-Lorentzian (resp.\ strictly 
$\Amp(X)$-Lorentzian). 
\end{theorem}

\begin{proof}
By definition, we need to check the following conditions for all 
$\xi_1,\dots,\xi_n\in \Amp(X)$:
\begin{enumerate}[(i)]
\item $D_{\xi_1}\dotsm D_{\xi_n} v\ge 0$ (resp.\ $>0$). 
\item If $n\ge 2$, then the bilinear form $(\zeta,\zeta')\mapsto D_\zeta 
    D_{\zeta'} D_{\xi_3}\dotsm D_{\xi_{n}}v$ on $\NS(X)_\R$ has at most one 
    positive eigenvalue (resp.\ and is nondegenerate). 
\end{enumerate}
We have
\[ D_{\xi_1}\dotsm
    D_{\xi_{n}}v = \int_X \gamma \xi_1\dotsm \xi_n.\]
(i) follows from (a) and Corollary \ref{c.multiHL}. Assume $n\ge 2$. By (b) 
and Corollary \ref{c.multiHL}, $(\gamma \xi_3\dotsm \xi_{n},\xi)\in 
\fHRw_{1,1}(X)$ (resp.\ $\in \fHR_{1,1}(X)$) for all $\xi\in \Amp(X)$, which 
implies (ii). 
\end{proof}

\begin{cor}\label{c.vLord}
Notation and assumptions as in Theorem \ref{t.vLor}. Let 
$\zeta_1,\dots,\zeta_m\in \NS(X)_\R$ be nef (resp.\ ample) and let $0\neq 
f\in \R[x_1,\dots,x_m]$ be dually Lorentzian of degree $\le n$. Then 
$v(\xi)=\vol_{\gamma f(\zeta_1,\dots,\zeta_m)}(\xi)$  is $\Amp(X)$-Lorentzian 
(resp.\ strictly $\Amp(X)$-Lorentzian). 
\end{cor}

\begin{proof}
Since the polynomial in question is $f(D_{\zeta_1},\dots,D_{\zeta_m})v$, the 
assertion follows from Theorem \ref{t.vLor} and Lemma \ref{l.RSW}. 
\end{proof}

\begin{cor}\label{c.vLor0}
Let $X$ be a smooth projective variety of dimension $k+n$. Let $g\in 
\cHRw{1,1}{k}_{e_1,\dots,e_r}$ and let $E_1,\dots,E_r$ be nef $\R$-twisted 
vector bundles on $X$ of ranks $e_1,\dots,e_r$. Let $\zeta_1,\dots,\zeta_m\in 
\NS(X)_\R$ be nef and let $0\neq f\in \R[x_1,\dots,x_m]$ be dually Lorentzian 
of degree $l\le n$. Let 
\[v(\xi)=\vol_{g(E_1,\dots,E_r)f(\zeta_1,\dots,\zeta_m)}(\xi)=\frac{1}{(n-l)!}\int_X g(E_1,\dots,E_r)f(\zeta_1,\dots,\zeta_m)\xi^{n-l}.\] 
Then
\begin{enumerate}
\item $v(\xi)$ is an $\Amp(X)$-Lorentzian polynomial on $\NS(X)_\R$. 

\item If, moreover, $g\in\cHR{1,1}^{k}_{e_1,\dots,e_r}$ (resp.\ $g\in 
    \cHRp{1,1}^{k}_{e_1,\dots,e_r}$), $E_1,\dots,E_r$ are ample (resp.\ and 
    have the same $\R$-twist modulo $\NS(X)$), and $\zeta_1,\dots,\zeta_m$ 
    are ample, then $v(\xi)$ is strictly $\Amp(X)$-Lorentzian. 
\end{enumerate} 
\end{cor}

\begin{proof}
By Lemma \ref{l.triv}(e), we have $g\in \cP^k_{e_1,\dots,e_r}$. We apply 
Corollary \ref{c.vLord}. Assumption (a) of Theorem \ref{t.vLor} follows from 
Corollary \ref{c.posfinal} and assumption (b) holds by definition and Lemma 
\ref{l.alt}(a). 
\end{proof}

\begin{remark}
Assume that $X$ has maximal Picard number and $n=l+2$. Then, by (a) above, we 
have $(g(E_1,\dots,E_r)f(\zeta_1,\dots,\zeta_m),h)\in \fHRw_{1,1}(X)$ for all 
$h\in \NS(X)_\R$ nef. Moreover, for $g,E_1,\dots,E_r,\zeta_1,\dots,\zeta_m$ 
as in (b) above, we have $(g(E_1,\dots,E_r)f(\zeta_1,\dots,\zeta_m),h)\in 
\fHR_{1,1}(X)$ for $h\in \NS(X)_\R$ ample. 
\end{remark}

\begin{cor}\label{c.vLor}
Notation and assumptions as in Corollary \ref{c.vLor0}. Let 
$\xi_1,\dots,\xi_s\in \NS(X)_\R$ be nef classes. Let 
\[v=\vol_{g(E_1,\dots,E_r)f(\zeta_1,\dots,\zeta_m);\xi_1,\dots,\xi_s}=\frac{1}{(n-l)!}\int_X g(E_1,\dots,E_r)f(\zeta_1,\dots,\zeta_m)(y_1\xi_1+\dots+y_s\xi_s)^{n-l}.\]
Then
\begin{enumerate}
\item $v(y_1,\dots,y_s)$ is Lorentzian.

\item If, moreover, $g\in\cHR{1,1}^{k}_{e_1,\dots,e_r}$ (resp.\ $g\in 
    \cHRp{1,1}^{k}_{e_1,\dots,e_r}$), $E_1,\dots,E_r$ are ample (resp.\ and 
    have the same $\R$-twist modulo $\NS(X)$),
    $\zeta_1,\dots,\zeta_m,\xi_1,\dots,\xi_s$ are ample, and 
    $\xi_1,\dots,\xi_s$ are $\R$-linearly independent, then 
    $v(y_1,\dots,y_s)$ is strictly Lorentzian. 
\end{enumerate}
\end{cor}

\begin{proof}
This follows from Lemma \ref{l.BL26} and Corollary \ref{c.vLor0}.
\end{proof}

\begin{lemma}\label{l.vol}
Let $\kappa\in \N^r$. For each $1\le i\le r$, let $X_i$ be a smooth 
projective variety of dimension $\kappa_i$ and let $\zeta_i\in \NS(X_i)_\R$ 
such that $\int_{X_i}\zeta_i^{\kappa_i}=1$ (for example $X_i=\PP^{\kappa_i}$ 
and $\zeta_i=\cO_{\PP^{\kappa_i}}(1)$). Let $X=X_1\times \dotsm \times X_r$ 
and let $\xi_i=\pi_i^*\zeta_i$, where $\pi_i\colon X\to X_i$ denotes the 
$i$-th projection. For any polynomial $g\in \R[y_1,\dots,y_r]$ homogeneous of 
degree $k\le \lvert \kappa\rvert$, we have 
\[\vol_{g(\xi_1,\dots,\xi_r);\xi_1,\dots,\xi_r}=(g_{\le \kappa})\spcheck_\kappa,\]
where $g_{\le \kappa}$ denotes the truncation of $g$. 
\end{lemma}

\begin{proof}
Let $v=\vol_{1;\xi_1,\dots,\xi_r}=\frac{x_1^{\kappa_1}\dotsm x_r^{\kappa_r} 
}{\kappa!}$. Then
\[\vol_{g(\xi_1,\dots,\xi_r);\xi_1,\dots,\xi_r}=\partial_g v=\partial_g \frac{x_1^{\kappa_1}\dotsm 
x_r^{\kappa_r}
}{\kappa!}= (g_{\le \kappa})\spcheck_\kappa.
\]
\end{proof}

Let $\cL^k_r$ denote the set of dually Lorentzian polynomials of $r$ 
variables and degree $k$. We have seen that $\cHRw{1,1}{k}_{\bone^r}$ can be 
defined without $\R$-twists (Lemmas \ref{l.alt}(b), \ref{l.extract}(b)). We 
are going to see that the same holds for $\cHR{1,1}^k_{\bone^r}$. Let 
$\cHRZ{1,1}^k_{\bone^r}$ denote the variant of $\cHR{1,1}^k_{\bone^r}$ 
defined without $\R$-twists. More precisely, $\cHRZ{1,1}^k_{\bone^r}$ 
consists of $g\in \R[x_1,\dots,x_r]$ homogeneous of degree $k$ such that for 
every smooth projective variety $X$ of dimension $d=k+2$, for all ample 
classes $\xi_1,\dots,\xi_r,h\in \NS(X)$, we have $(g(\xi_1,\dots,\xi_r),h)\in 
\fHR_{1,1}(X)$. 

\begin{cor}\label{c.dLor}
We have
\[\cL^k_r=\cHR{1,1}^k_{\bone^r}\cup
\{0\}=\cHRw{1,1}{k}_{\bone^r},\quad
\cHR{1,1}^k_{\bone^r}=\cHRZ{1,1}^k_{\bone^r}.
\]
In particular, a homogeneous polynomial $g\in \R[x_1,\dots,x_r]$ of degree 
$k$ is dually Lorentzian if and only if for every smooth projective variety 
$X$ of dimension $k+2$, for every $\xi'\in H^{1,1}(X,\R)$, and for all ample 
classes $\xi,\xi_1,\dots,\xi_r\in \NS(X)$, we have $\int_X \xi^2 
g(\xi_1,\dots, \xi_r)\ge 0$ and 
\[\left(\int_X \xi\xi'g(\xi_1,\dots,\xi_r)\right)^2\ge \left(\int_X 
\xi^2g(\xi_1,\dots,\xi_r)\right)\left(\int_X \xi'^2g(\xi_1,\dots,\xi_r)\right).\] 
\end{cor}

\begin{proof}
By \cite[Theorem 1.4]{RSW}, $\cL^k_r\backslash\{0\}\subseteq 
\cHR{1,1}^k_{\bone^r}$. By definition, $\cHR{1,1}^k_{\bone^r}\subseteq 
\cHRZ{1,1}^k_{\bone^r}\subseteq \cHRw{1,1}{k}_{\bone^r}\backslash\{0\}$. Next 
we show $\cHRw{1,1}{k}_{\bone^r}\subseteq \cL^k_r$. Let $g\in 
\cHRw{1,1}{k}_{\bone^r}$. In the notation of Lemma \ref{l.vol}, for 
$\kappa_1,\dots,\kappa_r$ large enough, 
$g\spcheck_\kappa=\vol_{g(\xi_1,\dots,\xi_r);\xi_1,\dots,\xi_r}$, with 
$\xi_1,\dots,\xi_r\in \NS(X)$ nef. Thus, by Corollary \ref{c.vLor}, 
$g\spcheck_\kappa$ is Lorentzian. In other words, $g$ is dually Lorentzian. 
Therefore, $\cL^k_r\backslash\{0\}= \cHR{1,1}^k_{\bone^r}= 
\cHRZ{1,1}^k_{\bone^r}=\cHRw{1,1}{k}_{\bone^r}\backslash\{0\}$. The last 
assertion is a restatement of the equality $\cL^k_r=\cHRw{1,1}{k}_{\bone^r}$, 
by Lemmas \ref{l.ineq}, \ref{l.alt}(b), and \ref{l.extract}(b).
\end{proof}

\begin{cor}\label{c.HRdLor}
For $\min(p,q)=1$, every polynomial $g\in \cHRpw{p,q}{k}_{e_1,\dots,e_r}$ is
dually Lorentzian.
\end{cor}

\begin{proof}
Since $\cHRpw{p,q}{k}_{e_1,\dots,e_r}\subseteq 
\cHRpw{1,1}{k}_{e_1,\dots,e_r}$ (Lemma \ref{l.triv}(g)), we may assume 
$(p,q)=(1,1)$. Let $e=e_1+\dots+e_r$. By Remark \ref{r.triv}(c), the 
inclusion map $\cS^k_{e_1,\dots,e_r}\to \cS^k_{\bone^e}$ carries 
$\cHRpw{1,1}{k}_{e_1,\dots,e_r}$ into 
$\cHRpw{1,1}{k}_{\bone^{e}}=\cHRw{1,1}{k}_{\bone^{e}}$ (Lemma 
\ref{l.extract}(b)). We conclude by Corollary \ref{c.dLor}. 
\end{proof}

Combining Corollary \ref{c.HRdLor} and Example \ref{e.Sch}(b), we recover the 
theorem of Huh, Matherne, M\'esz\'aros, and St.~Dizier that Schubert 
polynomials are dually Lorentzian \cite[Theorem~6]{HMMD}. The geometric 
constructions used in the two proofs are related but different. In fact, the 
proof of \cite[Theorem~6]{HMMD} yields the stronger result that Schubert 
polynomials are covolume polynomials in the sense of \cite[Definition 
2.1]{Aluffi}.

\begin{proof}[Proof of Theorem \ref{t.iLor}]
(a) This is part of Corollary \ref{c.dLor}.

(b) The ``if'' and ``moreover'' parts are special cases of Corollary 
\ref{c.vLor}. For the ``only if'' part, we may assume that $f$ is a nonzero 
Lorentzian polynomial. Then $f=g\spcheck_\kappa$ for some $\kappa\in \N^r$ 
with $g=g_{\le \kappa}$ and $g\in \cL^{\lvert \kappa\rvert 
-k}_r\backslash\{0\}=\cHR{1,1}^{\lvert \kappa\rvert -k}_{\bone^r}$. By Lemma 
\ref{l.vol}, $g\spcheck_\kappa=\vol_{g(\xi_1,\dots,\xi_r);\xi_1,\dots,\xi_r}$ 
for $X=\PP^{\kappa_1}\times \dotsm \times \PP^{\kappa_r}$ and 
$\xi_1,\dots,\xi_r\in\NS(X)$ nef. 
\end{proof}

The following characterization of dually Lorentzian polynomials was suggested 
to us by Yiran Lin. Let $\Herm_d$ denote the space of Hermitian $d\times d$ 
matrices and let $\det\colon \Herm_d\to \R$ denote the determinant. Given a 
homogeneous polynomial $g\in \R[x_1,\dots,x_{n}]$ of degree $d-2$ and 
$H,H',H_1,\dots,H_n\in \Herm_d$, the \emph{generalized mixed discriminant} is 
defined to be 
\[\MD(H,H',g(H_1,\dots,H_n))\colonequals 
\frac{1}{d!}D_{H}D_{H'}g(D_{H_1},\dots,D_{H_n})\det.\] 

\begin{cor}\label{c.mdisc}
Let $g\in \R[x_1,\dots,x_{n}]$ be a nonzero homogeneous polynomial of degree 
$d-2$. Then $g$ is dually Lorentzian if and only if for all 
$H,H',H_1,\dots,H_n\in \Herm_d$ with $H,H_1,\dots,H_n$ positive definite and 
$H'\notin \R H$, we have $\MD(H,H,g(H_1,\dots,H_n))> 0$ and
\[\MD(H,H',g(H_1,\dots,H_n))^2>\MD(H,H,g(H_1,\dots,H_n))\MD(H',H',g(H_1,\dots,H_n)).\]
\end{cor}

\begin{proof}
The ``only if'' part is essentially \cite[Theorem 8.2]{RSW}. Let us briefly 
recall the argument. Let $\Herm_d^+\subseteq \Herm_d$ denote the cone of 
positive definite matrices. By Alexandrov's inequality \cite{Alexandrov}, 
$\det$ is strictly $\Herm_d^+$-Lorentzian. Thus, by Lemma \ref{l.RSW}, 
$g(D_{H_1},\dots,D_{H_n})\det$ is strictly $\Herm_d^+$-Lorentzian and the 
inequalities follow. 

The ``if'' part is a consequence of Corollary \ref{c.dLor}. There exists a 
bijection between $\Herm_d$ and $\Lambda^{1,1}_\R(\C^d)$, with $\Herm_d^+$ 
corresponding to the positive cone $C\subseteq \Lambda^{1,1}_\R(\C^d)$, such 
that, for $H\in \Herm_d$ corresponding to $\omega\in \Lambda^{1,1}_\R(\C^d)$, 
we have $\omega^d=\det(H)\vol$, where $\vol$ is a fixed volume form. The 
inequalities imply that for $\omega,\omega_1,\dots,\omega_n\in C$, we have 
$\omega^2 g(\omega_1,\dots,\omega_n)/\vol>0$ and the bilinear form 
$(\alpha,\beta)\mapsto \alpha \beta g(\omega_1,\dots,\omega_n)/\vol$ on 
$\Lambda^{1,1}_\R(\C^d)$ has exactly one positive eigenvalue and is 
nondegenerate. Thus, by a pointwise to global argument (\cite[Proposition 
5.5]{RT3} or \cite[Theorem 1.1]{DN2}), for every compact K\"ahler manifold 
$X$ and for all K\"ahler classes $\xi_1,\dots,\xi_n\in H^{1,1}(X,\R)$, 
$g(\xi_1,\dots,\xi_n)$ satisfies the Hodge--Riemann property on $H^{1,1}(X)$. 
Therefore, $g$ is dually Lorentzian by Corollary \ref{c.dLor}.  
\end{proof}

\subsection{Derivatives}\label{s.6.3}
For $g\in \cS_{e_1,\dots,e_r}$, we define
\[\partial_i g=\left.\frac{d}{dt} g(x_{1,1},\dots,x_{1,e_1};\dots ;x_{i,1}+t,\dots,x_{i,e_i}+t;\dots;x_{r,1},\dots,x_{r,e_r})\right|_{t=0}.\]

\begin{prop}\label{p.der}
Let $g\in \cS^k_{e_1,\dots,e_r}$. Let $Y$ be a smooth projective variety of 
dimension $n+l$, where $n\le k$. Let $X$ be a smooth projective variety of 
dimension $d$  and let $\pi_X\colon X\times Y \to X$ and $\pi_Y\colon X\times 
Y\to Y$ be the projections. Let $p+q+k=d+n$. Assume $H^{p-2,q-2}(X)=0$. Let 
$\gamma\in H^{l,l}(Y,\R)$, $\xi_1,\dots,\xi_r\in \NS(Y)_\R$. Let 
$v=\vol_{\gamma;\xi_1,\dots,\xi_r}$ and let 
$\partial_v=v(\partial_1,\dots,\partial_r)$. Let $E_1,\dots,E_r$ be 
$\R$-twisted vector bundles on $X$ of ranks $e_1,\dots,e_r$, respectively. 
Let $F_i=E_i\boxtimes \cO_{Y}(\xi_i)$ on $X\times Y$. Let $h\in 
H^{1,1}(X,\R)$. Assume that the following conditions hold. 
\begin{enumerate}[(i)]
\item $\langle -,-\rangle_{g(F_1,\dots,F_r)\pi_Y^*\gamma\pi_X^*h^2}$ is 
    positive definite on $H^{p-1,q-1}(X\times Y)$. 
\item $(g(F_1,\dots,F_r)\pi_Y^*\gamma,\pi_X^*h)\in \fHR_{p,q}(X\times Y)$ 
    (resp.\ $\in \fHRw_{p,q}(X\times Y)$).
\end{enumerate}
Then 
\begin{enumerate}
\item $((\partial_vg)(E_1,\dots,E_r),h)\in \fHR_{p,q}(X)$ (resp.\ $\in
    \fHRw_{p,q}(X)$).
\item $((\partial_v g)(E_1,\dots,E_r),g(E_1,\dots,E_r))\in \HR_{p,q}(X)$ 
    (resp.\ $\in \HRw_{p,q}(X)$) if $n=1$ (resp.\ and if 
    $\langle-,-\rangle_{g(E_1,\dots,E_r)h}$ is nondegenerate on 
    $H^{p-1,q-1}(X)$). 
\end{enumerate}
\end{prop}

\begin{proof}
We have 
\begin{gather*}
g(F_1,\dots,F_r)=\sum_{i\in\N^r}\frac{1}{i!}\pi_Y^*(\xi_1^{i_1}\dotsm \xi_r^{i_r})\pi_X^*(\partial_1^{i_1}\dotsm \partial_r^{i_r} g)(E_1,\dots,E_r),\\
\partial_v g =\sum_{\substack{i\in \N^r\\\lvert i\rvert=n}}\frac{\partial^iv}{i!}\partial_1^{i_1}\dotsm \partial_r^{i_r} g,\quad \partial^i v=\int_Y \gamma \xi_1^{i_1}\dotsm \xi_r^{i_r}.
\end{gather*}
Thus
\begin{equation}\label{e.der}
\pi_{X*} (g(F_1,\dots,F_r)\pi_Y^*\gamma) = (\partial_v g)(E_1,\dots,E_r).
\end{equation} 
By (ii) (resp.\ and (i)) and Lemma \ref{l.funcprod}, it follows that 
$((\partial_v g)(E_1,\dots,E_r),h)\in \fHR_{p,q}(X)$ (resp.\ $\in 
\fHRw_{p,q}(X)$). 

Assume $n=1$. If $v=0$ in the case of $\HR_{p,q}$, then $H^{p,q}(X)=0$ by (a) 
and (b) follows trivially.  Thus we may assume $v\neq 0$ in both cases. Note 
that $v=b_1y_1+\dots+b_ry_r$, where $b_i=\int_Y \gamma \xi_i$. Choose $i$ 
such that $b_i\neq 0$. Then $\xi_i \neq 0$. Consider the subspace 
$\pi_X^*H^{p,q}(X)\oplus \pi_X^*H^{p-1,q-1}(X)\pi_Y^*\xi_i$ of 
$H^{p,q}(X\times Y)$. The restriction of $\langle 
-,-\rangle_{g(F_1,\dots,F_r)\pi_Y^*\gamma}\in \Herm(H^{p,q}(X\times Y))$ to 
this subspace has the form 
\[G=\begin{pmatrix}
H&\Phi^*\\\Phi& 0
\end{pmatrix},\]
where
\begin{gather*}
H=\langle-,-\rangle_{(\partial_v g)(E_1,\dots,E_r)}\in \Herm(H^{p,q}(X)),\\
\Phi=\epsilon (-,-)_{b_i g(E_1,\dots,E_r)}\in \Sesq(H^{p-1,q-1}(X),H^{p,q}(X))
\end{gather*}
for some $\epsilon$ satisfying $\epsilon^4=1$. By (i), $H$ is negative 
definite on $H^{p-1,q-1}(X)h$. By (ii) (resp.\ and (i)) and the proof of 
Lemma \ref{l.funcprod}, $G$ is positive definite (resp.\ positive 
semidefinite) on $(\pi_X^*H^{p-1,q-1}(X) h)^\perp_G$. We conclude by Lemma 
\ref{l.block}. 
\end{proof}

\begin{remark}\label{r.der}
Condition (i) of Proposition \ref{p.der} is satisfied if $h\in \Amp(X)$ and
\begin{enumerate}
\item[(i$'$)] For every irreducible smooth ample divisor $Z$ of every 
    irreducible smooth ample divisor of $X$, $\langle 
    -,-\rangle_{(g(F_1,\dots,F_r)\pi_Y^*\gamma)|_{Z\times Y}}$ is positive 
    definite on $H^{p-1,q-1}(Z\times Y)$. 
\end{enumerate}
This follows from Bertini's theorem. Indeed, since $(p-1)+(q-1)\le d-2$, the 
map $\iota^*\colon H^{p-1,q-1}(X\times Y)\to H^{p-1,q-1}(Z\times Y)$ is an 
injection by the K\"unneth formula and the Lefschetz hyperplane theorem. Here 
$\iota\colon Z\times Y\to X\times Y$ denotes the embedding. 
\end{remark}

\begin{defn}\label{d.volume}
We let $\rV^{n}_r$ (resp.\ $\rV^{n,p,q}_r$) denote the set of volume 
polynomials of the form $\vol_{1;\xi_1,\dots,\xi_r}$, where 
$\xi_1,\dots,\xi_r\in \NS(Y)_\R$ are ample and $Y$ is a smooth projective 
variety of dimension~$n$ (resp.\ such that $H^{p',q'}(Y)=0$ for all $(p',q')$ 
satisfying $p'\le p$, $q'\le q$, and $p'\neq q'$). Such a polynomial is said 
to be \emph{rationally congruent} if one can choose the classes 
$\xi_1,\dots,\xi_r\in \NS(Y)_\R$ such that $\xi_i-\xi_j\in \NS(Y)_\Q$ for all 
$1\le i,j\le r$. 

We let $\overline{\rV}^{n,p,q}_r$ denote the closure of $\rV^{n,p,q}_r$ in 
$\cS^{n}_{\bone^r}$. 
\end{defn}

For $p<0$ or $q<0$, $\rV^{n,p,q}_r=\rV^n_r$. Recall that 
$\overline{\rV}^{n,p,q}_r$ is contained in the set $\rL^n_r$ of Lorentzian 
polynomials (Corollary \ref{c.vLor}) of degree $n$ in $r$ variables. As an 
example, the volume polynomial of any $r$-tuple of convex bodies in $\R^n$ 
belongs to $\overline{\rV}^{n,p,q}_r$ \cite[Section 5.4]{FultonToric}. 

\begin{example}\label{ex.Polya}
Assume $n\ge 1$. We have $\overline{\rV}^{n,p,q}_2=\rL^n_2$. A bivariate 
polynomial $\sum_{i=0}^n a_i x^i y^{n-i}$ belongs to 
$\overline{\rV}^{n,p,q}_2$ if and only if $a_0,\dots,a_n$ is an ultra 
log-concave sequence of nonnegative real numbers without internal zeroes.  
Recall that a sequence $(a_i)_{0\le i\le n}$ of nonnegative real numbers is 
called \emph{ultra log-concave} if $(a_i/\binom{n}{i})_{0\le i\le n}$ is 
log-concave. See \cite[Example 2.26, Section 4.2]{BH}. 

If $a_0,\dots,a_n$ is a P\'olya frequency sequence with $a_0a_n>0$, then 
$v(x,y)=\sum_{i=0}^n a_i x^i y^{n-i}\in \rV^{n,p,q}_2$. Indeed, by the 
Aissen--Schoenberg--Whitney theorem \cite{ASW}, there exist positive real 
numbers $t_1,\dots, t_n$ such that $v(x,y)=a_0\prod_{i=1}^n (y+t_ix)$, which 
is the volume polynomial of $\cO(bt_1)\boxtimes \dots \boxtimes \cO(bt_n)$ 
and  $\cO(b)^{\boxtimes n}$ on $(\PP^1)^n$, where $b=a_0^{1/n}$. 
\end{example}

\begin{theorem}\label{t.cHR}
Let $k\ge n$ and $p,q\ge 0$. Let $v\in \rV^{n,p-2,q-2}_r$ and let 
$\partial_v=v(\partial_1,\dots,\partial_r)$. Then 
\begin{enumerate}
\item $\partial_v(\cHRww{p,q}{k}_{e_1,\dots,e_r})\subseteq
    \cHRww{p,q}{k-n}_{e_1,\dots,e_r}$.
\item $\partial_v(\cHRpww{p,q}{k}_{e_1,\dots,e_r})\subseteq 
    \cHRpww{p,q}{k-n}_{e_1,\dots,e_r}$ if $v$ is rationally congruent. 
\end{enumerate}
\end{theorem}

Theorem \ref{t.intro8} is a special case of (a). 

\begin{proof}
We apply Proposition \ref{p.der} with $l=0$ and $\gamma=1$ in each case. By 
the K\"unneth formula, $H^{p-2,q-2}(X\times Y)=0$. If $v$ is rationally 
congruent, we may assume that $\xi_i-\xi_j\in \NS(Y)$, which implies that 
$F_i$ and $F_j$ have the same $\R$-twist modulo $\NS(X\times Y)$. We apply 
Remark \ref{r.der} to check condition (i). We may assume $p=q\le 1$ or $p\neq 
q$. Then $\cHRpw{p,q}{k}_{e_1,\dots,e_r}\subseteq \cP^{k}_{e_1,\dots,e_r}$. 
We may assume $g\neq 0$. Then condition (i$'$) holds by Corollary 
\ref{c.posfinal}. Indeed, by the K\"unneth formula and the Lefschetz 
hyperplane theorem, $H^{p-2,q-2}(Z\times Y)=0$ for every $Z$ as in (i$'$). 
Condition (ii) holds by the assumption and Remark \ref{r.changeh} (resp.\ 
Lemma \ref{l.alt}(a)). 
\end{proof}

\begin{cor}\label{c.cHR}
Let $D=a_1\partial_1+\dots+a_r\partial_r$ with $a_1,\dots,a_r\in \R_{> 0}$.
Let $k\ge 1$ and $p,q\ge 0$.
\begin{enumerate}
\item $D(\cHR{p,q}^k_{e_1,\dots,e_r})\subseteq
    \cHR{p,q}^{k-1}_{e_1,\dots,e_r}$.
\item $D(\cHRp{p,q}^k_{e_1,\dots,e_r})\subseteq 
    \cHRp{p,q}^{k-1}_{e_1,\dots,e_r}$ if $\{a_i-a_j\mid 1\le i,j\le r\}$ 
    spans a $\Q$-vector space of dimension $\le 1$. 
\end{enumerate}
\end{cor}

The assumption in (b) is automatic if $r\le 2$.

\begin{proof}
This follows from Theorem \ref{t.cHR} applied to $v=a_1y_1+\dots+a_ry_r$, 
which is the volume polynomial of the $r$ ample $\R$-twisted line bundles 
$\cO_{\PP^1}(a_1),\dots,\cO_{\PP^1}(a_r)$ on $\PP^1$. In (b), we have 
$\{a_i-a_j\mid 1\le i,j\le r\}\subseteq a\Q$ for some $a\in \R_{>0}$, and 
$a^{-1}v$ is rationally congruent. 
\end{proof}

\begin{cor}\label{c.partial}
Let $k\ge n$, $p,q\ge 0$, $v\in \overline{\rV}^{n,p-2,q-2}_r$. Let
$\partial_v=v(\partial_1,\dots,\partial_r)$. We have
\begin{gather*}
\partial_v(\cHRppw{p,q}{k}_{e_1,\dots,e_r})\subseteq
\cHRppw{p,q}{k-n}_{e_1,\dots,e_r}.
\end{gather*}
\end{cor}

This applies in particular to $a_1\partial_1+\dots + a_r\partial_r$ for
$a_1,\dots,a_r\in \R_{\ge 0}$.

\begin{proof}
This follows from Theorem \ref{t.cHR} by continuity.
\end{proof}

\begin{remark}\label{r.pHR}
Let $X$ be a smooth projective variety of dimension $d=p+q+k-1$ such that 
$H^{p-2,q-2}(X)=0$ and let $E_1,\dots,E_r$ be ample $\R$-twisted vector 
bundles on $X$ of ranks $e_1,\dots,e_r$ (resp.\ having the same $\R$-twist 
modulo $\NS(X)$). For $D$ as in Corollary \ref{c.cHR} (resp.\ satisfying the 
condition in (b)) and $g\in \cHR{p,q}^k_{e_1,\dots,e_r}$ (resp.\ $g\in 
    \cHRp{p,q}^k_{e_1,\dots,e_r}$), we have 
\[((Dg)(E_1,\dots,E_r),g(E_1,\dots,E_r))\in \HR_{p,q}(X)\] 
by Proposition \ref{p.der} and the proofs of Theorem \ref{t.cHR} and 
Corollary \ref{c.cHR}. In particular, for any $j\ge 0$, any partition 
$\lambda$ satisfying $\lvert \lambda\rvert=k+j$, and any ample $\R$-twisted 
vector bundle $E$ on $X$ of rank $e\ge \lambda_1$, we have 
$(s_\lambda^{[j+1]}(E),s_\lambda^{[j]}(E))\in \HR_{p,q}(X)$. For $p=q=1$, we 
recover \cite[Theorem 3.2]{RT} and \cite[Theorem 10.2]{RT2} by Lemma 
\ref{l.block}. 

For $D=a_1\partial_1+\dots+a_r\partial_r$ with $a_i\ge 0$ for all $1\le i \le 
r$ and $g\in \cHRw{p,q}{k}_{e_1,\dots,e_r}$ (resp.\ $g\in 
    \cHRpw{p,q}{k}_{e_1,\dots,e_r}$), we have
\[((Dg)(E_1,\dots,E_r),g(E_1,\dots,E_r))\in \HRw_{p,q}(X).\] 
Indeed, we may assume that $D$ satisfies the assumption of Corollary 
\ref{c.cHR} and it suffices to apply Proposition \ref{p.der} and the proofs 
of Theorem \ref{t.cHR} and Corollary \ref{c.cHR}. 
\end{remark}

\begin{example}\label{ex.Polya2} Let $\lambda$ and $\mu$ be partitions and let $0\le n\le \lvert \lambda \rvert+\lvert \mu\rvert$.
\begin{enumerate}
\item Let $b_0,\dots,b_n$ be a log-concave sequence of nonnegative real 
    numbers without internal zeroes. Then $\sum_{i=0}^n b_i 
    s_\lambda^{[i]}(x_1,\dots,x_e)s_\mu^{[n-i]}(y_1,\dots,y_{e'})$ belongs 
    to the closure of $\cHR{p,q}^{\lvert \lambda \rvert +\lvert 
    \mu\rvert-n}_{e,e'}$. Indeed, we may assume $\lambda_1\le e$ and 
    $\mu_1\le e'$. In this case, the assertion follows from Theorem 
    \ref{t.cHR} and Example \ref{ex.Polya} applied to $v=\sum_{i=0}^n b_i 
    \binom{n}{i}x^iy^{n-i}$ and $s_\lambda(\underline x) s_\mu(\underline 
    y)\in \cHR{p,q}^{\lvert \lambda \rvert +\lvert \mu\rvert}_{e,e'}$ 
    (Example \ref{e.Sch}(a)). 
\item Let $a_0,\dots,a_n$ be a P\'olya frequency sequence such that $a_j>0$ 
    for some $j$ satisfying $n-\lvert \mu \rvert\le j\le \lvert 
    \lambda\rvert$. Then, for $e\ge \lambda_1$ and $e'\ge \mu_1$, we have 
    \[\sum_{i=0}^n \frac{a_i}{\binom{n}{i}} s_\lambda^{[i]}(x_1,\dots,x_e)s_\mu^{[n-i]}(y_1,\dots,y_{e'})\in \cHR{p,q}^{\lvert \lambda \rvert 
    +\lvert \mu\rvert-n}_{e,e'}
    \] 
    by Theorem \ref{t.cHR} and Example \ref{ex.Polya} applied to 
    $v=\sum_{i=0}^{m'-m} a_{i+m}x^iy^{m'-m-i}$ and 
    $s_\lambda^{[m]}(\underline x) s_\mu^{[n-m']}(\underline y)\in 
    \cHR{p,q}^{\lvert \lambda \rvert +\lvert \mu\rvert-n+m'-m}_{e,e'}$ 
    (Example \ref{e.Sch}(a)), where $a_m$ and $a_{m'}$ are the first and 
    last nonzero terms of the sequence $(a_i)$, respectively. 
\end{enumerate}
\end{example}

Taking $\mu=\bone^n$ and $e'=1$ in (b), we get $\sum_{i=0}^n a_iy^i 
s_\lambda^{[i]}(x_1,\dots,x_e)\in \cHR{p,q}^{\lvert \lambda \rvert}_{e,1}$, 
which implies Theorem \ref{t.Polya}. Next we proceed to give a couple of 
refinements of Example \ref{ex.Polya2} in the case $e'=1$. 

\begin{cor}\label{c.Polya}
Let $g\in \cHR{p,q}^k_{e_1,\dots,e_r}$, $l\ge 0$, $0\le j\le k+l$. Then
\begin{equation}\label{e.Polya}
\sum_{i=0}^j g^{[i]}(\ux)c_{l-j+i}(y_1,\dots,y_l)\in \cHR{p,q}^{k +l-j}_{e_1,\dots,e_r,\bone^l}.
\end{equation}
\end{cor}

Recall that $c_n$ denotes the elementary symmetric polynomial of degree $n$.

\begin{proof}
By Lemma \ref{l.HL},
\[f=g(\ux)c_l(y_1,\dots,y_l)=g(\ux)y_1\dots
y_l\in\cHR{p,q}^{k +l}_{e_1,\dots,e_r,\bone^l}.
\]
The polynomial in \eqref{e.Polya} equals $f^{[j]}$, which belongs to
$\cHR{p,q}^{k +l-j}_{e_1,\dots,e_r,\bone^l}$ by Corollary \ref{c.cHR}.
\end{proof}

\begin{cor}\label{c.Polya2}
Let $g\in \cHR{p,q}^k_{e_1,\dots,e_r}$ and let $a_m,\dots,a_n$ (where $m\in 
\Z$, $n\in \N$) be a P\'olya frequency sequence such that $a_j> 0$ for some 
$0\le j\le k$. Then 
\begin{equation}\label{e.Polya2}
\sum_{i=m}^n a_i y^{i-m} g^{[i]}(\ux)\in \cHR{p,q}^{k-m}_{e_1,\dots,e_r,1}.
\end{equation}
\end{cor}

By convention, $g^{[i]}=0$ for $i<0$. 

\begin{proof}
Let $s\in [m,n]$ and $l\ge 0$ be such that $a_s> 0$, $a_{s+l}> 0$, and 
$a_{i}=0$ for $i\notin [s,s+l]$. By assumption, $s\le k$ and $s+l\ge 0$. By 
the Aissen--Schoenberg--Whitney theorem \cite{ASW}, there exist positive real 
numbers $t_1,\dots, t_l$ such that 
\[\sum_{i=m}^n a_i y^i=a_sy^s\prod_{i=1}^l (1+t_iy).\]
Thus $a_iy^i=a_s y^s c_{i-s}(t_1y,\dots,t_ly)$. The polynomial in 
\eqref{e.Polya2} equals 
\[a_sy^{s-m}\sum_{i=0}^{s+l} g^{[i]}(\ux)c_{i-s}(t_1y,\dots,t_ly),\]
which belongs to $\cHR{p,q}^{k-m}_{e_1,\dots,e_r,1}$ by Corollary 
\ref{c.Polya} and Lemma \ref{l.HL}. 
\end{proof}

\begin{cor}
Let $g\in \cHRppw{p,q}{k}_{e_1,\dots,e_r}$ and let $a_0,\dots,a_n$ be a 
log-concave sequence of nonnegative real numbers without internal zeroes. 
Assume either $\min(p,q)\le 1$ or $(i!a_i)_{0\le i\le n}$ is log-concave. 
Then 
\begin{equation}\label{e.Polya3}
\sum_{i=0}^n a_i y^{i} g^{[i]}(\ux)\in \cHRppw{p,q}{k}_{e_1,\dots,e_r,1}.
\end{equation}
\end{cor}

\begin{proof}
Let $b_0,\dots,b_m$ be an ultra log-concave sequence of nonnegative real 
numbers without internal zeroes with $m\ge 1$ and let $e\ge 1$. By Lemma 
\ref{l.HL}, $g(\ux)y^{m+e-1}\in \cHRppw{p,q}{k+m+e-1}_{e_1,\dots,e_r,1}$. For 
an $\R$-twisted vector bundle $E$ of rank $e$ on $X$ with projection 
$\pi\colon \PP(E)\to X$, $\pi_*(\zeta^{m+e-1})=s_{\bone^m}(E)$, where 
$\zeta=c_1(\cO_{\PP(E)}(1))$. Thus, by Lemma \ref{l.func2}, 
$f=g(\ux)s_{\bone^{m}}(y_1,\dots,y_e)\in 
\cHRppw{p,q}{k+m}_{e_1,\dots,e_r,e}$. By Example \ref{ex.Polya}, 
\[v(u_1,\dots,u_r,w)=\sum_{i=0}^m b_i(u_1+\dots+u_r)^iw^{m-i}\in \overline{\rV}^{m,p-2,q-2}_{r+1}.\]
Thus, by Corollary \ref{c.partial}, 
\begin{equation}\label{e.Polyapf}
\sum_{i=0}^m \frac{i!}{(e-1+i)!}b_ig^{[i]}(\ux)s_{\bone^i}(\uy)=\frac{1}{(e-1+m)!}\partial_v(f)\in \cHRppw{p,q}{k}_{e_1,\dots,e_r,e}.
\end{equation}

Case $(i!a_i)_{0\le i\le n}$ log-concave. Taking $e=1$ and 
$b_i=\frac{1}{m^i}\binom{m}{i}i!a_i$ (where by convention $a_i=0$ for $i>n$) 
in \eqref{e.Polyapf}, we get $\sum_{i=0}^m 
\frac{m!}{m^i(m-i)!}a_ig^{[i]}(\ux)y^i \in 
\cHRppw{p,q}{k}_{e_1,\dots,e_r,1}$, which implies \eqref{e.Polya3} by letting 
$m\to \infty$. 

Case $\min(p,q)\le 1$. By Lemma \ref{l.extract}(b) and Remark 
\ref{r.triv}(c), 
$\cHRw{p,q}{k}_{e_1,\dots,e_r,e}=\cHRpw{p,q}{k}_{e_1,\dots,e_r,e}\subseteq 
\cHRpw{p,q}{k}_{e_1,\dots,e_r,\bone^{e}}=\cHRw{p,q}{k}_{e_1,\dots,e_r,\bone^{e}}$. 
Taking $b_i=(e-1)!(e/m)^i\binom{m}{i}a_i$ and $y_2=\dots=y_e=0$ in 
\eqref{e.Polyapf}, we get $\sum_{i=0}^m \frac{(e-1)!e^i}{(e-1+i)!}\cdot 
\frac{m!}{m^i(m-i)!}a_ig^{[i]}(\ux)y^i \in \cHRw{p,q}{k}_{e_1,\dots,e_r,1}$, 
which implies \eqref{e.Polya3} by letting $e,m\to \infty$. 
\end{proof}

\subsection{Multiplication by derived Schur polynomials}\label{s.6.5}

The results of this subsection are not used elsewhere in this article. 

\begin{prop}\label{t.final}
Let $g(\underline{x})\in \cHR{p,q}^{k}_{e_1,\dots,e_r}$ and let $j\le n$ such
that
\begin{equation}\label{e.final}
g(\underline{x})c_{j}(y_1,\dots,y_n)\in \cHRw{p,q}{k+j}_{e_1,\dots,e_r,n}.
\end{equation}
Let $X$ be a smooth projective variety of dimension $d$ satisfying 
$H^{p-2,q-2}(X)=0$ and let $F$ be an ample $\R$-twisted vector bundle on $X$ 
of rank $n+1$. Let $C\subseteq P=\PP_\bullet(F)$ be a closed subvariety of 
dimension $d_{C}\ge d+2$ dominating $X$ and let $Q$ be the restriction of the 
universal quotient bundle on $P$ to $C$. Let $\pi\colon C\to X$ be the 
projection. Assume that there exists an alteration $\phi\colon Z\to C$ with 
$Z$ smooth projective such that $H^{p-2,q-2}(Z)=0$ and $(\pi\phi)^*\colon 
H^{p-1,q-1}(X)\to H^{p-1,q-1}(Z)$ is a bijection. Let $E_1,\dots,E_r$ be 
ample $\R$-twisted vector bundles on $X$ of ranks $e_1,\dots,e_r$ and let 
$h\in \NS(X)_\R$ be an ample class. Then $(g(E_1,\dots,E_r)\pi_{*}(c_{j}(Q)) 
,h)$ is a Hodge--Riemann pair on $H^{p,q}(X)$ for $d_C-d\le j$ and 
$p+q+j+k=d_C$. 
\end{prop}

\begin{proof}
We apply the cone theorem (Theorem \ref{t.cone0}) to 
$\gamma=g(E_1,\dots,E_r)$. We check that the assumptions of the theorem are 
satisfied. By the assumption $g(\underline{x})\in 
    \cHR{p,q}^k_{e_1,\dots,e_r}$, Lemma \ref{l.HL}, and 
Remark \ref{r.triv}(b), we have $g(\underline{x})z^m\in
    \cHR{p,q}^{k+m}_{e_1,\dots,e_r,1}$ for all $m\ge 0$, which implies (A).
By \eqref{e.final} and Corollary \ref{c.partial}, 
$g(\underline{x})c_{j-i}(y_1,\dots,y_n)\in 
\cHRw{p,q}{k+j-i}_{e_1,\dots,e_r,n}$ for all $0\le i\le j$. Thus, by Lemma 
\ref{l.HL}, and Remark \ref{r.triv}(b), 
    $g(\underline{x})c_{j-i}(y_1,\dots,y_n)z^i\in
    \cHRw{p,q}{k+j}_{e_1,\dots,e_r,n,1}$. Then (B) follows from
    Lemma \ref{l.func3} and (C) follows from Lemma \ref{l.func3} and Remark \ref{r.pHR}.
\end{proof}

\begin{cor}\label{c.11}
Let $g(\underline{x})\in \cHR{p,q}^{k}_{e_1,\dots,e_r}$ such that 
\begin{equation}\label{e.11}
g(\underline{x})c_n(y_1,\dots,y_n)\in \cHRw{p,q}{k+n}_{e_1,\dots,e_r,n}
\end{equation}
for all $n\ge 1$. Then
\[g(\underline{x})s^{[j]}_\lambda(y_1,\dots,y_e)\in
\cHR{p,q}^{k+\lvert \lambda\rvert-j}_{e_1,\dots,e_r,e}\]
whenever $\lambda_1\le e$ and $0\le j\le \lvert \lambda\rvert$.
\end{cor}

\begin{proof}
By Corollary \ref{c.partial}, $g(\underline{x})c_{n-j}(y_1,\dots,y_n)\in 
\cHRw{p,q}{k+n-j}_{e_1,\dots,e_r,n}$ for all $0\le j\le n$. Thus it suffices 
to apply Proposition \ref{t.final}, derived Kempf--Laksov formula (Corollary 
\ref{c.KLtwist}), and Lemma \ref{l.KL}. 
\end{proof}

Next we show that Hodge--Riemann polynomials are stable under multiplication 
by signed Segre polynomials. Recall that $s_{\bone^l}$ equals $(-1)^l$ times 
the Segre polynomial of degree~$l$. 

\begin{prop}
Let $X$ be a smooth projective variety of dimension $d=p+q+k+l$ satisfying 
$H^{p-2,q-2}(X)=0$. Let $\gamma\in H^{k,k}(X,\R)$. Let $E$ be an ample 
$\R$-twisted vector bundle on $X$. Assume the following: 
\begin{enumerate}
\item For every smooth closed subvariety $Y$ of $X$ of dimension $p+q+k$ 
    satisfying $H^{p-2,q-2}(Y)=0$, for every ample class $h_Y\in 
    \NS(Y)_\R$, we have $(\gamma|_Y,h_Y)\in \fHR_{p,q}(Y)$. 
\item For every smooth closed subvariety $Y$ of $P=\PP(E)$ of dimension 
    $p+q+k$ satisfying $H^{p-2,q-2}(Y)=0$, for every ample class $h_Y\in 
    \NS(Y)_\R$, we have $(\gamma|_Y,h_Y)\in \fHRw_{p,q}(Y)$. 
\end{enumerate}
Then, for every ample class $h\in \NS(X)_\R$, we have $(\gamma
s_{\bone^l}(E),h)\in \fHR_{p,q}(X)$.
\end{prop}

\begin{proof}
Let $e$ denote the rank of $E$. In the case $e=1$, we have 
$s_{\bone^l}(E)=c_1(E)^l$ and $(\gamma c_1(E)^l,h)\in \fHR_{p,q}(X)$ by (a) 
and Corollary \ref{c.multiHL}. 

Assume $e>1$. Let $\pi\colon P=\PP(E)\to X$ be the projection. By assumption, 
$\zeta=c_1(\cO_P(1))$ is ample. Let $I$ be the open interval consisting of 
$t\in \R$ such that $\zeta+t\pi^*h$ is ample. Let 
\[\delta_i(t)=(\zeta+t\pi^*h)^{e-1+l-i}\pi^*(h^i\gamma).\]
We apply Corollary \ref{c.iterate} and Remark \ref{r.Cp} to
\[V=H^{p,q}(P)=U\oplus fW,\quad W=H^{p-1,q-1}(X),\quad U=\pi^* H^{p,q}(X),\]
$f=\zeta\wedge \pi^*-\colon W\hookrightarrow V$, $\iota=\pi^*(h\wedge
-)\colon W\hookrightarrow V$, and $H_t(-,-)=\langle-,-
\rangle_{\delta_0(t)}$. We check that the assumptions are satisfied. Clearly
$\iota W\subseteq U$. We have
\[
H^{(i)}_t(-,-)=\langle-,-\rangle_{\delta_0^{(i)}(t)},\quad \delta^{(i)}_0(t)=\frac{(e-1+l)!}{(e-1+l-i)!}\delta_i(t).
\]
Thus \eqref{e.f} holds for $0\le i\le l$, with $\kappa_i=e-1+l-i>0$.
\begin{enumerate}
\item[(A)] We have $\pi_*\delta_l(t)=h^l\gamma$ and $(h^l\gamma,h)\in
    \fHR_{p,q}(X)$ by (a) and Corollary \ref{c.multiHL}.

\item[(B)] We have $(\delta_i(t),\pi^*h)\in \fHRw_{p,q}(P)$ for all $0\le
    i\le l$ and $t\in I$, by (b), Corollary \ref{c.multiHL}, and
    continuity.

\item[(C$'$)]  We have $(\delta_1(0),\zeta)\in \fHRw_{p,q}(P)$ if $l\ge 1$, 
    by (b), Corollary \ref{c.multiHL}, and continuity. 

\item[(e)] Since $\pi_*\delta_{l+1}(0)=0$, we have $H_0^{(l+1)}|_{U\times
    U}=0$.
\end{enumerate}
By Corollary \ref{c.iterate}, $(\pi_*\delta_0(0),h)\in \fHR_{p,q}(X)$, and we
conclude by the formula $\pi_*\delta_0(0)=\gamma \pi_*(\zeta^{e-1+l})=\gamma
s_{\bone^l}(E)$.
\end{proof}

\begin{cor}\label{c.Segre}
Let $g(\underline{x})\in \cHR{p,q}^k_{e_1,\dots,e_r}$ (resp.\
$\cHRp{p,q}^k_{e_1,\dots,e_r}$). Then
$g(\underline{x})s_{\bone^l}(y_1,\dots,y_e)\in
\cHR{p,q}^{k+l}_{e_1,\dots,e_r,e}$ (resp.\
$\cHRp{p,q}^{k+l}_{e_1,\dots,e_r,e}$).
\end{cor}

\subsection{Log-concavity}\label{s.6.6}

\begin{theorem}\label{t.Lor}
Let $g\in \cHRpw{1,1}{k}_{e_1,\dots,e_r}$. Let $X$ be a smooth projective 
variety of dimension $d$ and let $E_1,\dots,E_r$ be nef $\R$-twisted vector 
bundles of ranks $e_1,\dots,e_r$. Let $A=(a_{ij})$ be an $s\times r$ matrix 
with entries in $\R_{\ge 0}$. Let $v\in \R[y_1,\dots,y_s]$ be a Lorentzian 
polynomial of degree~$n$ such that $m=d-k+n\ge 0$. Then the polynomial 
\[f(z_1,\dots,z_s)=\int_X (\partial_{T^m v, A}g)(E_1,\dots,E_r)\]
is Lorentzian. Here $T=z_1\frac{\partial}{\partial y_1}+\dots+ 
z_s\frac{\partial}{\partial y_s}$ and $\partial_{T^m v, A}$ is obtained from 
$T^m v$ by substituting $D_i=\sum_{j=1}^r a_{ij}\partial_j$ for $y_i$, $1\le 
i\le s$. 
\end{theorem}

\begin{lemma}\label{l.volint}
Let $\gamma\in H^{l,l}(Y,\R)$, $\gamma'\in H^{l',l'}(X,\R)$, $\xi_i\in 
\NS(Y)_\R$, $1\le i\le s$. Let $\pi_X\colon X\times Y\to X$ and $\pi_Y\colon 
X\times Y\to Y$ be the projections.  Let $\zeta_j=\sum_{i=1}^s  a_{ij}\xi_i$, 
$F_j=E_j\boxtimes \cO_Y(\zeta_j)$, $1\le j \le r$. Assume $m=\dim(X\times 
Y)-k-l-l'\ge 0$. Then 
\[\frac{1}{m!}\int_X(\partial_{T^m \vol_{\gamma;\xi_1,\dots,\xi_s},A}
g)(E_1,\dots,E_r)\gamma'
=\vol_{g(F_1,\dots,F_r)\pi_X^*\gamma'\pi_Y^*\gamma;\pi_Y^*\xi_1,\dots,\pi_Y^*\xi_s}(z_1,\dots,z_s).\]
\end{lemma}

\begin{proof}
For $\alpha=(\alpha_1,\dots,\alpha_s)\in \N^s$ with $\lvert \alpha\rvert=m$, 
we have 
\begin{multline*}
\partial^\alpha \vol_{g(F_1,\dots,F_r)\pi_X^*\gamma'\pi_Y^*\gamma;\pi_Y^*\xi_1,\dots,\pi_Y^*\xi_s} = \int_{X\times Y} g(F_1,\dots,F_r)\pi_X^*\gamma'\pi_Y^*(\gamma\xi_1^{\alpha_1}\dotsm \xi_s^{\alpha_s})\\
=\int_X (\partial_{\vol_{\gamma\xi_1^{\alpha_1}\dotsm \xi_s^{\alpha_s};\xi_1,\dots,\xi_s}A}g)(E_1,\dots,E_r)\gamma'=\frac{1}{m!}\partial^\alpha\int_X(\partial_{T^m \vol_{\gamma;\xi_1,\dots,\xi_s},A}
g)(E_1,\dots,E_r)\gamma'.
\end{multline*}
Here we used \eqref{e.der} in the second equality and
\[\vol_{\gamma\xi_1^{\alpha_1}\dotsm \xi_s^{\alpha_s};\xi_1,\dots,\xi_s}=\partial^\alpha \vol_{\gamma;\xi_1,\dots,\xi_s} =\frac{1}{m!}\partial_z^\alpha T^m \vol_{\gamma;\xi_1,\dots,\xi_s}\]
in the third equality.
\end{proof}

\begin{proof}[Proof of Theorem \ref{t.Lor}]
By Lemma \ref{l.extract}(b), $g\in \cHRw{1,1}{k}_{e_1,\dots,e_r}$. By Lemma 
\ref{l.vol}, $v=\vol_{p(\xi_1,\dots,\xi_s);\xi_1,\dots,\xi_s}$ for $p$ dually 
Lorentzian and $\xi_1,\dots,\xi_s\in \NS(Y)_\R$ nef. The theorem then follows 
from Lemma \ref{l.volint} applied to $\gamma=p(\xi_1,\dots,\xi_s)$ and 
$\gamma'=1$ and Corollary \ref{c.vLor}. 
\end{proof}

\begin{cor}\label{c.Lor}
Let $g$, $X$, $E_1,\dots,E_r$, $A$, $D_1,\dots,D_s$ be as in Theorem 
\ref{t.Lor}. Then the polynomial 
\[p(z_1,\dots,z_s)=\sum_{\substack{\alpha\in \N^s\\\lvert \alpha\rvert=k-d}}z_1^{\alpha_1}\dotsm z_s^{\alpha_s}\int_X \left(\frac{D_1^{\alpha_1}\dotsm D_s^{\alpha_s}g}{\alpha_1!\dotsm \alpha_s!}\right)(E_1,\dots,E_r)\]
is dually Lorentzian. Moreover, for $n=n_1+\dots+n_s$ and $m=d-k+n$, the 
polynomial
\begin{equation}\label{e.Lor}
f(z_1,\dots,z_s)=\sum_{\substack{m_1+\dots+m_s=m\\0\le m_i\le n_i}}\frac{z_1^{m_1}}{m_1!}\dotsm \frac{z_s^{m_s}}{m_s!}\int_X \left(\frac{D_1^{n_1-m_1}\dotsm D_s^{n_s-m_s}g}{(n_1-m_1)!\dotsm (n_s-m_s)!}\right)(E_1,\dots,E_r)
\end{equation}
is Lorentzian.
\end{cor}

\begin{proof}
By Theorem \ref{t.Lor} applied to 
\[v(y_1,\dots,y_s)= \vol_{1;\xi_1,\dots,\xi_s}(y_1,\dots,y_s)=\frac{y_1^{n_1}}{n_1!}\dotsm \frac{y_s^{n_s}}{n_s!},\]
$f$ is Lorentzian. Here $\xi_i=\pi_i^*c_1(\cO_{\PP^{n_i}}(1))$, 
$\pi_i\colon\PP^{n_1}\times \dotsm \times \PP^{n_s}\to \PP^{n_i}$ is the 
$i$-th projection. Indeed, 
\begin{equation}\label{e.Tm}
T^m v=m!\sum_{\substack{m_1+\dots+m_s=m\\0\le m_i\le n_i}}\frac{z_1^{m_1}}{m_1!}\dotsm \frac{z_s^{m_s}}{m_s!}\frac{y_1^{n_1-m_1}}{(n_1-m_1)!}\dotsm \frac{y_s^{n_s-m_s}}{(n_s-m_s)!}.
\end{equation}
Take $n_1=\dots=n_s=k$. Then $f=p\spcheck_\kappa$, where 
$\kappa=(k,\dots,k)$. Thus $p$ is dually Lorentzian. 
\end{proof}

\begin{remark}
\begin{enumerate}
\item Taking $s=r$, $A=I_r$, and $g=s_{\lambda^1}(\underline{x_1})\dotsm 
    s_{\lambda^r}(\underline{x_r})$, we obtain the Lorentzian case of 
    Theorem \ref{t.intro9}. 
    
\item One can also prove that $p$ is dually Lorentzian more directly using 
    Corollary \ref{c.dLor} by considering $\pi_{Y*}g(F_1,\dots,F_r)$. We 
    opted for the Lorentzian way of presentation in order to treat the 
    Lorentzian and strictly Lorentzian cases in parallel. 
\end{enumerate}
\end{remark}

\begin{cor}\label{c.logconcave}
Let $\alpha=(\alpha_1,\dots,\alpha_s)\in \N^s$ and let $1\le i,j\le s$ such 
that $\alpha_i,\alpha_j>0$. Let $X$ be a smooth projective variety of 
dimension $d$ and let $k=d+\lvert \alpha \rvert$. Let $g$, $E_1,\dots,E_r$, 
$A$, and $D_1,\dots,D_s$ be as in Theorem \ref{t.Lor}. Then 
\[\left(\int_X \frac{D^\alpha g}{\alpha!}(E_1,\dots,E_r)\right)^2\ge \int_X \frac{D^{\alpha-\be_i+\be_j} g}{(\alpha-\be_i+\be_j)!}(E_1,\dots,E_r)\cdot \int_X \frac{D^{\alpha+\be_i-\be_j} g}{(\alpha+\be_i-\be_j)!}(E_1,\dots,E_r).\]
Here $D^\alpha=D_1^{\alpha_1}\dotsm D_s^{\alpha_s}$ and $\be_l\in\N^s$ 
denotes the vector defined by $(\be_l)_{l'}=\delta_{ll'}$. 
\end{cor}

\begin{proof}
This follows from Corollary \ref{c.Lor} applied to $(n_1,\dots,n_s)\ge 
\alpha+\be_i+\be_j$ and $m=n_1+\dots+n_s-\lvert \alpha\rvert $ and the ultra 
log-concavity of coefficients of Lorentzian polynomials \cite[Proposition 
4.4]{BH}. (In fact, we can even take $(n_1,\dots,n_s)= \alpha+\be_i+\be_j$ 
and $m=2$.) 
\end{proof}

In some cases, the polynomial \eqref{e.Lor} in Corollary \ref{c.Lor} is 
strictly Lorentzian. 

\begin{theorem}\label{t.sLor}
Let $\lambda^1,\dots,\lambda^r$ be partitions, $0\le s\le r$. Let 
$m,n_1,\dots,n_r,e_1,\dots, e_r\ge 0$ such that $n_i\le \lvert 
\lambda^i\rvert$ and $e_i\ge (\lambda^i)_1$ for all $1\le i\le r$ and $m\le 
n_i$ for all $1\le i\le s$. Let $X$ be a smooth projective variety of 
dimension $m+\sum_{i=1}^r (\lvert \lambda^i\rvert -n_i)$ and let $E_1,\dots, 
E_r$ be ample $\R$-twisted vector bundles on $X$ of ranks $e_1,\dots,e_r$, 
respectively. The polynomial 
\[f(z_1,\dots,z_s)=\sum_{\substack{m_1+\dots+m_s=m\\0\le m_i\le n_i}}\frac{z_1^{m_1}\dotsm z_s^{m_s}}{m_1!\dotsm m_s!}\int_X s_{\lambda^1}^{[n_1-m_1]}(E_1)\dotsm s_{\lambda^s}^{[n_s-m_s]}(E_s)s_{\lambda^{s+1}}^{[n_{s+1}]}(E_{s+1})\dotsm s_{\lambda^r}^{[n_r]}(E_r)\]
is strictly Lorentzian. 
\end{theorem}

Taking $s=r$, we obtain the strictly Lorentzian case of Theorem 
\ref{t.intro9}. 

\begin{proof}
We follow the same strategy as in the proof of Corollary \ref{c.Lor}. Let 
$g=s_{\lambda^1}(\underline{x_1})\dotsm s_{\lambda^s}(\underline{x_s})$. We 
take $Y=\PP^{n_1}\times \dotsm \times \PP^{n_s}$, 
\[v(y_1,\dots,y_s)= 
\vol_{1;\xi_1,\dots,\xi_s}(y_1,\dots,y_s)=\frac{y_1^{n_1}}{n_1!}\dotsm 
\frac{y_s^{n_s}}{n_s!},
\] 
where $\xi_i=\pi_i^*c_1(\cO_{\PP^{n_i}}(1))$, 
$\pi_i\colon Y\to \PP^{n_i}$ is the $i$-th projection. By \eqref{e.Tm} and 
Lemma \ref{l.volint}, 
\[f(z_1,\dots,z_s)=\frac{1}{m!}\int_X(\partial_{T^m v, I_s}
g)(E_1,\dots,E_s)\gamma' 
=\vol_{g(F_1,\dots,F_s)\pi_X^*\gamma';\pi_Y^*\xi_1,\dots,\pi_Y^*\xi_s}(z_1,\dots,z_s),
\]
where $\pi_X\colon X\times Y\to X$ and $\pi_Y\colon X\times Y\to Y$ are the 
projections, $F_i=E_i\boxtimes \cO_Y(\xi_i)$ for $1\le i\le s$, and 
$\gamma'=s_{\lambda^{s+1}}^{[n_{s+1}]}(E_{s+1})\dotsm 
s_{\lambda^r}^{[n_r]}(E_r)$. Since $\xi_i$ and $F_i$ are not ample,  we 
cannot quite apply Corollary \ref{c.vLor}(b). Instead, we adapt its proof as 
follows. To show that $f$ is strictly Lorentzian, it suffices to check the 
following for all $1\le \alpha_1,\dots,\alpha_m\le s$: 
\begin{enumerate}
\item $\frac{\partial}{\partial z_{\alpha_1}}\dotsm 
    \frac{\partial}{\partial z_{\alpha_m}} f > 0$. 
\item If $m\ge 2$, then the bilinear form  $(\bx,\by)\mapsto D_\bx 
    D_{\by}\frac{\partial}{\partial z_{\alpha_3}}\dotsm 
    \frac{\partial}{\partial z_{\alpha_{m}}} f$ on $\R^s$ is nondegenerate 
    and has exactly one positive eigenvalue. 
\end{enumerate}
We have 
\[\frac{\partial}{\partial 
    z_{\alpha_1}}\dotsm \frac{\partial}{\partial z_{\alpha_m}} f =\int_{X\times Y} s_{\lambda^1}(F_1)\dotsm s_{\lambda^s}(F_s)\pi_X^*\gamma' \pi_Y^*(\xi_{\alpha_1}\dotsm \xi_{\alpha_m}).
\]

Let $\zeta\in \NS(X)_\R$ be an ample class such that $E_i\langle 
-\zeta\rangle$ is nef for all $1\le i\le r$. Then, for each $1\le i\le s$, 
$h_i=\pi_X^* \zeta+\pi_Y^*\xi_i$ is the pullback of an ample class on 
$X\times \PP^{n_i}$ and $F_i\langle-h_i\rangle=\pi_X^*E_i\langle 
-\zeta\rangle$ is nef. Let $l'=\sum_{i=s+1}^r(\lvert \lambda^i\rvert-n_i)$. 
By Lemma \ref{l.HX} below, 
\[\int_{X\times Y} h_1^{\lvert \lambda^1\rvert}\dotsm h_s^{\lvert \lambda^s\rvert}\pi_X^*\zeta^{l'}\pi_Y^*(\xi_{\alpha_1}\dotsm \xi_{\alpha_m})>0\]
and, if $m\ge 2$ and $\eta\in\{h_1,\dots,h_s,\pi_X^* 
\zeta,\pi_Y^*\xi_1,\dots,\pi_Y^*\xi_s\}$, 
\[(h_1^{\lvert \lambda^1\rvert}\dotsm h_s^{\lvert \lambda^s\rvert}\pi_X^*\zeta^{l'}\pi_Y^*(\xi_{\alpha_3}\dotsm \xi_{\alpha_{m}}),\eta)\in \fHR_{1,1}(X\times Y).\]
Thus, by Theorem \ref{p.refined}, (a) holds and, if $m\ge 2$, 
\[(s_{\lambda^1}(F_1)\dotsm s_{\lambda^s}(F_s)\pi_X^*\gamma'\pi_Y^*(\xi_{\alpha_3}\dotsm \xi_{\alpha_{m}}),\pi_Y^*\xi_1)\in \fHR_{1,1}(X\times Y),\]
which implies (b) by the fact that $\pi_Y^*\xi_1,\dots,\pi_Y^*\xi_s$ are 
$\R$-linearly independent.
\end{proof}

\begin{lemma}\label{l.HX}
Let $s\ge 1$, $k,l,l'\ge 0$, $n\in \N^s$ satisfying $p+q+l\le n_i$ for all 
$1\le i\le s$. Let $P=X\times Y_1\times \dotsm \times Y_s$, where $X$, 
$Y_1,\dots, Y_s$ are smooth projective varieties with $d=\dim(X)$, 
$n_i=\dim(Y_i)$, satisfying $p+q+k +l+l'=d+\lvert n\rvert$. Let 
$\zeta_1,\dots,\zeta_{l'}\in\pi_X^* \Amp(X)$,  $\xi_1,\dots,\xi_{l}\in 
\bigcup_{i=1}^s \pi_i^* \Amp(Y_i)$. Let $\{1,\dots,k\}=\coprod_{i=1}^s K_i$ 
with $\# K_i\ge n_i$ for all $1\le i\le s$. For each $j\in K_i$, let $h_j\in 
\pi_{X,i}^*\Amp(X\times Y_i)$. Here $\pi_X\colon P\to X$, $\pi_i\colon P\to 
Y_i$, and $\pi_{X,i}\colon P\to X\times Y_i$ denote the projections. Then, 
for every $\eta\in \pi_X^* \Amp(X)\cup 
\bigcup_{i=1}^s(\pi_{X,i}^*\Amp(X\times Y_i)\cup \pi_i^*\Amp(Y_i))$ and $c\ge 
0$, $(\gamma,\gamma\eta)\in \HR_{p-c,q-c}(P)$, where $\gamma=h_1\dotsm 
h_{k}\xi_1\dotsm \xi_l\zeta_1\dotsm \zeta_{l'}\eta^{2c}$. 
\end{lemma}

\begin{proof}
We apply the generalized mixed Hodge--Riemann relations of Hu and Xiao 
\cite[Corollary~A]{HX1} (which extends a theorem of Xiao 
\cite[Theorem~A]{Xiao}).  For each $1\le j\le k$, there exists a unique $1\le 
i_j\le s$ such that $j\in K_{i_j}$ and $h_j$ has a representative $\hat h_j$ 
that is the pullback of a K\"ahler form on $X\times Y_{i_j}$. For each $1\le 
m\le l$, $\xi_m$ has a representative $\hat \xi_m$ that is the pullback of a 
K\"ahler form on $Y_{i'_m}$ for some $1\le i'_m\le s$. For each $1\le m\le 
l'$, $\zeta_m$ has a representative $\hat \zeta_m$ that is the pullback of a 
K\"ahler form on $X$. Let $J\subseteq \{1,\dots,  k\}$, $M\subseteq 
\{1,\dots,l\}$, $M'\subseteq \{1,\dots,l'\}$ such that $\#J+\#M+\#M'>0$ and 
let $\omega=\sum_{j\in J}\hat h_j+\sum_{m\in M} \hat\xi_m +\sum_{m\in 
M'}\hat\zeta_m$. If $\#J+\#M'=0$, then $\omega$ is the pullback of a K\"ahler 
form on $ \prod_{i\in I'_M} Y_i$, where $I'_M=\{i'_m\mid m\in M\}$, which has 
dimension $\ge p+q+l$, since $M$ is nonempty in this case. If $\#J+\#M'>0$, 
then $\omega$ is the pullback of a K\"ahler form on $X\times \prod_{i\in 
I_J\cup I'_M} Y_i$, where $I_J=\{i_j\mid j\in J\}$, which has dimension 
\[d+\sum_{i\in I_J\cup I'_M} n_i=d+\lvert n\rvert -\sum_{i\notin I_J\cup I'_M}n_i\ge p+q+l+l'+k-\sum_{i\notin I_J\cup I'_M}\# K_i=p+q+l+l'+\sum_{i\in I_J\cup I'_M}\#K_i\ge 
p+q+l+l'+\#J.
\]
Thus $\omega$ is $(p+q+\#J+\#M+\#M')$-positive in both cases. Moreover, since 
$d\ge p+q+l$ and $n_i\ge p+q+l$ for each $1\le i\le s$, $\eta$ has a 
representative $\hat \eta$ that is the pullback of a K\"ahler form on a 
smooth projective variety of dimension $\ge p+q+l\ge p+q$. We conclude by 
\cite[Corollary~A]{HX1}. 
\end{proof}

\begin{cor}
Let $r\ge 2$ and let $\lambda^1,\dots,\lambda^r$ be partitions. Let $X$ be a 
smooth projective variety of dimension $d$ and let $E_1,\dots,E_r$ be ample 
$\R$-twisted vector bundles of ranks $e_1,\dots,e_r$, respectively,  
satisfying $e_i\ge (\lambda^i)_1$ for all $1\le i \le r$. For each $3\le i\le 
r$, let $n_i$ be an integer satisfying $0\le n_i\le \lvert \lambda^i\rvert$. 
Then
\[
b_j=\int_X s_{\lambda^1}^{[j]}(E_1)s_{\lambda^2}^{[k-j]}(E_2)s_{\lambda^3}^{[n_3]}(E_3)\dotsm s_{\lambda^r}^{[n_r]}(E_r),\quad \max(0,k-\lvert \lambda^2\rvert)\le j\le \min(\lvert \lambda^1\rvert, k),
\]
where $k=\sum_{i=1}^r\lvert \lambda^i\rvert-d-\sum_{i=3}^r n_i$, is a 
strictly log-concave sequence of positive numbers. 
\end{cor}

Taking $r=2$, we obtain Corollary \ref{c.intro10}.

\begin{proof}
By Corollary \ref{c.posfinal}, $b_j>0$ for all $j$. Let $j$ be such that 
$\max(0,k-\lvert \lambda^2\rvert)< j< \min(\lvert \lambda^1\rvert, k)$. 
Applying Theorem \ref{t.sLor} to $s=m=2$, $n_1=j+1$, $n_2=k-j+1$, we get 
$b_j^2>b_{j-1}b_{j+1}$. Indeed, $a\frac{x^2}{2}+bxy+c\frac{y^2}{2}$ is 
strictly Lorentzian if and only if $(a, b, c)$ is a strictly log-concave 
sequence of positive numbers. 
\end{proof}

In the rest of this section, we deduce generalized Schur positivity from the 
log-concavity in Corollary \ref{c.logconcave} using the geometric 
characterization of generalized Schur positivity (Proposition 
\ref{p.converse}). The Fulton--Lazarsfeld characterization 
\cite[Theorem~I]{FL} was used in \cite[Corollary 7.2]{Pragacz} to show the 
Schur positivity of Schur polynomials of Schur powers (as pointed out by the 
referee) and in \cite[Remark 5.4]{RT2} to recover the Schur positivity of 
derived Schur polynomials (see Corollary \ref{c.pos}, Remark \ref{r.pos}). In 
both examples, the inequality follows from Hodge--Riemann relations in 
bidegree $(0,0)$. Here we show how to deduce (generalized) Schur positivity 
from Hodge--Riemann relations in bidegree $(1,1)$. 

Note that Proposition \ref{p.converse} does not immediately apply to the 
inequality in Corollary \ref{c.logconcave}, which involves products of 
integrals. However, log-concave sequences of nonnegative real numbers without 
internal zeroes are unimodal, and a symmetric unimodal sequence has its 
maximum at the central term or terms. Thus, under these additional 
assumptions, we obtain inequalities between integrals to which Proposition 
\ref{p.converse} applies. 

Let $S$ be a set with a distinguished element $0$. We say that a sequence 
$b_0,\dots,b_n$ in $S$ has an \emph{internal zero} if there exist $0\le 
i<j<k\le n$ such that $b_i\neq 0$, $b_j=0$, $b_k\neq 0$. 

\begin{cor}\label{c.multipos}
Let $g\in \cHRpw{1,1}{k}_{e_1,\dots,e_r, e_1,\dots,e_r, f_1,\dots,f_s}$ such
that
$g(\underline{x},\underline{y},\underline{z})=g(\underline{y},\underline{x},\underline{z})$.
Let $D=\sum_{i=1}^r a_i \partial_{x_i}$ and $D'=\sum_{i=1}^r a_i
\partial_{y_i}$, where $a_i\in \R_{\ge 0}$ for all $1\le i \le r$. Let $1\le m\le n$. Assume 
that the sequence $D^{m+i}D'^{n-i}g$, $-1\le i\le n-m +1$ has no internal 
zeroes. Then 
\begin{equation}\label{e.multipos}
\frac{D^m D'^n g}{m!n!}(\underline{x},\underline{x},\underline{z})-\frac{D^{m-1}D'^{n+1}g}{(m-1)!(n+1)!}(\underline{x},\underline{x},\underline{z})\in \cP^{k-m-n}_{e_1,\dots,e_r,f_1,\dots,f_s}.
\end{equation}
\end{cor}

\begin{proof}
We may assume that $d=k-m-n \ge 0$. Let $X$ be a smooth projective variety of 
dimension $d$. Let $E_1,\dots,E_r,E'_1,\dots,E'_r,F_1,\dots,F_s$ be ample 
$\R$-twisted vector bundles on $X$ of ranks 
$e_1,\dots,e_r,e_1,\dots,e_r,f_1,\dots,f_s$, respectively, with the same 
$\R$-twist modulo $\NS(X)$. Let 
\[b_i=\int_X
\frac{D^{m+i}D'^{n-i}g}{(m+i)!(n-i)!}(E_1,\dots,E_r,E'_1,\dots,E'_r,F_1,\dots,F_s).\] 
By Lemma \ref{l.triv}(e), $g\in 
\cP^{k}_{e_1,\dots,e_r,e_1,\dots,e_r,f_1,\dots,f_s}$. Thus, by Theorem 
\ref{t.Schurdef}, $b_i\ge 0$ for all $i$ and $b_{-1},\dots,b_{n-m+1}$ has no 
internal zeroes. By Corollary \ref{c.logconcave}, the sequence 
$b_{-1},\dots,b_{n-m+1}$ is log-concave. Thus $b_0b_{n-m}\ge b_{-1}b_{n-m+1}$ 
by Lemma \ref{l.logcon} below. Now take $E'_i=E_i$. Then $b_i=b_{n-m-i}$. 
Therefore $b_0^2\ge b_{-1}^2$, and consequently $b_0\ge b_{-1}$. In other 
words, if we denote by $p(\ux,\uz)\in 
\cS^{k-m-n}_{e_1,\dots,e_r,f_1,\dots,f_s}$ the polynomial in 
\eqref{e.multipos}, then $\int_X p(E_1,\dots,E_r,F_1,\dots,F_s)\ge 0$. By 
Proposition \ref{p.converse}, $p(\ux,\uz)\in 
\cP^{k-m-n}_{e_1,\dots,e_r,f_1,\dots,f_s}$. 
\end{proof}

\begin{lemma}\label{l.logcon}
Let $b_0,\dots,b_n$ be a sequence of nonnegative real numbers. The following
conditions are equivalent:
\begin{enumerate}
\item The sequence is log-concave with no internal zeroes.

\item For all $0< i\le j<n$, $b_{i}b_j\ge b_{i-1}b_{j+1}$.
\end{enumerate}
\end{lemma}

\begin{proof}
(b)$\implies$(a). The log-concavity is clear by taking $i=j$. If there exists
an internal zero, then there exists $i\le j$ such that $b_i=b_j=0$ and
$b_{i-1}\neq 0\neq b_{j+1}$, contradicting (b).

(a)$\implies$(b). We may assume $i=1$ and $j=n-1$. We may assume $b_0b_n\neq
0$. By log-concavity,
\[b_1^2 b_2\dotsm b_{n-1}\ge b_0b_2^2b_3\dotsm b_{n-1}\ge
\dots \ge b_0b_1\dotsm b_{n-2}b_n.
\]
Since there are no internal zeroes, $b_1\dotsm b_{n-2} \neq 0$.
\end{proof}

As a special case of Corollary \ref{c.multipos}, we have the following Schur
log-concavity of derivative sequences.

\begin{cor}\label{c.Slogcon}
Let $f\in\cS^k_{e_1,\dots,e_r}$ such that 
$f(\underline{x})f(\underline{y})\in 
\cHRpw{1,1}{2k}_{e_1,\dots,e_r,e_1,\dots,e_r}$. Let $a_1,\dots,a_r\in \R_{\ge 
0}$ and let $D=\sum_{i=1}^r a_i\partial_i$.  Then, for $1\le m\le n$, we have
\[\frac{D^mf}{m!} \frac{D^nf}{n!}-\frac{D^{m-1}f}{(m-1)!}\frac{D^{n+1}f}{(n+1)!}\in \cP^{2k-m-n}_{e_1,\dots,e_r}\]
and, in particular, 
\[f^{[m]}f^{[n]}-f^{[m-1]}f^{[n+1]}\in \cP^{2k-m-n}_{e_1,\dots,e_r}.\]
\end{cor}

In particular, Theorem \ref{t.ifinal} holds. For $l<0$ we adopt the 
convention $\cP^l_{e_1,\dots,e_r}=\{0\}$. 

\begin{proof}
We may assume $D^nf\neq 0$. The first assertion follows from Corollary 
\ref{c.multipos} applied to $g=f(\underline{x})f(\underline{y})$. For the 
last assertion it suffices to take $a_1=\dots = a_r=1$. 
\end{proof}

By Example \ref{e.Sch}(a), this applies in particular to products of Schur 
polynomials. 

\begin{cor}\label{c.prodSchlogcon}
Let $\lambda^1,\dots,\lambda^r$ be partitions and let
\[f(\ux)=s_{\lambda^1}(x_{1,1},\dots,x_{1,e_1})\dotsm
    s_{\lambda^r}(x_{r,1},\dots,x_{r,e_r}). \]
Then, for $1\le m\le n$, we have 
\[f^{[m]}f^{[n]}-f^{[m-1]}f^{[n+1]}\in
\cP^{2(\lvert \lambda^1\rvert+\dots +\lvert
\lambda^r\rvert)-m-n}_{e_1,\dots,e_r}.
\]
\end{cor}

Corollary \ref{c.ifinal} follows since the map $\cS^k_{e,\dots,e}\to \cS^k_e$
sending $f$ to $f(\ux,\dots,\ux)$ carries $\cP^k_{e,\dots,e}$ into $\cP^k_e$,
by Remark \ref{r.triv}(b) or by the fact that products of Schur polynomials
are Schur positive.

Unwinding the proofs, we see that the Schur positivity of 
$(s_\lambda^{[n]})^2-s_{\lambda}^{[n-1]}s_\lambda^{[n+1]}$ follows easily 
from the log-concavity of the sequence $i\mapsto \int_X 
s^{[i]}_\lambda(E)s^{[2n-i]}_\lambda(E)$ proved in \cite[Theorem 10.5]{RT2}, 
of which Corollary \ref{c.logconcave} is a generalization. Here $X$ is a 
smooth projective variety of dimension $2(\lvert \lambda\rvert -n)$ and $E$ 
is a nef vector bundle on $X$. Indeed, the log-concavity implies 
\[\left(\int_X s^{[n]}_\lambda(E)s^{[n]}_\lambda(E)\right)^2\ge \left(\int_X s^{[n-1]}_\lambda(E)s^{[n+1]}_\lambda(E)\right)^2\]
and 
\[\int_X s^{[n]}_\lambda(E)s^{[n]}_\lambda(E)\ge\int_X s^{[n-1]}_\lambda(E)s^{[n+1]}_\lambda(E)\] 
and it suffices to apply \cite[Proposition~3.4]{FL} together with resolution 
of singularities (or Proposition \ref{p.converse}). 

\begin{proof}[Proof of Corollary \ref{c.idLor}]
By \cite[Proposition 4.10, Theorem 5.12]{RSW}, 
$f(\underline{x})f(\underline{y})$ is dually Lorentzian. Thus, by Corollary 
\ref{c.dLor}, $f(\underline{x})f(\underline{y})\in 
\cHRw{1,1}{2k}_{\bone^{2r}}$, where $k=\deg(f)$. By the case 
$e_1=\dots=e_r=1$ of Corollary \ref{c.Slogcon},  $\frac{D^mf}{m!} 
\frac{D^nf}{n!}-\frac{D^{m-1}f}{(m-1)!}\frac{D^{n+1}f}{(n+1)!}\in 
\cP^{2k-m-n}_{\bone^r}$. 
\end{proof}

\begin{remark}
\begin{enumerate}
\item The conclusion of Corollary \ref{c.idLor} is stronger than the 
    monomial-positivity of 
    $(\frac{D^nf}{n!})^2-\frac{D^{n-1}f}{(n-1)!}\frac{D^{n+1}f}{(n+1)!}$ 
    for all $n\ge 1$. In fact, the analogue of Lemma \ref{l.logcon} does 
    not hold for monomial-positivity. For example, for 
    $(f_0,f_1,f_2,f_3)=(4x^2y,(x+y)^2,x+y,1)$, 
    $f_1^2-f_0f_2=(x+y)(f_1f_2-f_0f_3)=x^4+2x^2y^2+4xy^3+y^4$ and 
    $f_2^2-f_1f_3=0$ are monomial-positive, but 
    $f_1f_2-f_0f_3=x^3-x^2y+3xy^2+y^3$ is not monomial-positive. The issue 
    is that the proof of Lemma \ref{l.logcon} uses division, which does not 
    preserve monomial-positivity. 
\item In Corollary \ref{c.idLor},  $f^{[m]}f^{[n]}-f^{[m-1]}f^{[n+1]}$ is 
    not dually Lorentzian in general. For example, $f=x^3+2x^2y+4xy^2+8y^3$ 
    is dually Lorentzian, but 
    $(f^{[1]})^2-f^{[0]}f^{[2]}=14x^4+64x^3y+312x^2y^2+448xy^3+512y^4$ is 
    not dually Lorentzian. 
\item It follows easily from Corollary \ref{c.idLor} that all $2\times 2$ 
    minors of the matrix $\left(\frac{D^{j-i}f}{(j-i)!}\right)_{0\le 
    i,j<\infty}$ are monomial-positive (see Lemma \ref{l.Li} below for a 
    generalization). Here by convention $D^n=0$ for $n<0$. However, 
    $3\times 3$ minors of the matrix are not monomial-positive in general. 
    In fact, they are not even nonnegative on $\R_{\ge 0}^r$ in general. 
    Indeed, for $f=x^2+3xy+9y^2$, which is dually Lorentzian, we have 
    \[\det\begin{pmatrix}
      f^{[1]}(1,0) & f^{[2]}(1,0) & 0\\
      f^{[0]}(1,0) & f^{[1]}(1,0) & f^{[2]}(1,0)\\
      0 & f^{[0]}(1,0) & f^{[1]}(1,0)
    \end{pmatrix}=-5.\]
    In particular, the sequence $(f^{[0]}(1,0), f^{[1]}(1,0), 
    f^{[2]}(1,0))=(1,5,13)$ is not a P\'olya frequency sequence. 
\end{enumerate}
\end{remark}

Our results on (generalized) Schur positivity extend formally to products of 
multiple derivatives as follows. Let $M$ be a non-unital commutative monoid, 
written multiplicatively, equipped with a preorder $\le$ satisfying 
$f_1f_2\le g_1g_2$ whenever $f_1\le g_1$ and $f_2\le g_2$. We say that a 
sequence $f_0,\dots,f_n$ in $M$ is \emph{$M$-concave} if $f_if_j\ge 
f_{i-1}f_{j+1}$ for all $0<i\le j<n$. 

\begin{example}
A sequence $b_0,\dots,b_n$ in $\R_{\ge 0}$ is $\R_{\ge 0}$-concave if and
only if it is log-concave and has no internal zeroes, by Lemma
\ref{l.logcon}.
\end{example}

\begin{example}
We equip the multiplicative monoid $\cP_{e_1,\dots,e_r}$ with the partial
order induced by Schur positivity: $f\le g$ if and only if $g-f\in
\cP_{e_1,\dots,e_r}$. We say that a sequence $f_0,\dots,f_n$ in
$\cP_{e_1,\dots,e_r}$ is \emph{Schur log-concave} if it is
$\cP_{e_1,\dots,e_r}$-concave.
\end{example}

The following lemma is obvious.

\begin{lemma}\label{l.concave}
Let $f_0,\dots,f_n$ and $g_0,\dots,g_n$ be $M$-concave sequences. Then 
$f_n,\dots,f_0$ and $f_0g_0,\dots,f_ng_n$ are $M$-concave. 
\end{lemma}

\begin{example}
Let $\lambda^1,\dots,\lambda^r$ be partitions and let $f=s_{\lambda^1}\dotsm 
s_{\lambda^r}\in \Z[x_1,\dots, x_e]$. Let $u_i=f^{[i]}f^{[n-i]}$. By 
Corollary \ref{c.ifinal} and Lemma \ref{l.concave}, $u_0,\dots,u_n$ is 
$\cP_e$-concave. In other words, for $0<i\le j<n$, 
    $u_iu_j-u_{i-1}u_{j+1}$ is Schur positive. In particular, taking $r=1$ and $i=j$,
    this answers a question of Ross and Toma \cite[Question 10.11]{RT2}.
\end{example}

The following statement was inspired by a result of Ping Li \cite[Proposition 
4.6]{Li}.  For partitions $\mu$ and $\nu$, we write $\mu\ge \nu$ if $\lvert 
\mu\rvert =\lvert \nu\rvert$ and $\sum_{i=1}^j \mu_i\ge \sum_{i=1}^j \nu_i$ 
for all~$j$. 

\begin{cor}\label{c.Li}
Let $\lambda^1,\dots,\lambda^r$ be partitions and let 
\[f(\ux)=s_{\lambda^1}(x_{1,1},\dots,x_{1,e_1})\dotsm
    s_{\lambda^r}(x_{r,1},\dots,x_{r,e_r}). \]
Let $\mu=(\mu_1,\dots,\mu_m)$ and
    $\nu=(\nu_1,\dots,\nu_m)$ be partitions such that $\mu\ge\nu$. Then
    \[f^{[\nu_1]}\dotsm f^{[\nu_m]}- f^{[\mu_1]}\dotsm f^{[\mu_m]}\in \cP^{m(\lvert \lambda^1\rvert+\dots +\lvert
\lambda^r\rvert)-\lvert \mu\rvert}_{e_1,\dots,e_r}.
\]
\end{cor}

In the case $r=1$ and $f=c_e(x_1,\dots,x_e)=x_1\dotsm x_e$, we have
$f^{[i]}=c_{e-i}(x_1,\dots,x_e)$ and we recover \cite[Proposition 4.6]{Li}.

Corollary \ref{c.Li} follows from Corollary \ref{c.prodSchlogcon} and the 
following. 

\begin{lemma}\label{l.Li}
Let $f_0,\dots,f_n$ be an $M$-concave sequence. Let $\mu=(\mu_1,\dots,\mu_m)$ 
and $\nu=(\nu_1,\dots,\nu_m)$ be partitions such that $\mu\ge \nu$ and 
$\mu_1\le n$. Then 
\begin{equation}\label{e.ffinal}
f_{\nu_1}\dotsm f_{\nu_m}\ge f_{\mu_1}\dotsm f_{\mu_m}.
\end{equation}
\end{lemma}

\begin{proof}
We may assume $\mu\gneqq \nu$. By induction, we may assume that there does 
not exist $\lambda$ such that $\mu\gneqq \lambda \gneqq \nu$. In this case, 
by a result of Brylawski \cite[Proposition 2.3]{Brylawski}, there exist $i<j$ 
such that $\mu_i=\nu_i+1$ and $\mu_j=\nu_j-1$ and $\mu_k=\nu_k$ for $k\neq 
i,j$. Then, by $M$-concavity, $f_{\nu_i}f_{\nu_j}\ge f_{\mu_i}f_{\mu_j}$ and 
\eqref{e.ffinal} follows. 
\end{proof}

\appendix
\section{The Dowling--Zhao conjecture\\by Shouda 
Wang\textsuperscript{2}}\label{s.app} \footnotetext[2]{KTH Royal Institute of 
Technology, SE-100 44 Stockholm, Sweden.} 

In this appendix, we explain how to prove the Dowling--Zhao conjecture using 
Corollary \ref{c.idLor}. Recall that a matroid $M$ is given by a groundset  
$[n]=\{1,\dots,n\}$  and  a nonempty family $\cI(M)$ of subsets of $[n]$ with 
the following properties: 
\begin{itemize}
    \item[] a). If $A\subset B$ and $B\in \mathcal{I}(M)$, then $A\in 
        \mathcal{I}(M)$. 
    \item[] b). If $A,B\in\cI(M)$ and $\lvert A\rvert<\lvert B\rvert$, then 
        there exists $b\in B\setminus A$, such that $A\cup\{b\}\in\cI(M)$. 
\end{itemize}
The subsets of $[n]$ contained  in $\mathcal{I}(M)$ are called independent 
sets. We refer to \cite{Oxley} for general background in matroid theory.

Log-concavity has been a central theme in recent developments in matroid 
theory. A particularly intriguing sequence of conjectures made by Mason 
states that if $I_k(M)$ denotes the number of size $k$ independent sets of 
$M$, then  for $1\le k \le n-1$, one has 
\begin{align}
I_k(M)^2
&\ge I_{k-1}(M)I_{k+1}(M),\label{e.MasonWeak}\\
\bigl(k!I_k(M)\bigr)^2
&\ge (k-1)!I_{k-1}(M)\,(k+1)!I_{k+1}(M),\label{e.MasonStrong}\\
\left(\frac{I_k(M)}{\binom{n}{k}}\right)^2
&\ge \frac{I_{k-1}(M)}{\binom{n}{k-1}}\frac{I_{k+1}(M)}{\binom{n}{k+1}}.\label{e.MasonUltra}
\end{align}
It can be directly checked that \eqref{e.MasonUltra} implies 
\eqref{e.MasonStrong}  and \eqref{e.MasonStrong} implies \eqref{e.MasonWeak}. 
The strongest inequality \eqref{e.MasonUltra} was proved independently in 
\cite{ALOV} and \cite{BH}. 

A natural strengthening of this result is whether \eqref{e.MasonUltra} holds 
coefficientwise. We use the standard multiindex notation:  for 
$x=(x_1,\dots,x_n)$ and $A\subseteq [n]$, write $x^A=\prod_{i\in A}x_i$. Then 
for 
\[I_k(x)=\sum_{\substack{A\in\cI(M)\\ |A|=k}}x^A,
\]
the coefficientwise version of \eqref{e.MasonUltra}  says that 
\[\left(\frac{I_k(x)}{\binom{n}{k}}\right)^2
-\frac{I_{k-1}(x)}{\binom{n}{k-1}}\frac{I_{k+1}(x)}{\binom{n}{k+1}}\]
is monomial-positive. Note in particular, $I_k(M) = I_k(1,\dots, 1)$. 
However,  this strengthening fails already for the uniform matroid $U_{2,2}$ 
on $\{1,2\}$ as pointed out by Igor Pak. Indeed, $U_{2,2}$ is the matroid on 
$[2]$ where every subset is independent. So, 
\[\left(\frac{I_1(x)}{2}\right)^2-I_0(x)I_2(x)
=\frac{(x_1+x_2)^2}{4}-x_1x_2 =\frac{(x_1-x_2)^2}{4},\] which is a 
nonnegative polynomial but not monomial-positive. One may still ponder 
whether the coefficientwise versions of  \eqref{e.MasonWeak} and 
\eqref{e.MasonStrong} are true, and these are precisely a conjecture of 
Thomas Dowling and a conjecture of Cui Kui Zhao, respectively.

\begin{conj}\label{c.DowlingZhao}
Let $M$ be a matroid on $[n]$ and let $1\le k\leq n-1$. 
\begin{enumerate}[(1)]
\item  (\cite{Dowling}) The polynomial 
\[I_k(x)^2-I_{k-1}(x)I_{k+1}(x)\]
is monomial-positive. 
\item (\cite{Zhao}) The polynomial 
\[I_k(x)^2-\left(1+\frac{1}{k}\right)I_{k-1}(x)I_{k+1}(x)\]
is monomial-positive. 
\end{enumerate}
\end{conj}

Both conjectures were  proved independently in \cite{AMCDEHW} and 
\cite{CCLW}. The author of the appendix later observed that the  machinery of 
the present paper, combined with standard facts from Lorentzian polynomials,  
gives an immediate proof of both conjectures. Furthermore, it is possible to 
obtain the sharp constant in the log-concavity inequalities. 

\begin{prop}
\label{prop:a} 
    The polynomial
\begin{equation}
    I_k(x)^2-\left(1+\frac{1}{\min(k,n-k)}\right)I_{k-1}(x)I_{k+1}(x)
    \label{eqa2}
\end{equation}
is monomial-positive. For all fixed $n\geq 1$ and $1\leq k\leq n-1$, the 
constant $1+\frac{1}{\min(k,n-k)}$ is sharp. 
\end{prop}

We recall the following  well-known fact. 

\begin{lemma}
For any matroid $M$ on $[n]$, the polynomials 
$$
f_1(x,y)
= \sum_{I\in\cI(M)} y^{|I|}x^{[n]\setminus I},\qquad
f_2(x,y)
= \sum_{I\in\cI(M)} y^{n-|I|}x^{ I},
$$
are Lorentzian. 
\end{lemma}

\begin{proof}
The homogeneous multivariate Tutte polynomial 
$$
Z_{q,N}(x,y)
= \sum_{A\subseteq[n]} q^{-r_N(A)}y^{n-|A|} x^A
$$
is a Lorentzian polynomial for all $q\in(0,1)$ (see \cite{BH}*{Theorem 
4.10}). Here $N$ is a matroid on $[n]$ and $r_N$ is the rank function of $N$. 
Now 
$$
\lim_{q\to 0} q^{r_N([n])}Z_{q,N}(x,y) =\lim_{q\to 0}  \sum_{A\subseteq[n]} q^{r_N([n])-r_N(A)}y^{n-|A|} x^A
=
\sum_{A: r_N(A)=r_N([n])} y^{n-|A|} x^A
$$
is also Lorentzian because Lorentzian polynomials are closed under scaling 
and limits \cite[Section 2]{BH}. Note that $r_N(A)=r_N([n]) $ means $A$ is 
spanning in $N$. Let  $N:=M^*$, the dual matroid of $M$. Then $A$ being 
spanning in $N$ is equivalent to  $[n]\setminus A$ being independent in $M$. 
Therefore, 
$$
\lim_{q\to 0} q^{r_N([n])}Z_{q,N}(x,y)  = f_1(x,y),
$$
and $f_1$ is Lorentzian as claimed. 

For $f_2$, we have 
$$
 \lim_{q\to 0} Z_{q,M}(qx,y)
  = \lim_{q\to 0}\sum_{A\subseteq[n]} q^{|A|-r_M(A)} y^{n-|A|}x^A = \sum_{I\in \cI(M)} y^{n-|I|}x^I = f_2(x,y).
$$
Therefore,  since Lorentzian polynomials are closed under coordinate scaling 
and taking limits \cite[Section 2]{BH}, $f_2$ is Lorentzian as well. 
\end{proof}

Recall that $N$ is the normalization operator $N(z^\alpha) = 
z^\alpha/\alpha!$ as before. Let 
$$
g_1(x,y) = g_1(x_1,\dots,x_n, y)=\sum_{I\in \mathcal{I}(M)} |I|! y^{n-|I|} x^I.
$$
Then 
$$
N(g_1(x_1^{-1},\dots, x_n^{-1}, y^{-1} )x_1\cdots x_n y^n) = f_1(x,y)
$$
and therefore $g_1$ is a dually Lorentzian polynomial. Let 
$$
g_2(x,y)  =\sum_{I\in \mathcal{I}(M)} (n-|I|)! y^{|I|} x^{[n]\setminus I}.
$$
Then 
$$
N(g_2(x_1^{-1},\dots, x_n^{-1}, y^{-1} )x_1\cdots x_n y^n) = f_2(x,y)
$$
and therefore $g_2$ is also a dually Lorentzian polynomial.

\begin{proof}[Proof of Proposition \ref{prop:a}]
First, we have 
$$
\partial_y^{n-k} g_1|_{y=0} =k!(n-k)!\sum_{|I|=k} x^I.
$$
By Corollary \ref{c.idLor}, 
$$
\left(\frac{\partial^{n-k}_y g_1|_{y=0}}{(n-k)!}\right)^2-
\frac{\partial^{n-k+1}_y g_1|_{y=0}}{(n-k+1)!}\frac{\partial^{n-k-1}_y g_1|_{y=0}}{(n-k-1)!} = k!^2\left(I_k(x)^2-\left(1+\frac{1}{k}\right)I_{k-1}(x)I_{k+1}(x)\right)
$$
is monomial-positive. This already proves Zhao's conjecture and, \textit{a 
fortiori}, Dowling's conjecture. 

Similarly, for $g_2$ we have 
$$
\partial_y^{k} g_2|_{y=0} = k!(n-k)!\sum_{|I| = k} x^{[n]\setminus I}.
$$
By Corollary \ref{c.idLor}, 
\begin{align*}
    \left(\frac{\partial^{k}_y g_2|_{y=0}}{k!}\right)^2 &-
\frac{\partial^{k+1}_y g_2|_{y=0}}{(k+1)!}\frac{\partial^{k-1}_y g_2|_{y=0}}{(k-1)!} \\
&= (n-k)!^2 
\left(\prod_{i\in [n]}x_i\right)^2
\left(I_k(x^{-1})^2-\left(1+\frac{1}{n-k}\right)I_{k-1}(x^{-1})I_{k+1}(x^{-1})\right)
\end{align*}
is monomial-positive. This shows that 
$$ 
I_k(x)^2-\left(1+\frac{1}{n-k}\right)I_{k-1}(x)I_{k+1}(x)
$$
is monomial-positive. We have thus proved \eqref{eqa2}. 

For the sharpness, take $M=U_{n,n}$. In the case $k\leq n-k$,  take any 
$C\subseteq [n]$ with $|C| = 2k$. The coefficient of $x^C$ in $I_k(x)^2$ is  
$2k\choose k$ while  the coefficient of $x^C$ in $I_{k+1}(x)I_{k-1}(x)$ is  
$2k\choose k-1$. Hence 
\begin{align*} 
I_k(x)^2-cI_{k-1}(x)I_{k+1}(x)
\end{align*}
is monomial-positive only if 
$$
c\leq 1+\frac{1}{k} = 1+\frac{1}{\min(k,n-k)}
$$
in this first case.

For the case $k\geq n-k$, take two disjoint subsets $C,D$ of $[n]$ with $|C| 
=2( n-k)$ and $|D| = 2k-n$. Then the coefficient of $x^C(x^D)^2$ in 
$I_k(x)^2$ is  $2(n-k)\choose n-k$ while  the coefficient of $x^C(x^D)^2$ in 
$I_{k+1}(x)I_{k-1}(x)$ is  $2(n-k)\choose n-k-1$. Hence 
\begin{align*} 
I_k(x)^2-cI_{k-1}(x)I_{k+1}(x)
\end{align*}
is monomial-positive only if 
$$
c\leq 1+\frac{1}{n-k} = 1+\frac{1}{\min(k,n-k)}
$$
in this second case as well.
\end{proof}

\end{document}